\documentclass[11pt]{amsart}
\usepackage[a4paper,text={14.5cm,23cm},centering]{geometry}
\usepackage[utf8]{inputenc}
\usepackage[T1]{fontenc}
\usepackage[svgnames,hyperref]{xcolor}
\usepackage{lmodern}
\usepackage{upgreek} 
\usepackage{amssymb} 
\usepackage{placeins} 
\usepackage{mathabx} 
\usepackage{xifthen} 
\usepackage{booktabs} 
\usepackage{stmaryrd} 
\usepackage{amsmath}

\usepackage{caption} 
\usepackage[shortlabels,inline]{enumitem}
\setlist[2]{leftmargin=5mm}
\setlist[enumerate]{label=\rm{(\arabic*)}}
\setlist[enumerate,2]{label=\rm({\it\roman*}), }
\setlist[itemize]{label=\raisebox{0.25ex}{\tiny$\bullet$}}

\usepackage{booktabs, tabularx, multirow}
\newcolumntype{L}{>{$}l<{$}}
\newcolumntype{C}{>{$}c<{$}}
\usepackage{longtable}

\usepackage{tikz}
\usetikzlibrary{cd,arrows,positioning,calc}
\usetikzlibrary{shapes.geometric}
\tikzset{map/.style={row sep=0em, column sep=0em}}
\usetikzlibrary{intersections, through}

\newcommand{\nicecolor}{Navy}
\usepackage[backref=page, pdfauthor={J. Schneider}, colorlinks, citecolor=\nicecolor, linkcolor=\nicecolor, urlcolor=\nicecolor]{hyperref}
\usepackage[all]{hypcap} 

\theoremstyle{plain}
\newtheorem{theorem}{Theorem}[section]

\newtheorem{proposition}[theorem]{Proposition}
\newtheorem{lemma}[theorem]{Lemma}

\theoremstyle{definition}
\newtheorem{definition}[theorem]{Definition}

\newtheorem{example}[theorem]{Example}
\newtheorem{remark}[theorem]{Remark}
\newtheorem{algorithm}[theorem]{Algorithm}
\newtheorem{notation}[theorem]{Notation}

\numberwithin{equation}{theorem}  

\newcommand{\p}{\mathbb{P}}
\let \P \relax
\newcommand{\P}{\mathbb{P}}
\newcommand{\A}{\mathbb{A}}
\newcommand{\F}{\mathbb{F}}
\newcommand{\Z}{\mathbf Z}

\newcommand{\Q}{\mathbf Q}

\renewcommand{\k}{\mathbf{k}}

\newcommand{\bk}{{\bar \k}}
\newcommand{\id}{\textup{id}}

\newcommand{\Cl}{\mathcal C}
\newcommand{\Dl}{\mathcal D}

\newcommand{\JJ}{\mathcal J}
\newcommand{\XX}{\mathcal X}

\newcommand{\dashto}{\dashrightarrow}

\newcommand{\smallmat}[1]{\left(\begin{smallmatrix}#1\end{smallmatrix}\right)}

\newcommand{\link}[2]{\stackrel{_{#1}\;_{#2}}{\dashrightarrow}}
\newcommand{\linkI}[1]{\stackrel{_{#1}}{\dashrightarrow}}
\newcommand{\linkIII}[1]{\stackrel{_{#1}}{\to}}

\renewcommand{\phi}{\varphi}

\renewcommand{\leq}{\leqslant}

\renewcommand{\geq}{\geqslant}

\DeclareMathOperator{\Bir}{Bir}

\DeclareMathOperator{\Aut}{Aut}

\DeclareMathOperator{\Spec}{Spec}
\DeclareMathOperator{\Pic}{Pic}

\DeclareMathOperator{\PGL}{PGL}

\DeclareMathOperator{\GL}{GL}
\DeclareMathOperator{\Gal}{Gal}
\DeclareMathOperator{\Bs}{Bs}
\DeclareMathOperator{\Bas}{Bas}
\DeclareMathOperator{\Exc}{Exc}
\DeclareMathOperator{\Mat}{Mat}

\DeclareMathOperator{\lcm}{lcm}

\renewcommand{\to}{ \, \tikz[baseline=-.6ex] \draw[->,line width=.5] (0,0) -- +(.5,0); \, }
\renewcommand{\mapsto}{ \, \tikz[baseline=-.6ex] \draw[|->,line width=.5] (0,0) -- +(.5,0); \, }

\newcolumntype{M}{>{$}c<{$}}

\makeatletter
\newcommand{\mylabel}[2]{#2\def\@currentlabel{#2}\label{#1}}
\makeatother

\DeclareMathOperator{\Sym}{Sym}
\DeclareMathOperator{\Frob}{Frob}
\newcommand{\ZZZ}{\Z}
\newcommand{\AAA}{\A}
\newcommand{\kk}{\k}
\newcommand{\FFF}{\F}
\newcommand{\PPP}{\p}
\newcommand{\PPPP}{\mathcal P}
\newcommand{\QQQQ}{\mathcal Q}

\newcommand{\comp}{\circ}
\newcommand{\from}{\colon}
\newcommand{\simto}{\stackrel{\sim}{\rightarrow}}

\title[Cremona group over $\F_2$]{Generators of the plane Cremona group over the field with two elements}
\author[Julia Schneider]{Julia Schneider}
\address{Julia Schneider, Universit\"{a}t Basel, Departement Mathematik und Informatik, Spiegelgasse $1$, CH-$4051$ Basel, Switzerland}
\curraddr{Institut für Mathematik, Universit\"at Z\"urich, Winterthurerstrasse 190, CH-8057 Z\"urich, Switzerland}
\email{julia.schneider@math.ch}

\begin{document}

\maketitle

\tableofcontents

\section{Introduction}\label{sec:Introduction}
The Cremona group of rank $n$ over a field $\kk$ is the group $\Bir_\kk(\PPP^n)$ of birational maps of the projective space $\PPP^n$ that are defined over $\kk$. Equivalently, it is the group of field automorphisms $\Aut_{\kk}\kk(x_1,\ldots,x_n)$.
We are interested in the following question:
Is there a ``nice'' set of generators of the plane Cremona group over a fixed field?
The classical Theorem of Noether and Castelnuovo states that over an algebraically closed field, the Cremona group of rank $2$ is generated by $\PGL_3(\kk)$ and the standard Cremona transformation given by $[x:y:z]\mapsto[yz:xz:xy]$ \cite{Noether1870, Castelnuovo1901}.
Over other fields, this question was already studied by Kantor, who gives in 1899 a list of sixteen types of birational maps that generate the Cremona group over the rational numbers $\Q$ \cite{Kantor}.
A (long) list of generators can be found in \cite{Iskovskikh92,IKT93} for any perfect field.
For certain fields such as the field of real numbers, generating sets have been described more explicitly \cite{BM14,RV05,zimmermann18}.
In general, generating sets over non-closed fields are more complicated to describe.

For a perfect field $\kk$, we fix an algebraic closure $\overline\kk$ and equip $\PPP^2$, or more generally any variety $X$ over $\k$, with the action of the absolute Galois group $\Gal(\overline\kk/\kk)=\Aut_\kk(\overline\kk/\kk)$. In this way, a closed point of degree $d$ in $X$ corresponds to a $\Gal(\overline\kk/\kk)$-orbit of $d$ points in $X(\overline\kk)$.
We will mostly be interested in the case when $\kk$ is a finite field. In this case, for every degree there exists a (unique) field extension and its Galois group is cyclic.

\bigskip
In a nutshell, this paper consists of three main parts: In Sections \ref{sec:GaloisDescent} and \ref{sec:PointsGeneralPosition}, we are interested in minimal rational del Pezzo surfaces over a finite field, and {Sarkisov links} between them (see Section~\ref{sec:rankRfibrations} for a definition). Section \ref{section:InfiniteFamilies} deals with minimal rational conic bundles and birational maps between them; here, we are mostly interested in finite fields of characteristic $2$.
Finally, in Section \ref{sec:BirationalMaps} we combine the two parts to describe a generating set of $\Bir_{\F_2}(\p^2)$. This is all preceeded by preliminaries in Section~\ref{sec:preliminaries}.
We present now our results.

\bigskip

Write $\Dl$ for the set of rational del Pezzo surfaces of Picard rank $1$, up to isomorphism, and $\Dl_d$ for those of degree $d$.
Over a perfect field it holds that $\Dl=\Dl_9\cup\Dl_8\cup\Dl_6\cup\Dl_5$.
Over a finite field each of them consists of exactly one element, and they can be described using a birational Galois descent:

\begin{proposition}\label{prop:UniqueMinimaldP}
	Let $\k=\F_{q}$ be a finite field, and $d\in\{9,8,6,5\}$. Then $\Dl_d$ contains exactly one element $X_d$. Moreover, for each $d$ consider the following field extension $L/\k$, surface $Y$, and birational map $\varphi\in\Bir(\p^2_L)$:
	\begin{enumerate}
		\item If $d=9$, let $Y=\p^2$, $L=\k$, and $\varphi=\id_{\p^2}$;
		\item If $d=8$, let $Y=\p^1\times\p^1$, $L=\F_{q^2}$, and $\varphi\colon (x,y)\mapsto (y,x)$;
		\item If $d=6$, let $Y=\p^2$, $L=\F_{q^6}$, and $\varphi\colon [x:y:z]\mapsto [xz:xy:yz]$;
		\item If $d=5$, let $Y=\p^2$, $L=\F_{q^5}$, and $\varphi\colon[x:y:z]\mapsto [xy:y(x-z):x(y-z)]$.
	\end{enumerate}
	Then there exists a birational morphism $\rho\colon (X_d)_L\to Y_L$ such that \[\rho\circ g\circ\rho^{-1}=\langle g\circ\varphi \rangle,\]
	where $g$ is a generator of $\Gal(L/\k)$.
\end{proposition}

We say that a point is in \emph{(del Pezzo-) general position} if its blow-up $Z\to X_d$ is again a del Pezzo surface. In this case, $Z$ is of Picard rank $2$, and gives rise to a Sarkisov link $\chi$. Writing $a$ for the degree of the blown up point, we use the notation $\chi\colon X_d\linkI{a} Z$ if $Z/\p^1$ is a conic bundle, and we write $\chi\colon X_d\link ab X_{e}$ if $Z$ admits a contraction onto a point of degree $b$ on a del Pezzo surface $X_e$ of degree $e=d-a+b$.

We say that two Sarkisov links are \emph{equivalent} if they are equal after pre- and postcomposition with an isomorphism.
In particular, there is a one-to-one correspondence between del Pezzo surfaces of Picard rank $2$, up to isomorphism, and Sarkisov links over $\Spec(\k)$, up to equivalence and taking the inverse. We say that a Sarkisov link is \emph{symmetric} if its source and target are equal.
It is interesting to know whether a Sarkisov link is equivalent to its inverse.

We also describe $\rho\circ\Aut_\k(X_d)\circ\rho^{-1}$ in each of the cases of Proposition~\ref{prop:UniqueMinimaldP}. This gives us an algorithm to count points of a fixed degree on $X_d$, up to $\Aut_\k(X_d)$.
Therefore, by counting points in (del Pezzo-) general position on each $X_d$, one counts the number of Sarkisov links starting at $X_d$.
For $\k=\F_2$, we carry out the computations of points in general position, up to the action of $\Aut_\k(X_d)$, using the computer algebra system sagemath \cite{Sagemath}, and obtain:

\begin{theorem}[Counting Sarkisov links over $\k=\F_2$]\label{thm:CountingSarkisovLinks}
	Let $\k=\F_2$. Let $Z$ be a rational del Pezzo surface with $\rho(Z)=2$. Then either $Z=\p^1\times\p^1$, or $Z$ is a minimal resolution of one of the following Sarkisov links, ordered by $K_Z^2\in\{1,\ldots,8\}$ and the bold number denoting the number of such links up to equivalence,
	where $X_d\in\Dl_d$:
	\[
	\begin{array}{c|rl@{\hskip 2em}rl@{\hskip 2em}rl@{\hskip 2em}rl}\toprule
		1 & \mathbf{38} & \p^2\link88\p^2 & \mathbf{18} & X_8\link77X_8 & \mathbf{11} & X_6\link55 X_6 & \mathbf{12} & X_5\link44 X_5 \\
		2 & \mathbf{10} & \p^2\link77\p^2 & \mathbf{5} & X_8\link66X_8 & \mathbf{4} & X_6\link44X_6 & \mathbf{4} & X_5\link33X_5 \\
		3 & \mathbf{2} & \p^2\link66\p^2 & \mathbf{2} & X_8\link52X_5 & \mathbf{2} & X_6\link33X_6 & \mathbf{} & (X_5\link25 X_8) \\
		4 & \mathbf{1} & \p^2\link51 X_5 & \mathbf{0} & X_8\link44 X_8 & \mathbf{1} & X_6\link22X_6 & \mathbf{} & (X_5\link15 \p^2) \\
		5 & \mathbf{1} &  \p^2\linkI4 Z/\p^1 & \mathbf{1} & X_8\link31X_6 & \mathbf{} & (X_6\link13X_8) & \mathbf{} & ~ \\
		6 & \mathbf{1} & \p^2\link33\p^2 & \mathbf{1} & X_8\linkI2Z/\p^1 & \mathbf{} & ~ & \mathbf{} & ~ \\
		7 & \mathbf{1} & \p^2\link21X_8 & \mathbf{} & (X_8\link12\p^2) & \mathbf{} & ~ & \mathbf{} & ~ \\
		8 & \mathbf{1} & \p^2\linkI{1}\F_1/\p^1 & \mathbf{} &~ & \mathbf{} & ~ & \mathbf{} & ~ \\\bottomrule
	\end{array}
	\]
	(A link $X\link ab X'$ in brackets means that its inverse already appears in the same row.)

	Moreover, every symmetric Sarkisov link in this list is (equivalent to) an involution, and hence the bold numbers denote the number of isomorphism classes of del Pezzo surfaces of Picard rank $2$.
\end{theorem}

It is known for any finite field, that in each case of $K_Z^2\geq 5$ there is a unique del Pezzo surface $Z$ of rank $2$. Also, it is known that in the case of $K_Z^2\in\{1,2\}$ all links are equivalent to involutions (Bertini respectively Geiser involution). For $\k=\F_2$, it comes as a surprise that all the symmetric Sarkisov links with $K_Z^2\in\{3,4\}$ are involutions, too.
It is also worth to mention the absence of a link $X_8\link44 X_8$ for $\k=\F_2$.
Already for $\k=\F_3$, both of these peculiarities disappear.

\begin{remark}\label{rem:CountingSarkisovLinksOverF3}
	Let $\k=\F_3$. Let $Z$ be a rational del Pezzo surface with $\rho(Z)=2$. Then either $Z=\p^1\times\p^1$, or $Z$ is a minimal resolution of one of the following links, where the bold number $c=a+\frac{b}{2}$ denotes the number of such del Pezzo surfaces up to isomorphism, where $a$ is the number of Sarkisov links that are involutions, and $b$ the number of links that are not equivalent to their inverse:
	\[
	\begin{array}{c|rl@{\hskip 2em}rl@{\hskip 2em}rl@{\hskip 2em}rl}\toprule
		1 & \textbf{?} & \p^2\link88\p^2 & \textbf{?} & X_8\link77X_8 & \mathbf{280} & X_6\link55 X_6 & \mathbf{324} & X_5\link44 X_5 \\
		2 & \textbf{?} & \p^2\link77\p^2 & \mathbf{63} & X_8\link66X_8 & \mathbf{42} & X_6\link44X_6 & \mathbf{48} & X_5\link33X_5 \\
		3 & \mathbf{8=5+\frac{6}{2}} & \p^2\link66\p^2 & \mathbf{8} & X_8\link52X_5 & \mathbf{6=5+\frac{2}{2}} & X_6\link33X_6 & \mathbf{} & (X_5\link25 X_8) \\
		4 & \mathbf{2} & \p^2\link51 X_5 & \mathbf{1} & X_8\link44 X_8 & \mathbf{2} & X_6\link22X_6 & \mathbf{} & (X_5\link15 \p^2) \\\bottomrule
	\end{array}
	\]
	(A link $X\link ab X'$ in brackets means that its inverse already appears in the same row. A question mark means that the author's computer was not able to finish the computation with sagemath \cite{Sagemath}; one would need to optimize the algorithm or use a more powerful computer.)
\end{remark}

It is worth mentioning that Lamy and Zimmermann constructed a surjective group homomorphism $\Bir_\kk(\PPP^2)\to\Asterisk_{\mathcal{B}} \ZZZ/2\ZZZ$ for many perfect fields,
where $\mathcal{B}$ consists of all Galois orbits of size $8$, up to automorphisms of the plane, respectively the associated Bertini involutions \cite{LZ19}. In other words, $\mathcal B$ corresponds to rational del Pezzo surfaces of degree $1$ and of Picard rank $2$, admitting a birational morphism to $\p^2$. (In fact, omitting the condition of having a birational morphism to $\p^2$ one obtains a larger quotient $\Bir_\kk(\PPP^2)\to\Asterisk_{\mathcal{B}'} \ZZZ/2\ZZZ\to\Asterisk_{\mathcal{B}} \ZZZ/2\ZZZ$.)
In an involved counting argument of points in general position on nodal cubics, they show that the size of $\mathcal{B}$ is at least the cardinality of the field (see \cite[Proposition 4.17]{LZ19} for finite fields).
In the case of $\FFF_2$ this gives two Bertini involutions but we find in Theorem~\ref{thm:CountingSarkisovLinks} that there are exactly $38$, up to a change of coordinates, i.e. $|\mathcal{B}|=38$. (Moreover, we find that $|\mathcal{B}'|=79$.)
Similarly, Zimmermann constructed a group homomorphism $\Bir_{\k}(\p^2)\to \Asterisk_{\mathcal{G}}\Z/2\Z$ using Geiser involutions on rational del Pezzo surfaces of degree $2$ and of Picard rank $2$ \cite{Zimmermann_RemarkGeiserInvolutions}, carefully studying by which del Pezzo surfaces of Picard rank $3$ they are dominated. Our results show that in the case of $\k=\F_2$ the cardinality of $\mathcal{G}$ equals $23$.
As a side note, for $\k=\F_3$ we can only give lower bounds due to the increased computational complexity: In this case, we have $|\mathcal{B}'|\geq 604$ and $|\mathcal{G}|\geq 153$ (see Remark~\ref{rem:CountingSarkisovLinksOverF3}).

One could continue counting the number of rational del Pezzo surfaces of Picard rank $\geq 3$ over $\k=\F_2$. We will not do this. Let us just remark that the two del Pezzo surfaces $Z_1,Z_2$ of Picard rank $2$ dominating the Sarkisov links $\chi_1,\chi_2\colon X_6\link33 X_6$ from Theorem~\ref{thm:CountingSarkisovLinks} do not have any point in (del Pezzo-) general position \cite[Remark~4.16]{LamySchneider}, and so are not dominated by a del Pezzo surface of Picard rank $3$. In fact, this means that these links are not part of any nontrivial relation, giving a new construction of a surjective group homomorphism $\Bir_{\F_2}(\p^2)\to\Z/2\Z\ast\Z/2\Z$.

\bigskip

Given a rational map $\pi\colon X\dashrightarrow\p^1$, we say that a birational map $\varphi\colon X\dashrightarrow X$ \emph{preserves the fibration} if there exists $\alpha\in\Aut_\k(\p^1)$ such that $\pi\circ\varphi=\alpha\circ\pi$, and we say it \emph{fixes the fibration} if $\alpha=\id_{\p^1}$. Writing $\Bir(X,\pi)$ for the birational maps preserving the fibration, this induces an exact sequence \[1\to \Bir(X/\pi)\to\Bir(X,\pi)\to \Aut_\k^{\pi}(\p^1)\to1.\]

Over a field $\k$, fix the following fibrations $\pi_i\colon\p^2\dashrightarrow\p^1$, $\pi_i([x:y:z])=[F_0:F_1]$ for $i\in\{1,2,4\}$:
\begin{enumerate}
  \item[\mylabel{it:F1}{(Fib1)}] If $i=1$ set $[F_0:F_1] = [y:z]$;
  \item[\mylabel{it:F2}{(Fib2)}] If $i=2$ set $[F_0:F_1] = [xy+y^2:x^2+xz+z^2]$;
  \item[\mylabel{it:F4}{(Fib4)}] If $i=4$ set $[F_0:F_1] = [xz+y^2:x^2+xy+z^2]$.
\end{enumerate}

While it is classical that $\JJ_1:=\Bir_\k(\p^2,\pi_1)$ has a semi-direct product structure $\PGL_2(\k(t))\rtimes\PGL_2(\k)$ for any field $\k$, in this text we are primarily interested in conic fibrations such as $\pi_4$ and $\pi_2$ in the case when $\k$ is a field of characteristic $2$.

If $\k=\F_{2^N}$ with $N$ odd, then the indeterminacy locus $\PPPP_i$ of $\pi_i$ is a point of degree $4$ for $i=4$, (respectively two points of degree $2$ for $i=2$); this is due to the fact that in this case $x^4+x+1$ (respectively $x^2+x+1$) is irreducible over $\k$. In both cases, no three of the four geometric components of $\PPPP_i$ are collinear.
In characteristic $2$, conics have a strange behaviour. For example, for $i=2,4$, every conic in the pencil given by $\pi_i$ is tangent to the line given by $x=0$; this is strange, and we call $x=0$ the \emph{strange line} with respect to $\pi_i$ \cite[Chapter IV, Example 3.8.2]{Hartshorne77}.

\begin{example}\label{ex:ExplicitMaps}
  Let $\k=\F_{2^N}$ with $N$ odd. For $i\in\{2,4\}$ write $t={F_1}/{F_0}$ where $F_0,F_1$ are as in \ref{it:F2}, \ref{it:F4}.
  For $[u:v]\in\p^1(\k(t))$,
  \begin{enumerate}
    \item if $i=4$ set $a(u,v)=\frac{u+tv}{tu^2+v^2}$,
    \item if $i=2$ set $a(u,v)=\frac{tu+v}{tu^2+v^2}$.
  \end{enumerate}
   Then,
  \[\chi_{i,[u:v]}\colon [x:y:z] \mapsto [x:ua(u,v)x+y:va(u,v)x+z]\]
  is an involution in $\Bir_{\k}(\p^2)$, it fixes the fibration $\pi_i$, and it fixes the strange line $x=0$.
\end{example}

More concretely, $\chi_{4,[1:0]}$ is the cubic map $[x:y:z]\mapsto[tx:sx+ty:tz]$, where $[s:t]=[y^2+xz:x^2+xy+z^2]$, and $\chi_{2,[1:0]}$ is the linear map $[x:y:z]\mapsto[x:x+y:z]$.

\begin{proposition}\label{prop:ParametrizationFiberingType}
  Let $\k=\F_{2^N}$ with $N$ odd.
  Let $\PPPP$ be a point of degree $4$, or two points of degree $2$, such that no three of the components are collinear.
  Fix a fibration $\pi\colon\p^2\dashrightarrow\p^1$ given by the pencil of conics through $\PPPP$.
  Let $\varphi\in\Bir_\k(\p^2,\pi)$ be a birational map preserving the fibration $\pi$.
  Then the  following hold:
  \begin{enumerate}
    \item\label{it:ParametrizationFiberingType--fix} If $\varphi$ \emph{fixes} the fibration $\pi$ and fixes the \emph{strange line}, then $\varphi$ is of the form as in Example~\ref{ex:ExplicitMaps}, up to conjugation by an element of $\Aut_\k(\p^2)$.
    \item\label{it:ParametrizationFiberingType--F2} If $\k=\F_2$, then $\varphi$ can be written as a composition of elements $\chi_{i,[u:v]}$ as in Example~\ref{ex:ExplicitMaps}, elements in $\Aut_\k(\p^2)$ and $\JJ_1$.
    In particular, $\varphi$ admits a decomposition into involutions in $\Bir_\k(\p^2)$.
  \end{enumerate}
\end{proposition}

\bigskip

To describe a set of generators of $\Bir_\k(\p^2)$, the starting point is the following set of generators that was obtained in \cite{Schneider-relations} using Iskovskikh's decomposition of birational maps into Sarkisov links \cite{Iskovskikh96}:

\begin{proposition}[{\cite[Corollary 6.14]{Schneider-relations}}]\label{proposition:GeneratorsAnyPerfectField}
  Let $\kk$ be any perfect field. The Cremona group $\Bir_\kk(\PPP^2)$ is generated by its subsets
  \begin{enumerate}
    \item $J_1$, the group of birational maps preserving a pencil of lines (i.e. de Jonqui\`eres maps),
    \item $J=\bigcup J_\PPPP$, the union of groups $J_\PPPP$ of birational maps that preserve a pencil of conics through four points $\PPPP$, of which no three are collinear, such that $\PPPP$ forms either one Galois orbit of size $4$, or two Galois orbits of size $2$,
    \item $G_{\leq 8}$, the group of birational maps such that all Galois orbits in its base locus are of size $\leq 8$.
  \end{enumerate}
\end{proposition}

We obtain:

\begin{theorem}\label{thm:F2GeneratedByInvolutions}
  Let $\k=\F_2$.
  Then $\Bir_{\F_2}(\p^2)$ is generated by all $\chi_{i,[u:v]}$ as in Example~\ref{ex:ExplicitMaps}, $\Aut_\k(\p^2)$, $J_1$, and $111$ additional maps.

  Moreover, $\Bir_{\F_2}(\p^2)$ is generated by involutions.
\end{theorem}

We do not claim that this is a minimal generating set. To know this one would have to study all relations between these generators. However, we know that at least $106$ are not redundant ($38+18+11+12=79$ Bertini, and $10+5+4+4=23$ Geiser \cite[Lemma~2.5]{Zimmermann_RemarkGeiserInvolutions}, and two $X_6\link33 X_6$ because they do not appear in any non-trivial relation).
Moreover, this gives a surjective group homomorphism $\Bir_{\F_2}(\p^2)\to \Asterisk_{i=1}^{106}\Z/2\Z$ that is defined by writing $\varphi\in\Bir_{\F_2}(\p^2)$ as a product of the generators in Theorem~\ref{thm:F2GeneratedByInvolutions}, and sending each of the $106$ generators on the corresponding word while sending all others onto the trivial word.

\bigskip

In an earlier version of this text, Theorem~\ref{thm:F2GeneratedByInvolutions} was proved independently. However, in the meantime Lamy and the author have proved that over any perfect field the plane Cremona group is generated by involutions \cite{LamySchneider}. This text has been revised heavily.
Many statements that have previously been stated solely for $\F_2$ have been generalized to a larger class of fields, and now the text refers in some instances to \cite{LamySchneider} to reduce overlap.

In \cite{LamySchneider}, one of the main ideas to decompose a Sarkisov link of del Pezzo type $\chi\colon X\link dd X'$ into involutions is to find a del Pezzo surface of degree $1$ or $2$ that dominates the minimal resolution of $\chi$, and use the associated Bertini or Geiser involution.
It is important to mention that this strategy fails for $\k=\F_2$ and Sarkisov links $X_6\link33 X_6$ because the associated cubic del Pezzo surface of Picard rank $2$ is not dominated by any del Pezzo surface, as already mentioned above. This is why \cite{LamySchneider} refers to this text for the case $\k=\F_2$: As mentioned in Theorem~\ref{thm:CountingSarkisovLinks}, there are exactly two such links $X_6\link33 X_6$, and they are involutions. So \cite{LamySchneider} needs Proposition~\ref{prop:F2CubicQuarticAreInvolution} that does not rely on \cite{LamySchneider}.

In \cite{LamySchneider}, birational maps in $J_{\PPPP}$ that fix every conic in the pencil are decomposed into involutions using the classical theorem of Cartan--Dieudonné. Here we give a proof using birational geometry in characteristic $2$.
\bigskip

Plane Cremona groups over finite fields have also been studied with a geometric group theory approach in \cite{GLU-finite} (giving an embedding into Neretin groups), and their finite subgroups in \cite{ProkhorovShramov-finite} (giving a bound for the Jordan constant).

\section{Preliminaries}\label{sec:preliminaries}

Let $\k$ be a perfect field. All varieties and morphisms will be assumed to be defined over $\k$, unless stated otherwise. We fix an algebraic closure $\overline\kk$ and equip a variety $X$ over $\k$ with the action of the absolute Galois group $\Gal(\overline\kk/\kk)=\Aut_\kk(\overline\kk/\kk)$. In this way, a closed point $p$ on $X$ corresponds to a $\Gal(\overline\kk/\kk)$-orbit of $d$ points in $X(\overline\kk)$.
In other words, $p_{\bar\k}=\{p_1,\ldots,p_d\}\subset X(\bar\k)$ forms one Galois-orbit.
The integer $d$ corresponds to the degree of the extension of the residue field $\k(p)$ over $\k$, and we will say that $p$ is a point \emph{of degree $d$}. We will call the $p_i$ \emph{geometric components} of $p$. We will use both languages.
Given a fibration $\pi\colon X\to\p^1$, by a \emph{fiber} we will typically mean a fiber of $\pi_{\bar\k}\colon X_{\bar\k}\to\p^1_{\bar\k}$ and consider it as a subvariety of $X_{\bar\k}$, or as a geometric component of a subvariety of $X$.

\subsection{Mori fiber spaces, Sarkisov links, and rank $r$ fibrations}\label{sec:rankRfibrations}

The Sarkisov program is a useful tool for studying generators and relations in Cremona groups. In particular, it states that over a perfect field, any birational map $\p^2\dashrightarrow\p^2$ can be written as a sequence of Sarkisov links.

For surfaces over a perfect field, a \emph{rank $r$ fibration} is a surjective morphism $\pi\colon X\to B$ with connected fibers such that $X$ is a smooth surface, $B$ is either a point or a smooth curve such that $\rho(X/B)=r$, and $-K_X$ is $\pi$-ample \cite{LZ19}. When $B=\Spec \k$ is a point, then rank $r$ fibrations are exactly del Pezzo surfaces of Picard rank $r$. A \emph{Mori fiber space} $X/B$ is a rank $1$ fibration.
A rank $2$ fibration $Z/B$ has exactly two contractions giving a commutative diagram \[\begin{tikzcd}
	&Z\ar[dl,"\eta_1",swap]\ar[dr,"\eta_2"]&\\
	X_1\ar[rr, dashrightarrow,"\chi"]\ar[d,"\pi_1",swap] && X_2\ar[d,"\pi_2"]\\
	B_1\ar[dr] && B_2\ar[dl]\\
	&B,&
\end{tikzcd}\]
where $\pi_i\colon X_i\to B_i$ are Mori fiber spaces. Any birational map $\chi\colon X_1\dashrightarrow X_2$ that can be written in such a diagram is called \emph{Sarkisov link}. In this diagram, the morphisms $\eta_i$ are either isomorphisms, or the blow-up of a point of degree $a_i$ in $X_i$. This gives rise to four types of Sarkisov links:
\begin{center}
	\begin{tabular}{r|cc}
		&$\eta_1$ isomorphism & $\eta_1$ blow-up \\\hline
		$\eta_2$ isomorphism &type IV& type I \\
		$\eta_2$ blow-up &type III&type II.
	\end{tabular}
\end{center}
The birational map $\chi\colon X_1\dashrightarrow X_2$ is uniquely determined by $Z/B$, up to pre- and postcomposing with an isomorphism, and up to taking the inverse.

\bigskip

Write $\Dl_d^r$ for the set of isomorphism classes of rational rank $r$ fibrations $X/\Spec\k$ with $K_X^2=d$, and write $\Cl_d^r$ for the set of isomorphism classes of rational rank $r$ fibrations $X/\p^1$ with $K_X^2=d$.
We will only use the notation for $r=1$ and omit the superscript $r$.

\begin{theorem}[{\cite[Theorem~2.4]{LamySchneider}}]
  Let $X/B$ be a rational Mori fiber space. Then $X$ lies in one of the following sets:
  \[\{\p^2\}, \Dl_5, \Dl_6, \Dl_8, \Cl_5, \Cl_6, \Cl_8,\]
  and
  \begin{itemize}
    \item $\Dl_8$ is the set of surfaces $X_8$ obtained via a link $\p^2\link21 X_8$ (blow up a point of degree $2$, then contract the strict transform of the line through this point),
    \item $\Dl_5\subset \Dl$ is the set of surfaces $X_5$ obtained via a link $\p^2\link51 X_5$ (blow up a point of degree $5$, then contract the strict transform of the conic through this point),
    \item $\Dl_6\subset \Dl$ is the set of surfaces $X_6$ obtained via a link $X_8\link31 X_6$ for some $X_8\in\Dl_8$ (blow up a point of degree $3$, then contract the strict transform of the conic through this point and the point of degree $2$ from $\p^2\link21 X_8$),
    \item $\Cl_8=\{\F_n\mid n\geq 0\}\subset \Cl$ the set of Hirzebruch surfaces with their structure as conic bundle,
    \item $\Cl_5\subset \Cl$ is the set of conic bundles obtained by blowing up a point of degree $4$ on $\p^2$, with fibration given by the pencil of conics through this point,
    \item $\Cl_6\subset \Cl$ is the set of conic  bundles obtained  by blowing up a point of degree $2$ on $X_8\in\Dl_8$, with fibration given by the pencil of conics through this point and the point of degree $2$ from $\p^2\link21 X_8$.
  \end{itemize}
\end{theorem}

\subsection{Finite fields}

Let $\k=\FFF_q$ be a finite field.
For any $d\geq1$, the Galois group $\Gal(\FFF_{q^d}/\k)$ is the cyclic group of order $d$ with generator $\alpha\from\alpha\mapsto\alpha^q$.
The (absolute) Galois group $\Gal(\overline{\FFF_q}/\k)$ is the inverse limit $\varprojlim\Gal(\FFF_{q^d}/\FFF_q)$; see \cite[Chapter 1, Section 1.5]{FriedJarden} for details.
We equip the projective plane with the action of the (absolute) Galois group $\Gal(\overline{\FFF_q}/\k)$.
So the action of the Galois group on $\PPP^2(\overline\k)$ is generated by the Frobenius morphism given by
\[
  \Frob\from [x:y:z]\mapsto [x^q:y^q:z^q],
\]
and the Galois orbit of a point $p\in\PPP^2(\overline{\k})$ consists of the points $\Frob^i(p)$ for $i=0,\ldots, d-1$ for some $d$.
Note that each geometric component of a Galois orbit on $\PPP^2$ of size $d$ is defined over $\FFF_{q^d}$.
Hence, for finite fields the following lemma is applicable for any two sets of four points in general position that consist of Galois orbits of the same sizes:

\begin{lemma}[{\cite[Lemma 6.11]{Schneider-relations}}]\label{lemma:SameFourPoints}
  Let $\kk$ be a perfect field and let $L/\kk$ be a finite Galois extension.
  Let $p_1,\ldots, p_4\in\PPP^2(L)$, no three collinear, and $q_1,\ldots,q_4\in\PPP^2(L)$, no three collinear, be points such that the sets $\{p_1,\ldots,p_4\}$ and $\{q_1,\ldots,q_4\}$ are invariant under the Galois action of $\Gal(L/\kk)$.
  Assume that for all $g\in\Gal(L/\kk)$ there exists $\sigma\in\Sym_4$ such that $g(p_i)=p_{\sigma(i)}$, and $g(q_i)=q_{\sigma(i)}$ for $i=1,\ldots,4$.
  Then, there exists $A\in\PGL_3(\kk)$ such that $Ap_i=q_i$ for $i=1,\ldots,4$.
\end{lemma}

In particular, this implies that over a finite field $\k=\F_q$, there is a unique point of degree $3$, a unique point of degree $4$, and a unique pair of two points of degree $2$ such that no three geometric components are collinear, up to a change of coordinates.
Hence, there is a unique Sarkisov link $\p^2\link33\p^2$, and a unique pencil of conics through a point of degree $4$, or two points of degree $2$. In particular, $|\Cl_5|=|\Cl_6|=1$.

Over more complicated fields such as $\Q$ this is not true: For any prime number $p$ and $\omega$ a root of $x^3-p$ (respectively a root of $x^4-p$), the Galois orbit of $[\omega^2:\omega:1]$ gives a point of degree $3$ (respectively $4$) such that no three of its geometric components are collinear, with non-isomorphic residue fields $\Q(\omega)$ for different prime numbers. Similarly, for every prime number $p$ one obtains a different element of $\Cl_5$ when taking a root of $x^4-p$.

Figure~\ref{figure:FanoPlane} depicts the seven points in $\PPP^2(\FFF_2)$ and the seven lines on $\PPP^2(\overline{\FFF_2})$ that are defined over $\FFF_2$.
It is meant as a reminder that one has to be careful with using the ``usual'' geometric intuition when working over $\FFF_2$.
\begin{figure}[h]
  \begin{center}
    \begin{tikzpicture}[scale=0.7]
    \draw (30:1)  -- (210:2)
          (150:1) -- (330:2)
          (270:1) -- (90:2)
          (90:2)  -- (210:2) -- (330:2) -- cycle
          (0:0)   circle (1);
    \node [fill=black, inner sep=2pt] at (0:0) {};
    \node [fill=black, inner sep=2pt] at (30:1) {};
    \node [fill=black, inner sep=2pt] at (90:2) {};
    \node [fill=black, inner sep=2pt] at (150:1) {};
    \node [fill=black, inner sep=2pt] at (210:2) {};
    \node [fill=black, inner sep=2pt] at (270:1) {};
    \node [fill=black, inner sep=2pt] at (330:2) {};
  \end{tikzpicture}
  \end{center}
  \caption{The seven $\FFF_2$-points and the seven $\FFF_2$-lines on $\PPP^2$}
  \label{figure:FanoPlane}
\end{figure}

\subsection{General position on $\p^2$ and $\p^1\times\p^1$ and some birational maps}
\label{section:GeneralPositionOnP2AndP1P1}

We fix the notation for $\P^1\times\P^1$ with coordinates $([x_0:x_1],[y_0:y_1])$:
A \textit{$(a,b)$-curve} on $\P^1\times\P^1$ is given by a polynomial of bidegree $(a,b)$, where $a$ is the degree in $x_0,x_1$ and $b$ is the degree in $y_0,y_1$.
A \textit{fiber} $f$ is a $(0,1)$-curve, and a \textit{section} $s$ is a $(1,0)$-curve.
So an $(a,b)$-curve is linearly equivalent to $as+bf$ and has self-intersection $2ab$.

\begin{lemma}\label{lemma:DelPezzoDescription}
  Let $\k$ be an algebraically closed field.
  Let $\pi\from S\to \P^2$ be a birational morphism, where $S$ is a smooth projective surface. The following conditions are equivalent:
  \begin{enumerate}
    \item\label{item:DelPezzoDescription--ample} $-K_S$ is ample (that is $S$ is a del Pezzo surface);
    \item\label{item:DelPezzoDescription--generalposition} the morphism $\pi$ is the blow-up at $0\leq r\leq 8$ proper points on $\P^2$ such that no $3$ are collinear, no $6$ are on the same conic, no $8$ lie on a cubic having a double point at one of them;
    \item\label{item:DelPezzoDescription--selfintersection} $K_S^2>0$ and any irreducible curve of $S$ has self-intersection $\geq -1$;
    \item\label{item:DelPezzoDescription--CK} $K_S^2>0$ and $C\cdot(-K_S)>0$ for any effective divisor $C$.
  \end{enumerate}

  Moreover, if there is a birational morphism $\rho\from S \to \P^1\times\P^1$, then the conditions \ref{item:DelPezzoDescription--ample} to \ref{item:DelPezzoDescription--CK} are equivalent to
  \begin{enumerate}[resume]
    \item\label{item:DelPezzoDescription--P1xP1} the morphism $\rho$ is the blow-up at $0\leq s\leq 7$ proper points on $\P^1\times\P^1$ such that no $2$ lie on a $(1,0)$- or $(0,1)$-curve, no $4$ lie on a $(1,1)$-curve,
    no $6$ lie on a $(2,1)$- or $(1,2)$-curve, not all $7$ lie on a $(2,2)$-curve having a double point at one of them.
  \end{enumerate}
\end{lemma}

\begin{proof}
  The equivalences of \ref{item:DelPezzoDescription--ample} to \ref{item:DelPezzoDescription--CK} are standard, see for example \cite[Theorem 1]{demazure80}.
  We show the equivalence of \ref{item:DelPezzoDescription--P1xP1} with the other conditions.
  The condition $s\leq 7$ is equivalent to $K_S^2\geq 1$ in \ref{item:DelPezzoDescription--selfintersection}.
  The Picard group $\Pic(S)$ is generated by $s$, $f$ and the exceptional divisors $E_1,\ldots, E_s$, where $s$ respectively $f$ denote the pull-back of a $(1,0)$- respectively $(0,1)$-curve.
  Note that $K_S=-2s-2f+\sum E_i$.
  Let $C$ be any irreducible curve on $S$, so $C\sim as+bf-\sum a_iE_i$ for some $a,b,a_1,\ldots,a_r\geq 0$.
  One computes
  \[
    C\cdot(-K_S)=2(a+b)-\sum a_i
  \]
  and (with the adjunction formula) the arithmetic genus is
  \[
    p_a(C)=\frac{C\cdot (C+K_S)}{2}+1=ab-(a+b)+1-\frac{1}{2}\sum a_i(a_i-1).
  \]
  For $a$ and $b$ at most $2$ we find
  \begin{center}
    \begin{tabular}{M | M | M} 
      (a,b) & ab-(a+b)+1 & 2(a+b)\\
      \hline
      (1,0) & 0 & 2\\
      (1,1) & 0 & 4\\
      (2,1) & 0 & 6\\
      (2,2) & 1 & 8
    \end{tabular}
  \end{center}
  In particular, the adjunction formula implies that if $(a,b)$ is $(1,0)$, $(1,1)$, or $(2,1)$ then $C$ is smooth, and if $(a,b)=(2,2)$ then $C$ has at most one double point.
  Assuming that $S$ is a del Pezzo surface and therefore satisfies condition \ref{item:DelPezzoDescription--CK}, we have $2(a+b)>\sum a_i$ and so \ref{item:DelPezzoDescription--P1xP1} follows.

  We assume now that condition~\ref{item:DelPezzoDescription--P1xP1} holds.
  The assumption implies that for $a,b\leq 2$ we have with the above table that $C\cdot (-K_S)>0$.
  For the remaining cases we may assume by symmetry that $a\geq 3$ and $b\leq a$.
  Since $C$ is irreducible, $C$ has no common components with $-K_S=2s+2f+\sum E_i$ and so $C\cdot(-K_S)\geq 0$.
  We assume $C\cdot (-K_S)=0$ and find a contradiction.
  The assumption gives with the adjunction formula
  \begin{align*}
      \sum a_i &= 2(a+b),\\
      \sum a_i^2 &= 2ab+2-2g,
  \end{align*} where $g=g(C)$.
  Using the Cauchy-Schwartz inequality we find
  \[
    4(a+b)^2=\left(\sum a_i\right)^2\leq 7\sum a_i^2 = 7(2ab+2-2g).
  \]
  This gives
  \[
    2(a-b)^2+ab\leq 7(1-g)\leq 7.
  \]
  One checks that the left hand side is $\geq 8$ for all $a\geq 3$, giving a contradiction. Indeed, if $b=a$, then $ab\geq 9$; if $b=a-1$, then the left hand side is at least $2+3\cdot 2=8$; if $b\leq a-2$ then the left hand side is at least $2\cdot 2^2=8$.
\end{proof}

The following lemma is very simple, but immensly helpful to check computationally the conditions for general position on a del Pezzo surface $S$, seen as a $\k$-structure on $\P^2$ or $\P^1\times\P^1$.

\begin{lemma}\label{lemma:MatrixForComputation}
  Let $S=\P^2$ (respectively $S=\P^1\times\P^1$), and let $p_1,\ldots,p_n$ be $n$ points on $S$.
  Let $s\leq n$.
  Then, there exists a curve of degree $d$ (respectively of bidegree $(a,b)$) through $p_1,\ldots,p_n$ and with singular points $p_{i_1},\ldots, p_{i_s}$ if and only if the kernel of the matrix $M\in \Mat_{n+3s,N}(\k)$ (respectively $M\in \Mat_{n+2s,N}$) has positive dimension,
  where $N$ is the number of monomials of degree $d$ in $\k[x,y,z]$ (respectively of bidegree $(a,b)$ in $\k[x_0,x_1,y_0,y_1]$) and $M=\left(\begin{smallmatrix}A\\ B\end{smallmatrix}\right)$ is such that
  \begin{enumerate}
    \item $A\in \Mat_{n, N}$ is the matrix whose $k$-th row is
      \[\left(\begin{array}{r r r }
        m_1(p_{i}) & \cdots & m_N(p_{i})
      \end{array}\right)\]
      for $i=1,\ldots,n$, where the $m_j$ denote all monomials in $\k[x,y,z]_d$ of degree $d$ (respectively in $\k[x_0,x_1,y_0,y_1]$ of bidegree $(a,b)$), and
    \item \begin{enumerate}
      \item in the case of $S=\P^2$, $B\in \Mat_{3s, N}$ is the matrix that consists of the $3s$ rows of the form
      \[\left(\begin{array}{r r r }
          \frac{\partial m_1}{\partial x}(p_{i_k}) & \cdots & \frac{\partial m_N}{\partial x}(p_{i_k}) \\
          \frac{\partial m_1}{\partial y}(p_{i_k}) & \cdots & \frac{\partial m_N}{\partial y}(p_{i_k}) \\
          \frac{\partial m_1}{\partial z}(p_{i_k}) & \cdots & \frac{\partial m_N}{\partial z}(p_{i_k}) \\
        \end{array}\right),\]
        for each $k=1,\ldots, s$,
        \item in the case of $S=\P^1\times\P^1$, we choose for each $k=1,\ldots, s$ affine coordinates $x\in\{x_0,x_1\}$, $y\in\{y_0,y_1\}$ and take the matrix $B\in \Mat_{2s, N}$ that consists of the $2s$ rows of the form
        \[\left(\begin{array}{r r r }
            \frac{\partial m_1}{\partial x}(p_{i_k}) & \cdots & \frac{\partial m_N}{\partial x}(p_{i_k}) \\
            \frac{\partial m_1}{\partial y}(p_{i_k}) & \cdots & \frac{\partial m_N}{\partial y}(p_{i_k})
          \end{array}\right)\]
          for each $k=1,\ldots, s$.
    \end{enumerate}
  \end{enumerate}
\end{lemma}

\begin{proof}
  A curve $C$ in $\P^2$ of degree $d$ is given by a polynomial $F$ of degree $d$, which is of the form $\sum_{i=1}^N \lambda_i m_i$, where $m_1,\ldots,m_N$ are all monomials in $\k[x,y,z]$ of degree $d$.
  Taking the (row) vector $v=(\lambda_1,\ldots,\lambda_n)$, the linear equation $A\cdot v=0$ means that $p_1,\ldots,p_n$ lie on the curve $C$.
  Then, $B\cdot v=0$ means that the $p_{i_k}$ for $k=1,\ldots,s$ are singular points of the curve $C$.

  The same argument works for a polynomial of degree $(a,b)$.
\end{proof}

Note that \[N=\begin{cases}\frac{(d+2)(d+1)}{2} & \text{if } S=\p^2\\ (a+1)(b+1) & \text{if } S=\p^1\times\p^1\end{cases}.\]

\begin{lemma}\label{lemma:DescriptionOfBertiniEtc}
  Let $\k$ be an algebraically closed field, and let $a,b$ be two integers. If they are equal, write $d=a=b$.
  Let $p_1,\ldots,p_a\in X$ be a set of $a$ points in {(del Pezzo--)} general position, and let $\sigma\colon Z\to X$ be the blow-up at these points with exceptional divisors $E_1,\ldots,E_a$.
  Let $\sigma'\colon Z\to X'$ be the contraction of some (pairwise disjoint) $(-1)$-curves $E_1',\ldots,E_b'$ on $Z$.
  If $X$, respectively $X'$, is $\p^2$, write $L$, respectively $L'$ for the pull-back under $\sigma$ (respectively $\sigma'$) of a general line.
  If $X$ (respectively $X'$) is $\p^1\times \p^1$, write $s,f$ (respectively $s',f'$) for the pull-back under $\sigma$ (respectively $\sigma'$) of the sections with respect to the two fibrations $\p^1\times\p^1\to\p^1$.
  Assume that $a,b$ and $E_j'$ are as below. Then $L'$ (respectively $s',f'$) are as below.

    \begin{longtable}{CLLLL}
      \multicolumn{4}{L}{X=X'=\p^2:}\\\toprule{}
      3 & E_j'=L-E_k-E_l, & \{j,k,l\}=\{1,2,3\} & \multicolumn{2}{L}{L'=2L-E_1-E_2-E_3}\\[\medskipamount]
      5 & E_1'=2L-\sum_{i=1}^5 E_i,& & \multicolumn{2}{L}{L'=3L-2E_1-\sum_{i=2}^5E_i}\\
       &  E_j'=L-E_1-E_j&  j=2,\ldots,5, &  \\[\medskipamount]
      6 & E_j'=2L-\sum_{i\neq j}E_i,& j=1,\ldots,6 & \multicolumn{2}{L}{L'=5L-2\sum_{i=1}^6E_i} \\[\medskipamount]
      7 & E_j'=3L-2E_j-\sum_{i\neq j}E_i,& j=1,\ldots,7 & \multicolumn{2}{L}{L'=8L-3\sum_{i=1}^7 E_i} \\[\medskipamount]
      8 & E_j'=6L-3E_j-2\sum_{i\neq j}E_i,& j=1,\ldots,8 & \multicolumn{2}{L}{L'=17L-6\sum_{i=1}^8E_i}.\\[\bigskipamount]
     \multicolumn{5}{L}{X=\p^2, X'=\p^1\times\p^1:}\\\toprule{}
     (2,1) & E_1'=L-E_1-E_2 & &   s'=L-E_1, & f'=L-E_2\\
     & & & \multicolumn{2}{L}{L=s'+f'-E_1'} \\[\medskipamount]
     (4,3) & E_j'=L-E_j-E_4 & j=1,2,3 &  s'=L-E_4, &f'=2L-\sum_{i=1}^4 E_i\\
     & & & \multicolumn{2}{L}{L=2s'+f'-\sum_{j=1}^3 E_j'}\\[\medskipamount]
    (6,5) & E_j'=2L-\sum_{k\neq j} E_k & j=1,\ldots,4 &\multicolumn{2}{L}{s'=3L-\sum_{i=1}^4 E_i-E_5-2E_6} \\
    &E_5'=L-E_5-E_6 & & \multicolumn{2}{L}{f'=3L-\sum_{i=1}^4 E_i-2E_5-E_6}\\
    & & & \multicolumn{2}{L}{L=3(s'+f')-2\sum_{j=1}^4E_j'-E_5'.}\\[\medskipamount]
    \multicolumn{4}{L}{X=X'=\p^1\times\p^1:}\\\toprule{}
     4 & E_j'=s+f-\sum_{i\neq j}E_i & j=1,\ldots,4 & \multicolumn{2}{L}{s'=s+2f-\sum_{i=1}^4 E_i,}\\
     & & &\multicolumn{2}{L}{f'=2s+f-\sum_{i=1}^4E_i.}
  \end{longtable}
\end{lemma}

\begin{proof}
  To see that the contraction $\sigma'\from Z\to X'$ exists, it is enough to check that the $E_i'$ are $(-1)$-curves that do not intersect pairwisely, that is, $E_i'^2=-1$ and $E_i\cdot E_j=0$ for $i\neq j$.
  If $X=X'=\p^2$, write $L'=dL-\sum_{i=1}^d a_iE_i$, and find the values for $d$ and $a_i$ via the equations $0=L'\cdot E_j'$, $3=(-K_X)\cdot L'$ (where $-K_X=3L-\sum_{i=1}^d E_i$) and $1=(L')^2$.
\end{proof}

\subsection{Generators of $\PGL_2(\k)$ and $\PGL_3(\k)$}
We recall some elementary generating sets of the projective linear groups $\PGL_2(\k)$ and $\PGL_3(\k)$.

\begin{lemma}\label{lemma:GeneratorsOfPGL2}
  Let $\kk$ be any field.
  Then $\PGL_2(\kk)$ is generated by the matrices
  \[
  \left(\begin{smallmatrix}0 & 1\\ 1&0\end{smallmatrix}\right),
  \left(\begin{smallmatrix}1 & 1\\ 0&1\end{smallmatrix}\right),
    \left(\begin{smallmatrix}1 & 0\\ 0&\lambda\end{smallmatrix}\right)
  \]
  for $\lambda\in\kk^*$.
  In particular, if $\kk=\FFF_2$ then $\PGL_2(\FFF_2)\simeq \Sym_3$ is generated by the two involutions $\left(\begin{smallmatrix}0 & 1\\ 1&0\end{smallmatrix}\right)$ and $\left(\begin{smallmatrix}1 & 1\\ 0&1\end{smallmatrix}\right)$.

  Moreover, $\PGL_2(\kk)$ is generated by involutions, for any field $\kk$.
\end{lemma}

\begin{proof}
  The general linear group $\GL_n(\kk)$ is generated by elementary matrices.
  For the case $n=2$ they are of the form
  \[
  \left(\begin{smallmatrix}0 & 1\\ 1&0\end{smallmatrix}\right),
  \left(\begin{smallmatrix}\lambda & 0\\ 0&1\end{smallmatrix}\right),
  \left(\begin{smallmatrix}1 & 0\\ 0&\lambda\end{smallmatrix}\right),
  \left(\begin{smallmatrix}1 & \lambda\\ 0&1\end{smallmatrix}\right),
  \left(\begin{smallmatrix}1 & 0\\ \lambda&1\end{smallmatrix}\right).
  \]
  Since $\left(\begin{smallmatrix}0 & 1\\ 1&0\end{smallmatrix}\right)
  \left(\begin{smallmatrix}1 & \lambda\\ 0&1\end{smallmatrix}\right)
  \left(\begin{smallmatrix}0 & 1\\ 1&0\end{smallmatrix}\right) = \left(\begin{smallmatrix}1 & 0\\ \lambda&1\end{smallmatrix}\right)$, we do not need $\left(\begin{smallmatrix}1 & 0\\ \lambda&1\end{smallmatrix}\right)$.
  In $\PGL_2$, $\left(\begin{smallmatrix}\lambda & 0\\ 0&1\end{smallmatrix}\right) = \left(\begin{smallmatrix}1 & 0\\ 0&\lambda\end{smallmatrix}\right)^{-1}$, and $\left(\begin{smallmatrix}1 & 0\\ 0&\lambda^{-1}\end{smallmatrix}\right)
  \left(\begin{smallmatrix}1 & 1\\ 0&1\end{smallmatrix}\right)
  \left(\begin{smallmatrix}1 & 0\\ 0&\lambda\end{smallmatrix}\right) = \left(\begin{smallmatrix}1 & \lambda\\ 0&1\end{smallmatrix}\right)$.
  The first part of the statement follows.
  For $\kk=\FFF_2$, the projective linear group is isomorphic to $\Sym_3$ via the action on the three $\FFF_2$-points on $\PPP^1$.

  Finally, note that all matrices in $\PGL_2(\kk)$ of the form  $\left(\begin{smallmatrix}a & b\\ c&-a\end{smallmatrix}\right)$ are involutions.
  As $\left(\begin{smallmatrix}1 & 0\\ 0&\lambda\end{smallmatrix}\right) = \left(\begin{smallmatrix}0 & 1\\ \lambda&0\end{smallmatrix}\right)
  \left(\begin{smallmatrix}0 & 1\\ 1&0\end{smallmatrix}\right)$ and
  $\left(\begin{smallmatrix}1 & 1\\ 0&1\end{smallmatrix}\right) =
  \left(\begin{smallmatrix}0 & 1\\ 1&0\end{smallmatrix}\right)
  \left(\begin{smallmatrix}1 & 0\\ 1&-1\end{smallmatrix}\right)
  \left(\begin{smallmatrix}0 & 1\\ -1&0\end{smallmatrix}\right)$,
  the three types of generating matrices are generated by involutions.
  This works in any characteristic.
\end{proof}

\begin{lemma}\label{lemma:PGL3GeneratedByInvolutions}
  Let $\kk$ be a field.
  The projective linear group $\PGL_3(\kk)$ is generated by involutions if and only if every element of $\kk$ is a cube.
\end{lemma}

\begin{proof}
  First, we assume that each element of $\kk$ is a cube.
  Permutation matrices are generated by involutions.
  As $\GL_3(\kk)$ is generated by elementary matrices, it is enough to check that $\left(\begin{smallmatrix}1 & \lambda & 0\\ 0&1 &0\\ 0&0&1\end{smallmatrix}\right)$ and $\left(\begin{smallmatrix}\lambda & 0 & 0\\ 0&1 &0\\ 0&0&1\end{smallmatrix}\right)$ in $\PGL_3(\kk)$ are generated by involutions for all $\lambda\in\kk^*$.
  For all $a\in\kk^*$, the matrix $M_a=\left(\begin{smallmatrix}0 & a & 0\\ a^{-1}&0 &0\\ 0&0&1\end{smallmatrix}\right)$ is an involution, and $M_aM_1=\left(\begin{smallmatrix}a & 0 & 0\\ 0&a^{-1} &0\\ 0&0&1\end{smallmatrix}\right)$.
  For $\lambda\in\kk^*$, let $a\in\kk$ be such that $a^3=\lambda$ and set $b=a^2$.
  Then, $\left(\begin{smallmatrix}\lambda & 0 & 0\\ 0&1 &0\\ 0&0&1\end{smallmatrix}\right) = \left(\begin{smallmatrix}1 & 0 & 0\\ 0&a^{-1} &0\\ 0&0&a\end{smallmatrix}\right) \left(\begin{smallmatrix}b & 0 & 0\\ 0&1 &0\\ 0&0&b^{-1}\end{smallmatrix}\right)$  is generated by involutions, where the equality holds in $\PGL_3(\kk)$.
  Observe that
  \[
    \left(\begin{smallmatrix}1& 1 & 0\\ 0&1 &0\\ 0&0&1\end{smallmatrix}\right)= \left(\begin{smallmatrix}0 & 1 & 0\\ 1&0 &0\\ 0&0&1\end{smallmatrix}\right) \left(\begin{smallmatrix}1 & 0 & 0\\ 1&-1 &0\\ 0&0&1\end{smallmatrix}\right) \left(\begin{smallmatrix}0 & 1 & 0\\ 1&0 &0\\ 0&0&1\end{smallmatrix}\right) \left(\begin{smallmatrix}-1 & 0 & 0\\ 0&1 &0\\ 0&0&1\end{smallmatrix}\right)
  \]
  is a product of involutions, and hence
  \[
    \left(\begin{smallmatrix}1& \lambda & 0\\ 0&1 &0\\ 0&0&1\end{smallmatrix}\right)= \left(\begin{smallmatrix}1 & 0 & 0\\ 0&\lambda^{-1} &0\\ 0&0&1\end{smallmatrix}\right) \left(\begin{smallmatrix}1 & 1 & 0\\ 0&1 &0\\ 0&0&1\end{smallmatrix}\right) \left(\begin{smallmatrix}1 & 0 & 0\\ 0&\lambda &0\\ 0&0&1\end{smallmatrix}\right)
  \]
  is generated by involutions.

  For the other direction, we assume that not every element of $\kk$ is a cube.
  We set
  \[
    G = \kk^*/\{a^3\mid a \in\kk^*\},
  \]
  which is a non-trivial abelian group with $3$-torsion.
  The determinant map $\GL_3(\kk)\to\kk^*$ induces a surjective group homomorphism $\PGL_3(\kk)\twoheadrightarrow G$.
  As $G$ does not contain any element of order $2$, the image of every involution is trivial.
  Hence, the group $\PGL_3(\kk)$ being generated by involutions contradicts the surjectivity of the group homomorphism.
\end{proof}

\subsection{Groups preserving a fibration}

Let $\k$ be any field, and let $\pi\colon\p^2\dashrightarrow \p^1$ be a dominant rational map.
We say that a map $f\in\Bir_{\k}(\PPP^2)$ \textit{preserves the fibration} $\pi\from\PPP^2\dashto\PPP^1$ if there exists an automorphism $\alpha\in\Aut_{\k}(\PPP^1)$ such that
\[
    \pi\comp f = \alpha\comp\pi.
\]
We denote the group of all such maps by $\Bir_\k(\p^2,\pi)$.
The group homomorphism $\Bir_\k(\p^2,\pi)\to\Aut_{\k}(\PPP^1)$ given by $f\mapsto\alpha$ induces an exact sequence
\[
1\rightarrow \Bir_\k(\p^2/\pi) \rightarrow \Bir_\k(\p^2,\pi) \rightarrow \Aut_k^{\pi}(\p^2)\rightarrow 1,
\]
and so $\Bir_\k(\p^2,\pi)$ is generated by $\Bir_\k(\p^2/\pi)$ and a set that surjects onto the image $\Aut_k^{\pi}(\p^2)$. We say that the elements of $\Bir_\k(\p^2/\pi)$ \textit{fix the fibration}.

We recall:
\begin{lemma}\label{lemma:BirMapsPreservingFibrationAreAutoOverFunctionField}
  Let $\kk$ be a field, and let $\pi\from\PPP^2\dashto\PPP^1$ be a rational dominant map defined over $\kk$ with generic fiber $C\subset \PPP^2_{\kk(\PPP^1)}$.
  Assume that $C$ is a conic that is geometrically irreducible.
  Then, $\Bir_\k(\p^2/\pi)\simeq \{f\in\Aut_{\kk(\PPP^1)}(\PPP^2)\mid f(C)=C\}$.
\end{lemma}

\begin{proof}
	We first recall the proof of $\Bir_\k(\p^2/\pi)\simeq \Bir_{\k(\p^1)}(C)$:
  We write $K=\kk(\PPP^1)$ and $L=\kk(\PPP^2)$ and consider the inclusion of fields $\kk\subset K\subset L$, where we identify $K$ with $\pi^*(K)\subset L$.
  For the first statement, we use the fact that for any variety $W$ over a field $\kk$ there is an (anti-)isomorphism $\Bir_\kk(W)\simeq \Aut_\kk(\kk(W))$ given by $f\mapsto f^*$.
  Applying this fact for $W=\PPP^2$ we obtain
  \begin{align*}
      \{f\in\Bir_\kk(\PPP^2)\mid \pi\comp f=\pi\} & \simeq \{f^*\in\Aut_\kk(L)\mid (\pi\comp f)^*=\pi^*\}.
  \end{align*}
  Since the condition $f^*\comp\pi^*=\pi^*$ is equivalent to say that $f^*|_{K}=\id_{K}$ (using the identification $K$ with $\pi^*(K)$), the above set is exactly the set of automorphisms of the field $L$ that fix the subfield $K$, namely $\Aut_{K}(L)$.
  Now we use that $\kk(C)=L$ and the fact from the beginning of the proof on $W=C$ and find an (anti-)isomorphism $\Aut_K(L)\simeq\Bir_K(C)$.

  Since every rational map from a smooth projective curve to a projective variety is a morphism, it holds that $\Bir_{K}(C)=\Aut_{K}(C)$. Since $C$ is a conic, the statement follows.
\end{proof}

In other words, writing $\pi\colon[s:t]\mapsto [F_0:F_1]$, where $F_0,F_1\in\k[s,t]$ are homogeneous of the same degree $d\leq 2$, there is a natural bijection
\[\Bir_\k(\p^2/\pi)=\{\phi\in \Bir_\k(\p^2) \mid \pi\circ\varphi=\pi\} \leftrightarrow \{A\in\PGL_3(\k(t)) \mid \text{$A$ fixes $tF_0-F_1=0$}\}.\]

\bigskip
For any field $\k$, the group $\JJ_1:=\Bir_{\k}(\P^2,\pi_1)$ of de Jonqui\`eres transformations consists of the maps that preserve the fibration $\pi_1\from [x:y:z]\dashto [y:z]$,
that is, they permute the lines through $[1:0:0]$. In this case, the generic fiber is isomorphic to $\p^1_{\k(t)}$.

\begin{lemma}\label{lemma:GeneratorsOfJ1}
	For an arbitrary field $\k$, the following holds:
  \begin{enumerate}
    \item The group $\JJ_1$ is isomorphic to the semi-direct product $\PGL_2(\k(t))\rtimes\PGL_2(\k)$, and via $(x,y)\mapsto[x:y:1]$, these are maps of the form
    \[
        \left(x,y\right)\mapsto \left(\frac{\alpha x+\beta}{\gamma x +\delta},\frac{ay+b}{cy+d}\right)
    \]
    with $\alpha,\beta,\gamma,\delta\in\k[y]$ satisfying $\alpha\delta-\beta\gamma\neq0$ and $a,b,c,d\in\k$ satisfying $ad-bc\neq 0$.
    \item $\JJ_1$ is contained in the group generated by $\Aut_{\k}(\p^2)\simeq \PGL_3(\k)$ and maps of the form
    \[
      g_p\from [x:y:z]\dashto[x:yp(t):zp(t)],
    \]
    where $p\in\k(t)^*$ is a rational function in $t={y}/{z}$.
  \end{enumerate}
\end{lemma}

\begin{proof}
  The exact sequence $\JJ_1=\Bir_\k(\p^2,\pi_1)\to\PGL_2(\k)$ from above splits:
  The image of the automorphisms $[x:y:z]\mapsto[x:ay+bz:cy+dz]$ in $\PGL_2(\k)$ is $\left(\begin{smallmatrix}a & b\\ c&d\end{smallmatrix}\right)$, so the image is $\PGL_2(\k)$ and the automorphisms of $\p^2$ surject onto the image.
  Moreover, $\Bir_\k(\p^2/\pi_1)$ is isomorphic to $\PGL_2(\k(t))$.

  So it remains to find generators of the kernel.
  By Lemma~\ref{lemma:GeneratorsOfPGL2}, $\PGL_2(\k(t))$ is generated by the following matrices in $\PGL_2(\k(t))$, corresponding to birational maps on $\AAA^2$, corresponding to birational maps on $\p^2$:
  \begin{align*}
      \left({\begin{smallmatrix}0 & 1\\ 1&0\end{smallmatrix}}\right) \from  (x,y)&\dashto(\frac{1}{x},y), &
      f\from [x:y:z]&\dashto [z^2:xy:xz],\\
      \left({\begin{smallmatrix}1 & 1\\ 0&1\end{smallmatrix}}\right)\from  (x,y)&\dashto(x+1,y), &
      f'\from[x:y:z]&\dashto [x+z:y:z]\in\Aut_{\k}(\p^2),\\
      \left({\begin{smallmatrix}1 & 0\\ 0&q\end{smallmatrix}}\right)\from  (x,y)&\dashto(\frac{x}{q(y)},y), &
      g_p\from[x:y:z]&\dashto [xz^d:yp(y,z):zp(y,z)]=[x:yq(t):zq(t)],
  \end{align*}
  where $p\in\k[y,z]_d$ is the homogenization of the non-zero polynomial $q\in\k[t]$ and $d$ is the degree of $q$, and $t=\frac{y}{z}$.
  Note that $f\notin \JJ_1$ but $f',g_p\in \JJ_1$.
  Note that since we include $\Aut_\k(\p^2)$, $f$ is redundant as $f = \beta\comp g_y\comp\alpha$ with the automorphisms $\alpha\from[x:y:z]\mapsto[y:z:x]$ and $\beta\from[x:y:z]\mapsto[y:x:z]$.
\end{proof}

\section{Groups preserving a pencil of conics}\label{section:InfiniteFamilies}

Let $\k$ be a perfect field.
We say that a del Pezzo surface $X$ of degree $5$ is of \emph{type $4$}, respectively of \emph{type $2+2$} if $\rho\colon X\to\p^2$ is the blow-up of a point of degree $4$, respectively the blow-up of two points of degree $2$. In both cases, this gives four geometric points $\PPPP=\{p_1,p_2,p_3,p_4\}$, and the pencil of conics through $\PPPP$ induces a morphism $\pi\colon X\to\p^1$ (unique up to postcomposition with elements of $\Aut_\k(\p^1)$):
\begin{center}
  \begin{tikzpicture}
    \matrix (m) [matrix of math nodes,row sep=1em,column sep=1em,minimum width=2em]
    {
        X & &\PPP^1.  \\
        \PPP^2 & & \\};
        \path[-stealth]
        (m-1-1) edge[] node[above] {$\pi$} (m-1-3)
        (m-1-1) edge[] node[left] {$\rho$} (m-2-1)
        ;
      \end{tikzpicture}
\end{center}
In both cases, $\pi\colon X\to\p^1$ has three singular fibers.
If $X$ is of type $4$ then $\pi\colon X\to\p^1$ is a Mori conic bundle and $X\in\Cl_5$. If $X$ is of type $2+2$, then there is one singular fiber such that both irreducible components are defined over $\k$, giving two contractions $\eta\colon X\to X_6$, $\eta'\colon X\to X_6'$ over $\p^1$ with $X_6,X_6'\in\Cl_6$.

Figures~\ref{figure:X4singularFibers} and~\ref{figure:X2singularFibers} illustrate $X$ of type $4$, respectively $2+2$ in the case of $\k=\F_2$. (See also \cite[Lemma~6.4 and Figure~6.1]{Schneider-relations}.)

\begin{figure}[!htbp]
  \caption{$X$ of type $4$. The Galois action on the singular fibers of $X$ and the seven $\FFF_2$-points}
  \label{figure:X4singularFibers}
  \begin{center}
    \begin{tikzpicture}[scale=0.5] 
      \node at (8.5,0) {$\XX_4$};
      \foreach \i in {0,2,4}
        {
        \draw[name path=L12] (-2+\i,2) -- (-3+\i,-0.5);
        \draw[name path=L34] (-3+\i,0.5) -- (-2+\i,-2);
        \ifthenelse {\i=4}
          {
            \path[name intersections={of=L12 and L34}];
            \node [fill=black,inner sep=2pt] at (intersection-1) {};
          }
          {}
        }
      \draw[name path=fiber1] (4.5,2) -- (4.5,-2); 
      \draw[name path=fiber2] (7,2) -- (7,-2); 
      \foreach \i in {1,0,-1}
        {
          \draw[name path=inv1, draw=none] (-2,\i) -- (7,\i);
          \path[name intersections={of=inv1 and fiber1}];
          \node [fill=black,inner sep=2pt,label=180:{}] at (intersection-1) {};
          \path[name intersections={of=inv1 and fiber2}];
          \node [fill=black,inner sep=2pt,label=180:{}] at (intersection-1) {};
        }
      \node[label=190:\tiny{$L_{12}$}] (L12) at (-2,2) {};
      \node[label=170:\tiny{$L_{34}$}] at (-2,-2) {};
      \node[label=190:\tiny{$L_{14}$}] (L14) at (0,2) {};
      \node[label=170:\tiny{$L_{23}$}] at (0,-2) {};
      \draw [{Latex[length=1.5mm]}-{Latex[length=1.5mm]}] (L12.north east) to [bend left] node [above]{} (L14.north west) ;
      \draw [{Latex[length=1.5mm]}-{Latex[length=1.5mm]}] (1.8,1) to [bend left] (1.8,-1) ;
      \node [] at (8.5,-2.5) {$\downarrow$};
      \node [] at (8.5,-4) {$\PPP^1$};
      \draw[name path=P1 ] (-4,-4) -- (7.5,-4);
      \node [fill=black,inner sep=2pt] at (1.25,-4) {};
      \node [fill=black,inner sep=2pt] at (4.5,-4) {};
      \node [fill=black,inner sep=2pt] at (7,-4) {};

    \end{tikzpicture}%
  \end{center}
\end{figure}

\begin{figure}[!htbp]
  \caption{$X$ of type $2+2$. The Galois action on the singular fibers of $\XX_2$ over $\k=\F_2$ and the seven $\FFF_2$-points}
  \label{figure:X2singularFibers}
  \begin{center}
    \begin{tikzpicture}[scale=0.5] 
      \node at (-5,0) {$X$};
      \foreach \i in {0,2,4}
        {
        \draw[name path=L12] (-2+\i,2) -- (-3+\i,-0.5);
        \draw[name path=L34] (-3+\i,0.5) -- (-2+\i,-2);
        \path[name intersections={of=L12 and L34}];
        \node [fill=black,inner sep=2pt,label=180:{}] at (intersection-1) {};
        \ifthenelse{\i=4}
          {
          \foreach \j in {0.7,1.5}
            {
            \draw[name path=inv1, draw=none] (0,\j) -- (6,\j);
            \path[name intersections={of=inv1 and L12}];
            \node [fill=black,inner sep=2pt,label=180:{}] at (intersection-1) {};
            \draw[name path=inv2, draw=none] (0,-\j) -- (6,-\j);
            \path[name intersections={of=inv2 and L34}];
            \node [fill=black,inner sep=2pt,label=180:{}] at (intersection-1) {};
            }
          }
          {
          \draw [{Latex[length=1.5mm]}-{Latex[length=1.5mm]}] (-2.2+\i,1) to [bend left] (-2.2+\i,-1) ;
          }

        }
        \node[label=190:\tiny{$L_{13}$}] (L12) at (-2,2) {};
        \node[label=170:\tiny{$L_{24}$}] at (-2,-2) {};
        \node[label=190:\tiny{$L_{14}$}] (L14) at (0,2) {};
        \node[label=170:\tiny{$L_{23}$}] at (0,-2) {};

      \node [] at (-5,-2.5) {$\downarrow$};
      \node [] at (-5,-4) {$\PPP^1$};
      \draw[name path=P1, ] (-4,-4) -- (4,-4);
      \node [fill=black,inner sep=2pt] at (-2.75,-4) {}; 
      \node [fill=black,inner sep=2pt] at (-0.75,-4) {};
      \node [fill=black,inner sep=2pt] at (1.25,-4) {};
    \end{tikzpicture}
  \end{center}
\end{figure}

\begin{lemma}\label{lemma:DoubleSectionNumbers}
  Let $\k$ be a perfect field and let $X$ be of type $4$ or $2+2$, coming from the blow-up $\rho\colon X\to\p^2$ of four points $\{p_1,\ldots,p_4\}\subset\p^2(\bar\k)$.
  Let $C\subset X$ be a curve. Write $d$ for the degree of $\Gamma=\rho(C)\subset\PPP^2$, and $m$ for the multiplicity of $\Gamma$ in $p_1,p_2$, and $m'$ for the multiplicity in $p_3,p_4$. Assume that $m'\geq m$.
  Then the following hold for a fiber $f$ in $X_{\bar\k}$:
   \begin{enumerate}
    \item\label{item:DoubleSectionNumbers--even} $C_{\bar\k}\cdot f$ is even.
    \item\label{item:DoubleSectionNumbers--double} $C_{\bar\k}\cdot f=2$ if and only if $d=m+m'+1$.
    Moreover, either $m'=m$ (and $d$ is odd) or $m'=m+1$ (and $d$ is even).
    \item\label{item:DoubleSectionNumbers--rational} If $C$ is rational and \ref{item:DoubleSectionNumbers--double} holds, then $C_{\bar\k}$ has exactly $m$ singular points.
  \end{enumerate}
  In particular, if $\k$ is a field of characteristic $2$ and $C$ is a double section, then $C_{\bar\k}$ has exactly $m$ singular points, and $d=m+m'+1$.
\end{lemma}

\begin{proof}
  We work over the algebraic closure $\bar\k$.
  Let $q_1,\ldots,q_r$ be the singular points of $C$ (including infinitely near points).
  After resolving the singularities, we can write \[C=dL-m(E_1+E_2)-m'(E_3+E_4)-\sum_{j=1}^rm_{q_j}E_{q_j},\]
  where $m_{q_j}$ denote the multiplicity of $C$ at $q_j$, and $E_{q_j}$ is the exceptional divisor of $q_j$ and $E_1,\ldots,E_4$ are the exceptional divisors of the four points of $\PPPP$, and $L$ is the pull-back of a general line in $\PPP^2$.
  We compute
  \begin{align*}
      C\cdot f &=\left(dL-m(E_1+E_2)-m'(E_3+E_4)\right)\cdot\left(2L-(E_1+\cdots+E_4)\right)\\
      &= 2d-2m-2m',
  \end{align*}
  hence $C\cdot f$ is even (providing \ref{item:DoubleSectionNumbers--even}), and $C\cdot f=2$ if and only if $d=m+m'+1$.
  Moreover, if $X$ is of type $4$ then $m=m'$.
  If $X$ is of type $2+2$, we write $f_1$ and $f_2$ for the irreducible components of the singular fiber that has both irreducible components defined over $\k$.
  We can write $f_1=L-E_1-E_2$ and $f_2=L-E_3-E_4$, so $C\cdot f_1=d-2m$ and $C\cdot f_2=d-2m'$.
  So if $2=C\cdot f=C\cdot f_1+C\cdot f_2$, then $(d-2m,d-2m')\in\{(0,2),(1,1),(2,0)\}$, implying that $m=m'+1$, $m=m'$, or $m'=m+1$.
  With the assumption $m'\geq m$, the first possibility does not occur.
  This yields \ref{item:DoubleSectionNumbers--double}.

  Note that $C_{\bar\k}\cdot f=2$ implies that $m_{q_j}=2$ for $j=1,\ldots,r$.
  The rationality of the curve inserted into the adjunction formula gives
  \begin{align*}
      0=g=\frac{1}{2}(d-1)(d-2) -2\frac{1}{2}m(m-1) -2\frac{1}{2}m'(m'-1)-r,
  \end{align*}
  so after replacing $m'=m$ and $d=2m+1$ (respectively $m'=m+1$ and $d=2m'$), we get $r=m=m'$ (respectively $r=m'-1=m$). This gives \ref{item:DoubleSectionNumbers--rational}.

  Finally, if $C$ is a double section then $C\cdot f=2$ and $C$ is rational, finishing the proof.
\end{proof}

\begin{remark}\label{remark:NoSections}
  Let $\k$ be a perfect field, and $X$ of type $4$ or $2+2$. The above lemma implies that there is no curve $C$ on $X$ with $C\cdot f=1$. Hence, $\rho\colon X\to \p^2$ has no sections.
\end{remark}

\begin{remark}\label{rem:DoubleSectionSingularPoints--AnyChar}
  Let $\k$ be a perfect field.
  Let $X$ be of type $4$ or $2+2$, let $f$ be a $\bar\k$-fiber on $X$. Let $C$ be a curve on $X$ with $C\cdot f=2$ and assume that it intersects $f$ in a unique point $p$.
  Having $C\cdot f=2$ implies that either $p$ lies on exactly one irreducible component of $f$ (and $p$ is either a smooth point of $C$ or singular with multiplicity $2$), or $p$ is the singular point of $f$ and a smooth point of $C$.
\end{remark}

\begin{notation}\label{notation:birationalmapcenter}
  Let $\k$ be a perfect field.
  Let $X$ be of type $4$ or of type $2+2$.
  Let $\Omega\subset X(\overline{\k})$ be a Galois orbit of points such that every fiber contains at most one point of the orbit.
  Depending on the position of the orbit $\Omega$, we define now a birational map $\varepsilon_\Omega\from X\dashto X$:
  \begin{enumerate}
    \item\label{item:birationalmapcenter--smooth}
    If the points of $\Omega$ lie on smooth fibers of $X$, we define $\varepsilon_\Omega$ to be the birational map that is given by the blow-up of $\Omega$, followed by the contraction of the strict transforms of the fiber through the points of the orbit.
    \item\label{item:birationalmapcenter--singular}
    If $\Omega$ lies on exactly one irreducible component of a singular fiber of $X$, we define $\varepsilon_{\Omega}$ to be the birational map that is given by the blow-up of $\Omega$, followed by the contraction of the strict transform of the other irreducible component of the same fiber.
  \end{enumerate}
  We obtain a birational map $\varepsilon_\Omega\colon X\dashto X'$ that preserves the fibration $X'\to\PPP^1$ induced by $\pi\from X\to\PPP^1$.
  Writing $\rho\colon X\to\p^2$ (respectively $\rho'\colon X'\to\p^2$) for a contraction of the four $(-1)$-sections on $X$ (respectively $X'$), there exists $\alpha\in\Aut_{\k}(\PPP^2)$ such that $\rho^{-1}\comp\alpha\comp\rho'$ is an isomorphism from $X'$ to $X$ \cite[Proposition 6.12]{Schneider-relations}.
  Composing $\varepsilon_\Omega$ with this isomorphism, we can assume that the source and target of $\varepsilon_\Omega$ are equal.
	We call the obtained $\varepsilon_{\Omega}\from X\dashto X$ an \emph{elementary transformation of $X$ centered at $\Omega$}.
  (If the position of $\Omega$ is not in one of the above cases, we do not define $\varepsilon_\Omega$.)
\end{notation}

Note that if $X$ is of type $4$, then elementary transformations are exactly the Sarkisov links $X\link{}{}X$, and only case~\ref{item:birationalmapcenter--smooth} occurs.

\bigskip

\begin{lemma}\label{lem:ElementaryTransformationOneBp}
  Let $\k$ be a perfect field.
  Let $X$ be of type $4$ or $2+2$, with $\rho\colon X\to\p^2$ the blow-up of four points $\PPPP$.
  \begin{enumerate}
    \item\label{it:ElementaryTransfOneBp--Aut} If $\varphi\in\Aut_\k(X)$, then $\rho\circ\varphi\circ\rho^{-1}\in \Aut_\k(\p^2)$.
    \item\label{it:ElementaryTransfOneBp--one} If $p\in X$ is a $\k$-rational point on $X$ such that an elementary transformation $\varepsilon_p$ centered at $p$ is defined, then $\rho\circ\varepsilon_p\circ\rho^{-1}\in \langle \Aut_\k(\p^2), J_1\rangle$.
  \end{enumerate}
\end{lemma}

\begin{proof}
  A crucial fact is that there are exactly four $(-1)$-sections on $X$, corresponding to the exceptional divisors of the blow-up $\rho\colon X\to\p^2$ (see for example \cite[Lemma 6.5]{Schneider-relations}).
  In~\ref{it:ElementaryTransfOneBp--Aut}, the $(-1)$-sections have to be permuted by $\varphi$, and so $\rho\comp\varphi\comp\rho^{-1}$ has no base points and is therefore an automorphism of $\p^2$.

  For \ref{it:ElementaryTransfOneBp--one}, let $p$ be a rational point.
  If $X$ is of type $4$, the relation $(\p^2,1,4)$ in \cite[Section A.2.4]{ZimmermannHab} (see also \cite[Figure~B.4]{LamySchneider}) gives a decomposition of $\rho\circ\varepsilon_p\circ\rho^{-1}$ as $\p^2\linkI1\F_1\link44\F_1\linkIII1\p^2 \in \langle\Aut_\k(\p^2), J_1\rangle$. One can check that this is a cubic map.

  If $X$ is of type $2+2$, there are two cases: If $p$ lies on a smooth fiber, then similar as above there is a decomposition of $\rho\circ\varepsilon_p\circ\rho^{-1}$ as $\p^2\linkI1\F_1\link22\F_1\link22\F_1\linkIII1\p^2 \in \langle\Aut_\k(\p^2), J_1\rangle$. Again, this is a cubic map.

  If $p$ lies on a singular fiber, write $F_{12}, F_{34}$ for the strict transforms in $X$ of the lines through the Galois orbit $\{p_1,p_2\}$, respectively $\{p_3,p_4\}$, and write $E_1,\ldots,E_4$ for  the exceptional divisors of $\rho$. So we can assume that $p$ lies on $F_{34}$ but not on $F_{12}$. As the $(-1)$-sections $E_3$ and $E_4$ do not pass $F_{12}$, the birational map $\varepsilon_p$ is an isomorphism on them. Moreover, since $p\notin F_{12}$, the three points $\rho(p)$, $p_1$, and $p_2$ are not collinear. The other two $(-1)$-sections in the image of $\varphi$ are the strict transforms of the two lines $L_1$ and $L_2$ through $\rho(p)$ and $p_1$, respectively $p_2$. Therefore, we can write $\rho\circ\varepsilon_p\circ\rho^{-1}$ as $\p^2\linkI1\F_1\link22\F_1\to\p^2\in\langle\Aut_\k(\p^2),J_1\rangle$, which is a quadratic map.
\end{proof}

\begin{lemma}\label{lemma:IntersectionOfSingularFiberNotBasePoint}
  Let $\kk$ be a perfect field and $X$, $X'$ be two surfaces defined over $\kk$.
  Let $\pi\from X\to\PPP^1$ and $\pi'\from X'\to\PPP^1$ be two surjective morphisms with connected fibers and let $\varphi\from X\dashto X'$ be a birational map preserving the fibrations, that is there is an automorphism $\alpha\in\Aut_\kk(\PPP^1)$ such that $\pi'\comp\varphi=\alpha\comp \pi$.
  Assume that each fiber of $\pi'$ contains at most two components.
  Then, for any singular fiber $F=\pi^{-1}(p)$ consisting of two $(-1)$-curves $F_1$ and $F_2$ (over $\bar\kk$), the intersection point $q=F_1\cap F_2$ is not a base point of $\varphi$.
\end{lemma}

\begin{proof}
  Let $\eta\from Z \to X$ and $\eta'\from Z \to X'$ be the minimal resolution of $\varphi$.
  Hence, $\eta$ and $\eta'$ are blow-ups of Galois orbits.
  Note that $\eta'$ contracts only curves lying in a fiber of $\pi\comp\eta$ since $\varphi$ preserves the fibrations.
  Let $C\subset Z$ be a $(-1)$-curve that is contracted by $\eta'$ (and thus is not contracted by $\eta$ by the minimality of the resolution).
  As $\eta(C)$ is contained in a fiber, there are two possibilities:
  Either $\eta(C)$ is a smooth fiber, or it is an irreducible component of a singular fiber, say $F$.
  In the second case, $\eta(C)$ is a $(-1)$-curve.
  As $\eta\from Z \to X$ is a sequence of blow-ups and $\eta(C)$ as well as $C$ have self-intersection $-1$, the morphism $\eta$ does not blow-up any points on $\eta(C)$.

  Therefore, for each irreducible component of a singular fiber, the intersection point of the two irreducible components is not a base point of $\eta$ and therefore it is not a base point of $\varphi$.
\end{proof}

\begin{lemma}\label{lem:ResolveAk}
  Let $\k$ be a perfect field. Let $X$ be of type $4$ or $2+2$, and let $C$ be a curve on $X$ such that $C\cdot f=2$.
  Let $\QQQQ=\{q_1,\ldots,q_d\}$ be a Galois-orbit of size $d$ on $X$, along which $C$ is singular.
  Then there exists a sequence $\varphi=\varphi_n\circ\cdots\circ\varphi_1$ of elementary transformations on $X$ centered at (Galois-orbits infinitely near to) $\QQQQ$ that resolves the singularity of $C$ at $\QQQQ$.
\end{lemma}
\begin{proof}
  Note that $m_{q_i}(C)=2$ since $C\cdot f=2$, and so there is a resolution of the singularity of $C$ at $\QQQQ$ into blow-ups with a resolution of type $A_k$. It is enough to remark that this resolution can be achieved with elementary transformations.
  Indeed, let $f_i$ denote the $\bar\k$-fiber going through $q_i$. Since $q_i$ is a singularity of $C$ and $C\cdot f_i=2$, $q_i$ is the unique intersection point of $C$ and $f_i$. As in Remark~\ref{rem:DoubleSectionSingularPoints--AnyChar}, $q_i$ lies on exactly one irreducible component of $f_i$. In particular, the elementary transformation of $X$ centered at $\QQQQ$ is defined. Blowing up $\QQQQ$, one observes that the curve contracted by $\varepsilon_\QQQQ$ does not intersect the strict transform of $C$. Hence, each blow-up in the resolution of the singularity can be achieved with an elementary transformation centered at $\QQQQ$, respectively at points infinitely near to $\QQQQ$.
\end{proof}

\subsection{Fixing vs. preserving a conic fibration}\label{sec:FixVsPreservingFibrationAndProofOfProp}

\begin{lemma}\label{lem:SingularFibersOntoSingularFibers}
  Let $\k$ be a perfect field.
  Let $\rho\colon X\to\p^2$ be the blow-up of $\p^2$ at four points $\PPPP$ such that $X$ is of type $4$ or $2+2$, with induced fibration $\pi\colon\p^2\dashrightarrow\p^1$.
  Let $\varphi\in \Bir_\k(\p^2,\pi)$. Then $\widehat\varphi = \rho^{-1}\circ\varphi\circ\rho\colon X\dashrightarrow X$ preserves the fibration of $X$ over $\p^1$.
  Moreover, if $X$ is of type $4$, then $\widehat\varphi$ sends the set of the three singular fibers onto itself.
  If $X$ is of type $2+2$, then $\widehat\varphi$ sends the set of two singular fibers with no geometric component defined over $\k$ onto itself.
\end{lemma}

\begin{proof}
  This follows from \cite[Lemma~6.10]{Schneider-relations}.
\end{proof}

\begin{proposition}\label{prop:FixesFibrationUpToAuto}
  Let $\k$ be a perfect field.
  Let $X\to \p^1$ be a rational Mori conic bundle, and $\varphi\colon X\dashrightarrow X$ a birational map preserving the fibration. Then there exists $\alpha\in\Aut_\k(X)$ such that $\alpha\circ\varphi$ fixes the fibration.
\end{proposition}

\begin{proof}
  If $X\to \p^1$ is a Hirzebruch surface it follows from Lemma~\ref{lemma:GeneratorsOfJ1}.

  If $X\in\Cl_5$ then $\Aut_\k(\p^2)\cap \Bir_\k(\p^2,\pi)$ surjects onto $\Aut_\k^\pi(\p^1)$ \cite[Lemma~3.20]{LamySchneider}.

  If $X\in\Cl_6$, then by \cite[Lemma~4.15]{Schneider-Zimmermann}, there is an exact sequence
  \[1\to \Aut_\k(X_6/\pi) \to \Aut_\k(X_6,\pi)\to D_\k\rtimes\Z/2\to 1,\]
  where $D_\k\rtimes \Z/2\Z\subset\Aut_\k(\p^1)$ is the induced action on the fibers.
  Then, by \cite[Proposition~3.25]{LamySchneider} we have in fact
  \[1\to \Bir_\k(X_6/\pi) \to \Bir_\k(X_6,\pi)\to D_\k\rtimes\Z/2\to 1.\]
\end{proof}

\begin{lemma}\label{lem:DecompositionElementaryTransformations}
  Let $\k$ be a perfect field and let $X$ be of type $4$ or $2+2$. Let $\varphi\colon X\dashrightarrow X$ be a birational map preserving the fibration.
  \begin{enumerate}
    \item\label{it:DecompositionElementaryTransformations--4} If $X$ is of type $4$, then $\varphi$ has a decomposition into elementary transformations and elments in $\Aut_\k(X)$.
    \item\label{it:DecompositionElementaryTransformations--2} If $X$ is of type $2+2$, let $\eta_i\colon X\to X_i$ be the two contractions with $X_i\in\Cl_6$ for $i=1,2$.
    Then $\varphi$ has a decomposition into elementary transformations, elements in $\Aut_\k(X)$, and elements in $\eta_i^{-1}\circ\Aut_\k(X_i)\circ\eta_i$ for $i=1,2$.
  \end{enumerate}
\end{lemma}

\begin{proof}
  The birational map $\psi=\eta_1\varphi\circ\eta_1^{-1}\colon X_1\dashrightarrow X_1$ preserves the fibration of the Mori conic bundle $X_1/\p^1$, and so there exists a decomposition of $\psi$ into Sarkisov links of type II (and automorphisms) between Mori conic bundles by \cite[Corollary~3.2]{Schneider-relations}.
  Write $r_i\in X_i(\k)$ for the base point of $\eta_i^{-1}$ for $i=1,2$.
  Let $\chi\colon Y\dashrightarrow Y'$ be a Sarkisov link of type II in the decomposition of $\psi$.
  In fact, $Y,Y'\in \{X_1,X_2\}$ by \cite[Lemma~6.12]{Schneider-relations} (see also \cite[Lemma~A.11]{LamySchneider}). Write $Y=X_a$ and $Y'=X_b$ with $a,b\in\{1,2\}$.
   Up to composition with an element of $\Aut_\k(X_b)$, we can assume that $\chi$ fixes the fibrations by Lemma~\ref{prop:FixesFibrationUpToAuto}.
   If the base point of $\chi$ is not on the fiber through $r_a$, then $\eta_b^{-1}\circ\chi\circ\eta_a$ is an elementary transformation centered at a Galois orbit lying on smooth fibers.
   If $\Bas(\chi)=r_a$ and $\Bas(\chi^{-1})=r_b$, then $\hat\chi=\eta_b^{-1}\circ\chi\circ\eta_a$ is an automorphism of $X$.
   If $\Bas(\chi)$ lies on the fiber through $r_a$ but is not $r_a$, and $\Bas(\chi^{-1})=r_b$, then $\hat\chi$ is an elementary transformation centered at $\eta_a^{-1}(\Bas(\chi))$, which is a $\k$-rational point on a singular fiber.
   If $\Bas(\chi)$ and $\Bas(\chi^{-1})$ lie on the fiber through $r_a$ respectively $r_b$ but are not $r_a$ respectively $r_b$, then $\chi$ is the composition of two elementary transformations centered at a $\k$-rational point on a singular fiber.
\end{proof}

\subsection{Double sections in characteristic $2$}

\begin{definition}\label{definition:DoubleSection}
  Let $\k$ be a perfect field of characteristic $2$.
  Let $X$ be of type $4$ or $2+2$. We say that a morphism $\sigma\from\PPP^1\to X$ (or its image $C={\sigma(\p^1)}$) is a \emph{double section} if $\sigma\from\PPP^1\to\sigma(\PPP^1)$ is birational and $\pi\comp\sigma=\Frob_2$ where $\Frob_2$ is given by $[s:t]\mapsto[s^2:t^2]$.
  In particular, any double section $C\subset X$ is rational and given by the image of a morphism $\sigma\from\PPP^1\to X$ of the form
  \[
    \sigma([s:t])=([a:b:c],[s^2:t^2])\in\PPP^2\times\PPP^1
  \]
  for some homogeneous polynomials $a,b,c\in\k[s,t]$ of the same degree.
\end{definition}

\begin{lemma}\label{lemma:DoubleSectionIsDouble}
  Let $\k$ be a perfect field of characteristic $2$ and let $X$ be of type $4$ or $2+2$.
  Let $\sigma\from\PPP^1\to X$ be a double section, and let $f$ be a $\bar\k$-fiber in $X$. Then $C_{\bar\k}\cdot f=2$ and $\pi\colon X\to\p^1$ induces a bijection $C(\bar\k)\to\PPP^1(\bar\k)$, where $C=\sigma(\PPP^1)$.
\end{lemma}

\begin{proof}
  The diagram
  \begin{center}
    \begin{tikzpicture}
      \matrix (m) [matrix of math nodes,row sep=1em,column sep=2em,minimum width=2em]
      {
           & X_\bullet  \\
          \PPP^1 & \PPP^1 \\};
          \path[-stealth]
          (m-2-1) edge[] node[left] {$\sigma$} (m-1-2)
          (m-2-1) edge[] node[below] {$\Frob_2$} (m-2-2)
          (m-1-2) edge[] node[right] {$\pi$} (m-2-2)
          ;
    \end{tikzpicture}
  \end{center}
  commutes.
  Hence, $C$ is parametrized by $[s:t]\mapsto\left([a:b:c],[s^2:t^2]\right)\in\PPP^2\times\PPP^1$, where $a,b,c\in\k[s,t]$ are homogeneous polynomials of the same degree.
  Inserting the parametrization of $C$ into the equation of any $\bar\k$-fiber $f$, that is, $\mu s+\lambda t=0$ for some $[\lambda:\mu]\in\p^1(\bar\k)$, yields $0=\mu s^2+\lambda t^2=\mu'^2 s^2+\lambda'^2 t^2$ for some $\lambda',\mu'\in\bar\k$.
  Hence, $C\cdot f=2$.
  Since $a\mapsto a^2$ is a bijective map on $\bar\k$, $C$ intersects any $\bar\k$-fiber in a unique $\bar\k$-point, giving bijectivity of $\pi$ on $C(\bar\k)$.
\end{proof}

\begin{remark}\label{rem:DoubleSectionSingularPoints--Char2}
  If $C$ is a double section, then Remark~\ref{rem:DoubleSectionSingularPoints--AnyChar} implies that $C$ intersects every fiber in a unique point (see Lemma~\ref{lemma:DoubleSectionIsDouble}).
  In particular, if $X$ is of type $4$ then $C$ passes through the singular point of all three singular fibers, and if $X$ is of type $2+2$ then $C$ contains (at least) the singular points of the two singular fibers whose components are not defined over $\k$.
\end{remark}

\begin{lemma}\label{lemma:PushForwardOfDoubleIsDouble}
  Let $\k$ be a perfect field of characteristic $2$.
  Let $X$ be of type $4$ or $2+2$.
  Let $C\subset X$ be a double section.
  Let $\varphi\from X\dashto X$ be a birational map (defined over $\k$) preserving the fibration.
  Then, $\varphi_*(C)\subset X$ is a double section.
\end{lemma}

\begin{proof}
  As $\varphi$ preserves the fibration $\pi\colon X\to\p^1$, there exists $\alpha\in\Aut_\k(\p^1)$ such that $\pi\circ\varphi=\alpha\circ\pi$. In particular, the curve $C$ is not contracted by $\varphi$ and so the map $\varphi\from C\dashto\varphi_*(C)$ is birational.
  Let $\sigma\from \PPP^1\to X$ be the double section with $\sigma(\PPP^1)=C$.
  We let $\sigma'=\varphi\comp\sigma\from \PPP^1\dashto X$, which is a morphism since any rational map from $\PPP^1$ to a projective variety is a morphism.
  We obtain the commutative diagram
  \begin{center}
    \begin{tikzpicture}
      \matrix (m) [matrix of math nodes,row sep=1em,column sep=1em,minimum width=2em]
      {
          X && X  \\
          \PPP^1 && \PPP^1 \\
          \PPP^1&& \p^1. \\};
          \path[-stealth]
          (m-1-1) edge[dashed] node[above] {$\varphi$} (m-1-3)
          (m-1-1) edge node[right] {$\pi$} (m-2-1)
          (m-2-1) edge[] node[above] {$\alpha$} (m-2-3)
          (m-1-3) edge node[right] {$\pi$} (m-2-3)
          (m-3-1) edge[bend left=90] node[left] {$\sigma$} (m-1-1)
          (m-3-1) edge node[above] {$\alpha'$} (m-3-3)
          (m-3-3) edge node {} (m-2-3)
          (m-3-1) edge node[left] {\scalebox{0.7}{$\Frob_2$}} (m-2-1)
          (m-3-3) edge node[right] {\scalebox{0.7}{$\Frob_2$}} (m-2-3)
          ;
        \end{tikzpicture}
  \end{center}
  In particular, $\pi\comp\sigma'=\alpha\circ\Frob_2$.
Note that taking square roots in $\k$ is a bijection $\k\to\k$, and so there exists $\alpha'\in\Aut_\k(\p^2)$ such that $\alpha\circ\Frob_2=\Frob_2\circ\alpha'$. Hence, $\sigma''=\sigma'\circ\alpha'^{-1}$ satisfies $\pi\circ\sigma''=\Frob_2$, and hence $\sigma''$ is a double section with $\sigma''(\p^1)=C$.
\end{proof}

\begin{definition}\label{def:strangeLine}
	Let $\k$ be a perfect field of characteristic $2$. Let $X$ be of type $4$ or of type $2+2$ given by the blow-up $\rho\colon X\to\p^2$ at four geometric points $\PPPP$. Let $L\subset \p^2$ be the unique line that is tangent to every conic through $\PPPP$. We will call $L$, as well as its strict transform on $X$, the \emph{strange line} with respect to $\PPPP$.
\end{definition}

\begin{lemma}[Smooth double sections]\label{lemma:SmoothDoubleSection}
  Let $\k$ be a perfect field in characteristic $2$.
  Let $X$ be of type $4$ or $2+2$, given by the blow-up $\rho\colon X\to\p^2$ of four points $p_1,\ldots,p_4$.
  Let $C$ be a double section on $X$ (defined over $\k$).
  Let $\Gamma=\rho(C)\subset\PPP^2$, and let $d$ be the degree of $\Gamma$.
  \begin{enumerate}
    \item\label{it:SmoothDoubleSection--1} If $d=1$, then $\Gamma$ is the {strange line}.
    \item\label{it:SmoothDoubleSection--2} If $d=2$, then $\Gamma$ is a conic that passes through two of the four points and is tangent to the line through the other two points. In particular, this happens only if $X$ is of type $2+2$.
    \item\label{it:SmoothDoubleSection--3} If $d=3$, then $\Gamma$ is a singular cubic passing through the four points.
  \end{enumerate}
  In particular, if $C$ is smooth, then $d\in\{1,2\}$.
  Moreover, if $d\in\{2,3\}$, there exists an elementary transformation $\varepsilon_r\colon X\dashrightarrow X$ centered at a $\k$-rational point $r\in X(\k)$
  such that $\rho((\varepsilon_r)_*(C))$ is the strange line.
\end{lemma}

\begin{proof}
  Let $\sigma\from\PPP^1\to X$ be the morphism  corresponding to the double section $C$, which is given by $\sigma([s:t])=([a:b:c],[s^2:t^2])$ for some homogeneous polynomials $a,b,c\in\k[s,t]$ of the same degree.
  By Lemma~\ref{lemma:DoubleSectionNumbers}, we have $d=m+m'+1$.
  If $d=1$ then $\Gamma$ is a line, and since $\Gamma$ comes from a double section it is tangent to every conic through $p_1,\ldots,p_4$. Hence, $\Gamma$ is the strange line.

  If $d=2$, then $m=0$ and $m'=1$, so $\Gamma$ is an irreducible conic passing through the orbit $p_3,p_4$.
  Consider the singular fiber $f=L_{12}+L_{34}$. Since $C\cdot f=2$, $\Gamma$ is tangent to $L_{12}$.
  Moreover, the intersection point of $\Gamma$ and $L_{12}$ is $\k$-rational, whose preimage $r\in X(\k)$ lies on exactly one irreducible component of the singular fiber $f$. Hence, an elementary transformation $\varepsilon_r$ centered at $r$ is defined, and $C'=(\varepsilon_r)_*(C)$ is a double section (\ref{lemma:DoubleSectionIsDouble}) with self-intersection $1$. Hence, $\rho(C')$ is the strange line as in \ref{it:SmoothDoubleSection--1}.

  If $d=3$, then $m=m'=1$ (because $m$ and $m'$ differ by at most $1$, see Lemma~\ref{lemma:DoubleSectionNumbers}). So $\Gamma$ is a singular cubic through $p_1,\ldots,p_4$.
  Moreover, the singularity is $\k$-rational, and its preimage under $\rho$ gives $r\in X(\k)$ which lies on a smooth fiber of $X$. Hence, an elementary transformation $\varepsilon_r$ centered at $r$ is defined, and $\rho((\varepsilon_r)_*(C))$ is the strange line similar to before.

  Finally, if $C$ is smooth then the number of singular points is $m=0$, and so $d=m'+1\in\{1,2\}$ (because $m'\leq m+1=1$).
\end{proof}

\begin{lemma}\label{lem:DoubleSectionOntoDoubleSection}
  Let $\k$ be a perfect field of characteristic $2$.
  Let $X$ be of type $4$ or $2+2$, and let $\varepsilon_\Omega\colon X\dashrightarrow X$ be an elementary transformation centered at a Galois-orbit $\Omega\subset X$ lying on the strange line.
  Then the strange line is sent onto itself, and if $\varepsilon_\Omega$ fixes the fibration then $\varepsilon_\Omega^2\in\Aut(X)$.
\end{lemma}
\begin{proof}
  Note that the strange line $T\subset X$ is the unique smooth double section with self-intersection $1$: Indeed, if $T\subset X$ is a such a smooth double section, then by Lemma~\ref{lemma:DoubleSectionNumbers}, $T=dL-m'(E_3+E_4)$, with $m'\in\{0,1\}$, and so $1=T^2=d^2-2m'^2$ implies $m'=0$ and $d=1$. Being a double section implies that $T$ is the strange line on $X$.
  As $T$ is tangent to every fiber $f$ going through a point in $\Omega$, blowing up $\Omega$ gives strict transforms $\tilde C$ and $\tilde f$ that intersect in a point on the exceptional divisor. In particular, $({\varepsilon_\Omega})_*(T)$ is again a smooth double section with self-intersection $1$.

  In other words, if $\varepsilon_\Omega$ fixes the fibration then it sends the fiber through $q\in\Omega\subset T$ onto the intersection point of $T$ with that same fiber, which is $q$. Hence, $\varepsilon_\Omega^2$ has no base points.
\end{proof}

\subsection{Explicit equations for odd extensions of $\F_2$, and proof of Proposition~\ref{prop:ParametrizationFiberingType}}

In this section, we fix the following two polynomials $f_2,f_4\in\F_2[x]$:
\begin{itemize}
  \item $f_2=x^2+x+1\in\F_2[x]$ is irreducible with two distinct roots in $\F_4$, and
  \item $f_4=x^4+x+1\in\F_2[x]$ is irreducible with four distinct roots in $\F_{16}$.
\end{itemize}
Therefore, $f_2$ (respectively $f_4$) is irreducible over $\k=\F_{2^N}$ for $N$ odd, with roots in $\k_2=\F_{2^{2N}}$ (respectively in $\k_4=\F_{2^{4N}}$).

\begin{notation}\label{not:EquationsX}
  Let $\k=\F_{2^N}$ with $N$ odd, and $f_2,f_4\in\k[x]$ as directly above.
  We fix equations for $X$ of type $4$, respectively type $2+2$:
  \begin{enumerate}
    \item For $X$ of type $4$, let $\xi_1,\ldots,\xi_4\in\k_4$ be the roots of $f_4$. Then
    \[\PPPP=\{p_i=[1:\xi_i:\xi_i^2]\}_{i=1}^4\]
    is a Galois orbit of size $4$.
    The two conics given by the zero set of \[F_0=y^2+xz,  \text{ and } F_1=x^2+xy+z^2\] define a basis of the pencil of conics through $\PPPP$.
    \item For $X$ of type $2+2$, let $\omega,\omega+1\in\k_2$ be the roots of $f_2$.
    We consider the union $\PPPP=\QQQQ_1\cup\QQQQ_2$ of the two Galois orbits $\QQQQ_1=\{[1:0:\omega],[1:0:\omega+1]\}$, $\QQQQ_2=\{[1:1:\omega],[1:1:\omega+1]\}.$
    The two conics given by the zero set of \[F_0=y(x+y) \text{ and } F_1=x^2+xz+z^2\] define a basis of the pencil of conics through $\PPPP$.
  \end{enumerate}
  In both cases, we fix the fibration of conics $\pi\colon\p^2\dashrightarrow\p^1$,
  \[\pi\colon[x:y:z]\mapsto[F_0:F_1],\]
  and we will write $\JJ_2$ (respectively $\JJ_4$) for the group $\Bir_\k(\p^2,\pi)$ of birational maps preserving the fibration in the case of type $2+2$ (respectively type $4$).
  The blow-up $\rho\colon X\to\p^2$ at $\PPPP$ can be described as the projection of the surface
  \[X=\{([x:y:z],[s:t])\in\p^2\times\p^1\mid tF_0-sF_1=0\}\]
  to $\p^2$, and we denote the projection $X\to\p^1$ again by $\pi$.

  Note that in these coordinates, the strange line with respect to $\PPPP$ is given by $x=0$.
\end{notation}

\bigskip
From now on until the end of this section, we will assume Notation~\ref{not:EquationsX}.
In these coordinates, we give now explicit equations for maps fixing the fibration that also fix the strange line.

\begin{lemma}\label{lem:SpecialInvolutions}
  Let $\k=\F_{2^N}$ with $N$ odd. For $i\in\{2,4\}$ write $t=F_1/F_0$ with $F_0,F_1$ as in Notation~\ref{not:EquationsX}.
  For $[u:v]\in\p^1(\k(t))$,
  \begin{enumerate}
    \item if $i=4$ set $a(u,v)=\frac{u+tv}{tu^2+v^2}$, respectively,
    \item if $i=2$ set $a(u,v)=\frac{tu+v}{tu^2+v^2}$.
  \end{enumerate}
   Then,
  \[\chi_{i,[u:v]}\colon [x:y:z] \mapsto [x:ua(u,v)x+y:va(u,v)x+z]\]
  is an involution in $\Bir_{\k}(\p^2)$ that fixes the fibration $\pi_i$.
  Moreover, it fixes the strange line $x=0$.
\end{lemma}

\begin{proof}
  A direct computation with the explicit definition gives that $\chi_{i,[u:v]}$ are involutions in characteristic $2$, and that they fix $x=0$.
  To check that they fix the fibration, write $b=ua$, $c=va$. If $i=4$ then a small computation gives $t(b^2+c)-(b+c^2)=0$ and so
  \begin{align*}
    \pi\circ\chi_{4,[u:v]}([x:y:z])&= [x^2(b^2+c)+y^2+xz:x^2(b+c^2)+x^2+xy+z^2]\\
    &= [x^2(b^2+c)+(y^2+xz):x^2(b+c^2)+(x^2+xy+z^2)]\\
    &= [x^2(b^2+c)+(y^2+xz):x^2t(b^2+c)+t(y^2+xz)]\\
    &=[1:t]=[F_0:F_1].
  \end{align*}
  Similar for $i=2$ (but with $t(b+b^2)-(c+c^2)=0$).
\end{proof}

\begin{proposition}\label{proposition:ExplicitFamily}
  Let $\k=\F_{2^N}$ with $N$ odd, and let $i\in\{2,4\}$.
  Let $\varphi\in\Bir_\k(\p^2/\pi_i)$ be such that
  $\varphi$ fixes the strange line $x=0$.
  Then $\varphi=\varphi_{i,[u:v]}$ for some $[u:v]\in\p^1(\k(t))$ with $t={F_1}/{F_0}$ as in Lemma~\ref{lem:SpecialInvolutions}.
\end{proposition}

\begin{proof}
  Let $A=(a_{ij})_{ij}\in\PGL_3(\k(t))$ be such that it preserves the conic $C\colon F_1+tF_0$ as in Notation~\ref{not:EquationsX}. By assumption, $A$ preserves the line $x=0$ (implying $a_{12}=a_{13}=0$), and hence the intersection point with $C$, which is $[0:1:\sqrt{t}]$.
  This implies $a_{32}=ta_{23}$ and $a_{22}=a_{33}$.
  So we can write $A$ as
  \[A =
    \left(\begin{smallmatrix}
      1 & 0 & 0\\
      b & d & e\\
      c & et & d\\
    \end{smallmatrix}\right)\in\PGL_3(\k(t)).\]

  All tangents to $C$ are of the form $T\colon (tz_0+y_0)x+x_0y+tx_0z$ for some $[x_0:y_0:z_0]\in C$. In particular, every tangent contains $[0:t:1]$. Hence $A$ fixes $[0:t:1]$.
  This implies that $e=et^3$ and hence $e=0$. In particular, we can assume $d=1$.
  So $A$ is given by the map \[[x:y:z]\mapsto[x:bx+y:cx+z].\]

  \textbf{Case $\JJ_4$}: The assumption that $A$ maps $C$ onto $C$ gives for $[x:y:z]\in C$ that
  \begin{align*}
    0 &= t(y^2+axz)+a^2x^2+axy+z^2+ x^2(t(b^2+ac)+ab+c^2).
  \end{align*}
  Replacing $ty^2+z^2=txz+x^2+xy$, we find $a=1$ and $(b,c)\in\A^2(\k(t))$ lie on the conic $t(b^2+c)+b+c^2=0$.
  The projection from $(b,c)=(0,0)$ (which corresponds to $\id_{\p^2}$) gives a parametrization $[u:v]\mapsto[tu^2+v^2:u(u+tv):v(u+tv)]$.

  \textbf{Case $\JJ_2$}: Similar as above, the assumption that $A$ maps $C$ onto $C$ gives $a=1$ and $(b,c)\in\A^2(\k(t))$ satisfy $t(b+b^2)+c+c^2=0$.
  The projection from $(b,c)=(0,0)$ gives a parametrization $[u:v]\mapsto[tu^2+v^2:u(tu+v):v(tu+v)]$.
\end{proof}

The next goal is to prove Proposition~\ref{prop:ParametrizationFiberingType}. For this, we will show the importance of the birational maps in $\Bir_\k(\p^2/\pi_i)$ that fix the strange line (Proposition~\ref{proposition:GeneratorsOfBirationalMap4}). Let us look again at double sections!

\bigskip

Given $X$ of type $4$ or $2+2$ as in Notation~\ref{not:EquationsX}, we introduce now the surface $Y$ that is the fiber product of $X$ and $\PPP^1$ over $\PPP^1$ with $\Frob_2\from\PPP^1\to\PPP^1$, $[s:t]\mapsto[s^2:t^2]$, giving rise to the commutative diagram
\begin{center}
  \begin{tikzpicture}
    \matrix (m) [matrix of math nodes,row sep=1em,column sep=2em,minimum width=2em]
    {
        Y=X\times_{\PPP^1}\PPP^1 & X  \\
        \PPP^1 & \PPP^1. \\};
        \path[-stealth]
        (m-1-1) edge[] node[above] {$F$} (m-1-2)
        (m-1-1) edge node[left] {$\pi_Y$} (m-2-1)
        (m-2-1) edge[] node[below] {$\Frob_2$} (m-2-2)
        (m-1-2) edge node[right] {$\pi$} (m-2-2)
        ;
      \end{tikzpicture}
\end{center}
In other words, $Y$ is given by the equation \[Y\colon t^2F_0+s^2F_1=0\p^2\times\p^1.\]

\begin{lemma}\label{lem:SingularitiesOfY}
  Let $\k=\F_{2^N}$ with $N$ odd.
  Let $X$ be of type $4$ or $2+2$ and let $Y$ be as above.
  Then $Y$ is singular and it has exactly three singularities over $\bar\k$:
  If $X$ is of type $4$, the singularities are $([0:1:a],[1:a^2])$ for $a\in\F_4^*$. If $X$ is of type $2+2$, the singularities are $([0:a:b],[a:b])$ for $[a:b]\in\p^2(\F_2)$.
  Moreover, $x=0$ in $Y$ is non-reduced.
\end{lemma}

\begin{proof}
  Note that the derivatives of $F=t^2F_0+s^2F_1$ with respect to $s,t$ are zero. Taking $\frac{\partial F}{\partial y}=\frac{\partial F}{\partial z}=0$ implies $x=0$. Hence $0=t^2y^2+s^2z^2=(ty+sz)^2$, so $[y:z]=[s:t]$.
  For type $4$, we have $0=\frac{\partial F}{\partial x}=t^2z+s^2y$ and so $[y:z]=[t^2:s^2]$. Hence $[y:z]=[1:a]$ with $a^3=1$, giving $a\in\F_4^*$.
  For type $2+2$, we have $0=\frac{\partial F}{\partial x}=t^2y+s^2z$ and so $[y:z]=[s^2:t^2]$. Hence $[y:z]=[y^2:z^2]$, which is equivalent to $[y:z]\in\p^2(\F_2)$.
  Finally, inserting $x=0$ in the equation of $Y$ gives $(ty+sz)^2=0$.
\end{proof}

\begin{lemma}[Sections vs.~double sections]\label{lemma:SectionVsDoubleSection}
  Let $\k=\F_{2^N}$ with $N$ odd.
  Given $X$ of type $4$ or $2+2$, let $Y$ in $\PPP^2\times\PPP^1$ be as defined directly above.
  Let $F\from Y\to X$ be the map given by $(p,[s:t])\mapsto (p,[s^2:t^2])$.
  Let $a,b,c\in\k[s,t]$ be three homogeneous polynomials of the same degree without common factors.
  Consider the maps $\sigma,\sigma'\from\PPP^1\to\PPP^2\times\PPP^1$ given by \begin{align*}
    \sigma\from& [s:t]\mapsto ([a:b:c],[s^2:t^2]),\\
    \sigma'\from& [s:t]\mapsto ([a:b:c],[s:t]).
  \end{align*}
  Then, $\sigma\from\PPP^1\to X$ is a double section if and only if $\sigma'\from\PPP^1\to Y$ is a section.

  In other words, if $C\subset X$ and $C'\subset Y$ are curves with $F(C')=C$, then $C$ is a double section in $X$ if and only if $C'$ is a section in $Y$.
\end{lemma}

\begin{proof}
  First, if $([a:b:c],[s^2:t^2])\in X$, then $([a:b:c],[s:t])\in Y$. Hence, if $\sigma\from\PPP^1\to X$ is a double section, then $\sigma'\from\PPP^1\to Y$, and $\sigma'$ is a section.

  It is clear that $\sigma([s:t])$ lies in $X$ if and only if $\sigma'([s:t])$ lies in $Y$.
  In this case, we also have directly that $\pi\comp\sigma=\Frob_2$ and that $\pi_Y\comp\sigma'=\id_{\PPP^1}$.

  So it is enough to show that when $\sigma'\from\PPP^1\to Y$ is a section, then $\sigma\from\PPP^1\to\sigma(\PPP^1)=C$ is birational.
  Writing $C'=\sigma'(\PPP^1)$ we get a commutative diagram
  \begin{center}
    \begin{tikzpicture}
      \matrix (m) [matrix of math nodes,row sep=1em,column sep=4em,minimum width=2em]
      {
          C' & C  \\
          \PPP^1 & \PPP^1. \\};
          \path[-stealth]
          (m-1-1) edge[] node[above] {$F$} (m-1-2)
          (m-1-1) edge node[right] {$\pi_Y$} (m-2-1)
          (m-2-1) edge[] node[below] {$\Frob_2$} (m-2-2)
          (m-1-2) edge node[right] {$\pi$} (m-2-2)
          (m-2-1) edge node[above] {$\sigma$} (m-1-2)
          (m-2-1) edge[bend left] node[left] {$\sigma'$} (m-1-1)
          ;
        \end{tikzpicture}
  \end{center}
  Since $\sigma\from\PPP^1\to C$ is dominant, we have an inclusion $\sigma^*\from \k(C)\hookrightarrow{} \k(\PPP^1)$.
  We prove that $\sigma^*\from k(C)\simto \k(\PPP^1)$ is an isomorphism, implying that $\sigma\from \PPP^1\to C$ is birational.
  With the parametrization of $C$ we see that $\frac{s^2}{t^2}\in \k(C)$.
  Letting $u=\frac{s}{t}$ we get an inclusion of fields $\k(u^2)\subset \k(C)\subset \k(\PPP^1)=\k(u)$.
  Since $\k(u^2)\subset \k(u)$ is an extension of degree $2$, either $\k(C)=\k(u)$ and we are done, or $\k(C)=\k(u^2)$.
  We will find a contradiction to the latter case.
  Assuming that $c\neq0$ (otherwise $a$ or $b$ is non-zero and we proceed analogously), we have that $\frac{a}{c}$ and $\frac{b}{c}$ are elements in $\k(C)=\k(u^2)$, hence $a,b,c\in\k[s^2,t^2]$.
  So we may write $a(s,t)=\alpha(s^2,t^2)$, $b(s,t)=\beta(s^2,t^2)$, and $c(s,t)=\gamma(s^2,t^2)$ for some $\alpha,\beta,\gamma\in\k[s,t]$ homogeneous of the same degree, which we may choose without common divisor.
  We consider the morphism $\tau\from\PPP^1\to\PPP^2\times\PPP^1$ given by $[s:t]\mapsto \left([\alpha(s,t):\beta(s,t):\gamma(s,t)],[s:t]\right)$ and observe that $\pi\comp\tau=\id_{\PPP^1}$.
  So $\tau$ is a section on $X$, which does not exist (see Remark~\ref{remark:NoSections}.
\end{proof}

\begin{lemma}\label{lemma:FiberProduct}
  Let $\k=\F_{2^N}$ with $N$ odd.
  Given $X$ of type $4$ or $2+2$, let $Y$ be as above.
  The rational map $\varphi\from Y\dashto \PPP^1_{[u:v]}\times\PPP^1_{[s:t]}$ given by
  \[
    ([x:y:z],[s:t])\mapsto ([sx:ty+sz],[s:t])
  \]
  is birational. Moreover, the following hold:
  \begin{enumerate}
    \item\label{it:FiberProduct--deg} The inverse map is of degree $3$ in $s,t$ for type $4$, and of degree $2$ in $s,t$ for type $2+2$.
    \item\label{it:FiberProduct--def} If $f$ is a smooth fiber given by $\lambda t+\mu s=0$ with $[\lambda:\mu]\in\p^1(\bar\k)\setminus\p^1(\F_2)$, then $\varphi|_f$ is an isomorphism.
    \item\label{it:FiberProduct--xis0}  $x=0$ in $Y$ is mapped onto $u=0$ in $\p^1\times\p^1$.
  \end{enumerate}
\end{lemma}

\begin{proof}
  It is enough to observe that the conic $C_t\colon tF_0+F_1=0$ in $\p^2_{\k(t)}$ admits a $\k(t)$-isomorphism to $\p^1_{\k(t)}$, given by $[x:y:z]\mapsto[x:ty+z]$ when $[x:y:z]\neq [0:1:t]$, and it is given by
  \begin{align*}
    [x:y:z]&\mapsto [ty+z:x+y+t^2z]&&\text{if $[x:y:z]\neq [t^3+1:1:t]$ (type $4$), respectively}\\
    [x:y:z]&\mapsto [ty+z:x+t^2y+z]&&\text{if $[x:y:z]\neq [t^2+t:1:t]$ (type $2+2$)}.
  \end{align*}
  Indeed, one can check that the following are the inverse maps.
  If $X$ is of type $4$, then the inverse is given by
  \begin{align*}
    [u:v] \mapsto [(1+t^3)u^2:u^2+t^2uv+v^2:tu^2+uv+tv^2].
  \end{align*}
  If $X$ is of type $2+2$, then the inverse is given by
  \begin{align*}
    [u:v] \mapsto [(t+t^2)u^2:u^2+uv+v^2:tu^2+t^2uv+tv^2].
  \end{align*}
  In particular, this implies \ref{it:FiberProduct--deg}

  \ref{it:FiberProduct--def}: Having $[s:t]\notin\p^2(\F_2)$ means in particular that we can write it as $[1:t]$ with $t\neq0$.
  For the case of type $2+2$ this is equivalent with not being on a singular fiber. Moreover, if $\varphi$ is not defined then $x=t^2+t=0$, and if $\varphi^{-1}$ is not defined then $t^2+t=0$. Hence $\varphi$ and $\varphi^{-1}$ are morphisms whenever $t\notin\F_2$.

  For the case of type $4$, not being on a singular fiber implies $t^3+1\neq 0$. If $\varphi$ is not defined then $x=t^3+1=0$, and if $\varphi^{-1}$ is not defined then $t^3+1=0$. Hence $\varphi$ and $\varphi^{-1}$ are morphisms whenever $t\notin\F_2$ corresponds to a smooth fiber.

  \ref{it:FiberProduct--xis0}: Observe that $x=0$ in $Y$ implies $ty+z=0$, hence $x=0$ is mapped onto $u=0$.
\end{proof}

\begin{remark}\label{rem:ExplicitDoubleSection}
  Lemma~\ref{lemma:FiberProduct} gives explicit equations for how a section in $\p^1\times\p^1$ corresponds to a double section in $X$.
  Since $\varphi\colon Y\dashrightarrow\p^1\times\p^1$ from Lemma~\ref{lemma:FiberProduct} fixes the fibration, a section $\sigma''\colon\p^1\to\p^1\times\p^1$ induces a section on $Y$ which then induces a double section $\sigma=F\circ\varphi^{-1}\circ\sigma''\colon\p^1\to X$ on $X$ (Lemma~\ref{lemma:SectionVsDoubleSection}):
  \begin{center}
    \begin{tikzcd}[row sep=tiny]
       & \p^1\times\p^1 \ar[r,dashed,"\varphi^{-1}"] & Y\ar[r,"F"] & X. \\
       \p^1\ar[ur,"\sigma''"]\ar[urrr,"\sigma",bend right=10]
    \end{tikzcd}
  \end{center}
  Explicitly, if $\sigma''\colon t\mapsto ([1:P(t)],[1:t])$ for some $P(t)\in\k[t]$, then
  \begin{align*}
    \sigma\colon t&\mapsto ([1+t^3:1+t^2P+P^2:t+P+tP^2],[1:t^2]) & \text{(type $4$),}\\
    \sigma\colon t&\mapsto ([t+t^2:1+P+P^2:t+t^2P+tP^2],[1:t^2]) & \text{(type $2+2$).}
  \end{align*}
  Note that if the degree of $P(t)$ is at least $1$, then $C=\sigma(\p^1)\subset X\subset\p^2\times\p^1$ has degree $2\deg(P)+1$ in $x,y,z$.

  We describe now the double section $C=\sigma(\p^1)\subset X$ in the case when $P(t)=a\in\k$ is a constant polynomial. Let $\Gamma=\rho(C)\subset \p^2$.
  (Recall that $\rho\colon X\to\p^2$ is the blow-up of four points $\PPPP$ as in Notation~\ref{not:EquationsX}.)

  For type $4$, the curve $\Gamma$ is the cubic given by
  \[(a^2+1)xy^2+a^2x^2z+y^3+z^3=0,\]
  which contains $\PPPP$ and has a cusp at $[1:0:a]$.

  For type $2+2$, the curve $\Gamma$ is the conic
  \[a^2x^2+(a^2+1)xy+yz+z^2=0,\]
  which contains the orbit $\{[1:1:\omega],[1:1:\omega+1]\}\subset\PPPP$ and is tangent to $y=0$ (which is the line passing through the other orbit in $\PPPP$) at the point $[1:0:a]$.
\end{remark}

\begin{lemma}[Lagrange on $\PPP^1\times\PPP^1$]\label{lemma:LagrangePolynomial}
  Let $\kk$ be a perfect field.
  Let $\{p_1,\ldots,p_n\}$ in $\PPP^1\times\PPP^1$ be $\Gal(\bar\kk/\kk)$-orbit of $n\geq2$ points such that no two points lie on the same fiber (with respect to the projection onto the second $\PPP^1$).
  Then either all points are of the form $([0:1],[1:t_i])$, or there exists a polynomial $L(t)\in \kk[t]$ of degree at most $n-1$ such that $p_i=([1:L(t_i)],[1:t_i])$ for $i=1,2,\ldots,n$ and some $t_i\in\bar\k$.
\end{lemma}

\begin{proof}
  We write $p_i=([u_i:v_i],[s_i:t_i])$, with $[s_i:t_i]\neq[s_j:t_j]$ for $j\neq i$ (since they all lie on different fibers).
  If one of the $s_i$ is equal to zero, then all points lie on one fiber, a contradiction to $n\geq2$.
  So we can assume that $s_i=1$ for all $i$, and $t_j\neq t_i$ for all $j\neq i$.
  If $u_i=0$ for one $i$ -- and hence for all -- we are done.
  So we assume now that $u_i\neq 0$ and write $p_i=([1:v_i],[1:t_i])$.
  Since the $p_i$ form an orbit, any $\sigma\in\Gal(\bar \kk/\kk)$ induces an action on the indices:
  We write $\sigma(i)$ for the index $j$ such that $\sigma(p_i)=p_j$.
  Hence, $\sigma(p_i)=p_j$ if and only if $(\sigma(v_i),\sigma(t_i))=(v_j,t_j)=(v_{\sigma(i)},t_{\sigma(i)})$.
  Consider now the Lagrange polynomial
  \[
    L(t) = \sum v_jl_j(t),
  \]
  where $l_j=\prod_{1\leq k\leq n, k\neq j}\frac{t-t_k}{t_j-t_k}\in\overline{\kk}[t]$.
  So $L\in\overline{\kk}[t]$ is a polynomial of degree at most $n-1$.
  In fact, $L$ is defined over $\kk$:
  For all $\sigma\in\Gal(\bar\kk/\kk)$ we have $\sigma(l_j)=l_{\sigma(j)}$, which implies that $\sigma(v_jl_j)=v_{\sigma(j)}l_{\sigma(j)}$.
  Therefore, $\sigma(L)=L$, implying that $L$ is defined over $\kk$, and $p_i=([1:L(t_i)],[1:t_i])$.
\end{proof}

\begin{lemma}\label{lemma:DoubleSectionThroughLargeOrbit}
  Let $\k=\F_{2^N}$ with $N$ odd.
  Let $X$ be of type $4$ or $2+2$ as in Notation~\ref{not:EquationsX}.
  Let $\{p_1,\ldots,p_n\}\subset X$ be a Galois orbit of size $n\geq2$ such that all points lie on different smooth fibers.
  Then there exists a double section $C\subset X$ containing the orbit such that either $\Gamma=\rho(C)\subset \PPP^2$ has degree $2d+1\leq 2n-1$ for some $d\geq1$, or $\Gamma$ is the strange line, or $\Gamma$ is a cubic and $X$ is of type $4$, or $\Gamma$ is a conic and $X$ is of type $2+2$.
\end{lemma}

\begin{proof}
  The map $F\from Y\to X$ given by $(p,[s:t])\mapsto(p,[s^2:t^2])$ as in the proof of Lemma~\ref{lemma:SectionVsDoubleSection} is bijective on the $\overline{\k}$-points, and it maps distinct fibers onto distinct fibers.
  So the points ${F}^{-1}(p_i)$ form an orbit on $Y$, all on distinct fibers that are smooth and not defined over $\k$ (using $n\geq2$).
  So $\varphi\colon Y\dashto\p^1\times\p^1$ from Lemma~\ref{lemma:FiberProduct} restricts to an isomorphism on these fibers, providing points $q_i=\varphi\comp F^{-1}(p_i)\in\PPP^1\times\PPP^1$ that form one orbit, all on distinct fibers.

  By Lemma~\ref{lemma:LagrangePolynomial}, either the $q_i$ are of the form $([0:1],[1:t_i])$ or there exists a polynomial $P\in\k[t]$ of degree $d\leq n-1$ such that $q_i=([1:P(t_i)],[1:t_i])$.
  In the first case, the $p_i$ lie on the double section $x=0$ in $X$ (Lemma~\ref{lemma:FiberProduct}).
  In the second case, we have already seen in Remark~\ref{rem:ExplicitDoubleSection} that either the degree $d$ of $P(t)$ is at least one and the degree of $\Gamma$ is $2d+1\leq2n-1$, or $P(t)=a\in\k$ and either $X$ is of type $4$ and $\Gamma$ is a cubic, or $X$ is of type $2$ and $\Gamma$ is a conic.
\end{proof}

\begin{proposition}\label{proposition:GeneratorsOfBirationalMap4}
  Let $\k=\F_{2^N}$ with $N$ odd. Let $X$ be of type $4$ or $2+2$ as in Notation~\ref{not:EquationsX}.
  Let $\Omega\subset X$ be a set of points that forms one Galois orbit of size $n\geq1$ such that there is an elementary transformation $\varepsilon_{\Omega}$ centered at $\Omega$.
  Then, $\varepsilon_{\Omega}$ can be written as a product of elements in $\Aut_\k(X)$, elementary transformations centered at a $\k$-rational point, and elementary transformations centered at a set of points forming Galois orbits of size $\leq n$ lying on the double section $x=0$.
\end{proposition}

\begin{proof}
  We do induction on $n$: For $n=1$ there is nothing to prove, so we can assume $n\geq 2$ and the induction assumption is that for all orbits $\QQQQ\subset X$ of size strictly less than $n$ we can decompose elementary transformations centered at $\QQQQ$ into birational maps as in the statement.

  Since $n\geq 2$ and $\varepsilon_{\Omega}$ is defined, all points of $\Omega$ lie on different smooth fibers.
  By Lemma~\ref{lemma:DoubleSectionThroughLargeOrbit} there exists a double section $C\subset X$ containing $\Omega$ such that either $\Gamma=\rho(C)\subset \p^2$ has degree $2d+1\leq 2n-1$ for some $d\geq 1$, or $\Gamma$ is the strange line, or $d\in\{2,3\}$.
  If $\Gamma$ is the strange line we are done.
  If $d\in\{2,3\}$ there exists an elementary transformation $\varepsilon_r$ centered at a $\k$-rational point $r\in X(\k)$ that sends $C$ onto $x=0$ by Lemma~\ref{lemma:SmoothDoubleSection}.
  Set $\Omega'=\varepsilon_r(\Omega)$, which is one Galois orbit of size $n$ that lies on $x=0$. Note that an elementary transformation $\varepsilon_{\Omega'}$ is defined.
  Set $s=\varepsilon_{\Omega'}(\Bas(\varepsilon_{r}^{-1}))\in X(\k)$, and one can choose an elementary transformation $\varepsilon_s$ centered at $s$ such that \[\varepsilon_{\Omega} = \varepsilon_s\circ\varepsilon_{\Omega'}\circ\varepsilon_r.\]
  This proves in particular that if $\Omega$ lies on a smooth double section, then there exists a decomposition of $\varepsilon_\Omega$ into elementary transformations as in the statement.

  We consider now the case when $\Gamma$ is of odd degree $2d+1\leq2n-1$. Writing $m$ for the number of singular points of $C$ as in Lemma~\ref{lemma:DoubleSectionNumbers} we find that $2d+1=2m+1$, and hence $d=m$. That is, $C$ has $d\leq n-1$ singular points, and they form Galois orbits of size at most $n-1$. In particular, none of the points lies on the same fiber as a point of $\Omega$.
  By Lemma~\ref{lem:ResolveAk}, there exists a sequence $\varphi=\varepsilon_{\QQQQ_l}\circ\cdots\circ\varepsilon_{\QQQQ_1}$ of elementary transformations centered at Galois orbits of size at most $n-1$ that resolves the singularities of $C$. By the induction hypothesis, there is a decomposition of each $\varepsilon_{\QQQQ_i}$ into elementary transformations as in the statement.
  So $\Omega'=\varphi(\Omega)$ forms a Galois orbit of size $n$ on the smooth double section $C'=\varphi_*(C)\subset X$. By the previous step, there is a decomposition of $\varepsilon_{\Omega'}$ into elementary transformations as in the statement.
  Finally, the base points of $\varepsilon_{\Omega'}\circ\varphi\circ\varepsilon_{\Omega}$ can be resolved with $l$ elementary transformations centered at Galois orbits of the same size as the ones in $\varphi$, that is, they are all of size at most $n-1$.
  By induction, we obtain a decomposition of $\varepsilon_{\Omega}$ as in the statement.
\end{proof}

\begin{lemma}\label{lem:F2ExactSequence}
  Let $\k=\F_{2^N}$ with $N$ odd. For $i\in\{2,4\}$ let $\pi\colon\p^2\dashrightarrow\p^1$ as in Notation~\ref{not:EquationsX} and consider the exact sequence
  \[1\to\Bir_\k(\p^2/\pi_i)\to\Bir_\k(\p^2,\pi_i)=\JJ_i\to\Aut_\k^{\pi_i}(\p^1)\to1.\]
  Then $\Aut_\k^{\pi_4}\simeq\Z/2\Z$ and $\Aut_\k^{\pi_2}(\p^1)\supset \Z/2\Z$.

  Moreover, if $\k=\F_2$, then the exact sequence involving $\JJ_2$ splits with $\Aut_\k^{\pi_2}(\p^1)=\Z/2\Z$, and the exact sequence involving $\JJ_4$ does not split.
\end{lemma}

\begin{proof}
  Case $\JJ_4$: By Lemma~\ref{lem:SingularFibersOntoSingularFibers}, any element in $\JJ_4$ permutes the three singular fibers.
  In $X$, there is one singular fiber over a point in $\p^1(\k)$, which therefore has to be fixed by any element of $\JJ_4$, and two over two points in $\p^1(\k_2)$ forming one Galois-orbit. Since $\alpha\in\Aut_\k(\p^1)$ is uniquely defined by the action on three points, either $\alpha$ is the identity, or $\alpha$ fixes the $\p^1(\k)$-point and exchanges the two points in the Galois orbit of size $2$, that is, $\alpha([s:t])=[t:s]$ in the coordinates of Notation~\ref{not:EquationsX}.
  In fact, the following map $\beta\in\Aut_\k(\p^2)\cap\JJ_4$ realizes the latter:
   \[[x:y:z]\mapsto[x:z:x+y].\]
  In fact, any element in $\alpha_4\in\Aut_\k^{\pi_4}(\p^1)$ is realized by an element in $\Aut_\k(\p^2)\cap\JJ_4$ by \cite[Lemma~3.20]{LamySchneider}.
  Checking all $168$ elements of $\Aut_{\F_2}(\p^2)$, one observes that $\beta$ and $\beta^{-1}$, which are elements of order $4$, are the only elements that lie in $\JJ_4$ and that induce $\alpha$. In particular, the sequence does not split.

  Case $\JJ_2$: Note that $\gamma\in\Aut_\k(\p^2)\cap\JJ_2$ given by
    \[\gamma\colon[x:y:z]\mapsto[x:y:y+z]\]
    realizes the involution $\alpha\colon[s:t]\mapsto[s:s+t]$.
  If $\k=\F_2$, the three singular fibers are exactly the three fibers over $\p^1(\F_2)$. By Lemma~\ref{lem:SingularFibersOntoSingularFibers}, any element of $\JJ_2$ preserves the set of the two fibers without $\F_2$-components, and so the third singular fiber has to be sent onto itself.
  Moreover, since $\gamma$ is of order $2$, the sequence splits.
\end{proof}

We are ready for the proof of Proposition~\ref{prop:ParametrizationFiberingType}.

\begin{proof}[Proof of Proposition~\ref{prop:ParametrizationFiberingType}]
  Up to conjugation by an element of $\Aut_\k(\p^2)$, we can assume that $\pi$ is $\pi_i$ as in Notation~\ref{not:EquationsX} with $i\in\{2,4\}$ (Lemma~\ref{lemma:SameFourPoints}).

  For \ref{it:ParametrizationFiberingType--fix}, assume that $\varphi$ fixes $\pi$ and it fixes $x=0$. Then $\varphi$ is of the form $\chi_{i,[u:v]}$ as in Proposition~\ref{proposition:ExplicitFamily}.

  For \ref{it:ParametrizationFiberingType--F2}, assume that $\k=\F_2$. Let $\rho\colon X\to\p^2$ be the blow-up of $\p^2$ at the base point of $\pi_i$,
  giving a surface $X$ of type $4$ (case $i=4$), respectively $2+2$ (case $i=2$).
  In the case $i=4$, we have $\varphi=\rho\circ\psi\rho^{-1}$, and $\psi\colon X\dashrightarrow X$ preserves the fibration and admits a decomposition into elementary transformations, and elements in $\Aut_\k(X)$ by Lemma~\ref{lem:DecompositionElementaryTransformations}.
  The same lemma says that in the case $i=2$, also elements conjugate to $\Aut_\k(Y)$ have to be considered, where $Y\in\Cl_6$ is a Mori conic bundle dominated by $X$. However, as seen in Lemma~\ref{lem:F2ExactSequence}, in the case of $\k=\F_2$, these are conjugate by $\rho$ to elements in $\Aut_\k(\p^2)$.

  By Proposition~\ref{proposition:GeneratorsOfBirationalMap4}, the elementary transformations admit a decomposition into
  \begin{enumerate}
    \item elements in $\Aut_\k(X)$ (which are conjugated by $\rho$ to elements in $\Aut_\k(\p^2)$),
    \item elementary transformations centered at a $\k$-rational point (which are conjugated by $\rho$ to an element in $\langle \Aut_\k(\p^2), \JJ_1\rangle$ by Lemma~\ref{lem:ElementaryTransformationOneBp}),
    \item and elementary transformations centered at a Galois orbit of points on the double section $x=0$, which in particular fix the double section $x=0$ (Lemma~\ref{lem:DoubleSectionOntoDoubleSection}), and can be chosen to be fixing the fibration, implying that they are conjugate by $\rho$ to $\chi_{i,[u:v]}$ as in Proposition~\ref{proposition:ExplicitFamily}.
  \end{enumerate}
  This proves the statement.
\end{proof}

\section{Birational Galois descent and minimal del Pezzo surfaces}\label{sec:GaloisDescent}

Let $\k$ be a perfect field and $X,Y$ two smooth projective surfaces defined over $\k$. Given a finite Galois extension $L/\k$, the Galois group $\Gal(L/\k)$ acts on $\Bir_L(X,Y)$ via the action of $g\in\Gal(L/\k)$ on the coefficients of the defining polynomials of $f\in\Bir_L(X,Y)$; denote this action by $g(f)\in\Bir_L(X,Y)$. Hence, $g(f(x))=g(f)(g(x))$ and $g(f)(x)=g(f(g^{-1}(x)))$ for $x\in X$.
Moreover, if $C\subset X$ is a subset defined over $L$ that is contracted by $f$, then $g(C)$ is contracted by $g(f)$ because $g(f)(g(C))= g(f(g^{-1}(g(C))))=g(f(C))$.

\begin{lemma}\label{lem:kStructure}
  Let $\k$ be a perfect field.
  Let $X$ and $Y$ be two smooth projective surfaces defined over $\k$.
  Let $\rho\colon X_L\to Y_L$ be a birational morphism that is defined over a finite Galois extension $L/\k$.
  For each $g\in\Gal(L/K)$, the birational map
  $\varphi_g = g^{-1}(\rho)\circ\rho^{-1}\in\Bir_L(Y)$
  satisfies
  \[
    \rho\circ g \circ \rho^{-1} = g\circ\varphi_g.
  \]
\end{lemma}

\begin{proof}
  Let $g\in\Gal(L/K)$.
  Since $\rho\colon X_L\to Y_L$ is a birational morphism defined over $L$, so is $g^{-1}(\rho)$. Hence, $\varphi_g = g^{-1}(\rho)\circ\rho^{-1}$ is a birational map in $\Bir_L(Y)$, and makes the following diagram commute:
  \[
  \begin{tikzcd}[]
  X \ar[rr,"g"] \ar[d,"\rho"] \ar[dr,"g^{-1}(\rho)"] && X\ar[d,"\rho"] \\
  Y \ar[r,"\varphi_g",dashed]& Y \ar[r,"g"] & Y.
  \end{tikzcd}
  \]
\end{proof}

\begin{definition}
  Let $\k$ be a perfect field and $L/\k$ a finite Galois extension. Let $X,Y$ be two smooth projective surfaces defined over $\k$ with a birational morphism $\rho\colon X_L\to Y_L$.
  The \emph{birational Galois descent of $X$ to $Y$ with respect to $\rho$} consists of the data $\{(g,\varphi_g)\}_{g\in\Gal(L/\k)}$ with $\varphi_g=g^{-1}(\rho)\circ\rho^{-1}\in\Bir_L(Y)$ for each $g\in\Gal(L/K)$.

  For $\Gamma = \rho\circ \Gal(L/K)\circ\rho^{-1}=\{g\circ \varphi_g\mid g\in\Gal(L/K)\}$, we will say that $(Y,\Gamma)$ is the \emph{$\k$-structure} of $X$ given by $\rho$ over $L$, and we will call the set $B$ of base points of $\rho^{-1}$ the \textit{base points} of the $\k$-structure.
\end{definition}

If $\rho$ is an isomorphism then this coincides with the classical notion of \emph{Galois descent} (see for example \cite[Proposition~4.4.4]{poonen}).

\begin{remark}\label{rem:UniqueKStructure}
  Note that if $(Y,\Gamma)$ is a $\k$-structure of two surfaces $X$ and $X'$ over $L$ with base points $B$ and $\varphi_g=\varphi'_g$ for every $g\in\Gal(L/\k)$, then $X$ and $X'$ are isomorphic.
\end{remark}

\smallskip
We say that a set $A$ of proper points on a del Pezzo surface $X$ is in \emph{(del Pezzo) general position} if the blow-up of $X$ at these points is again a del Pezzo surface.
Furthermore, we say that $A$ is in \emph{general position with} a set $B$ of proper points on $X$ if their union is in general position and $A\cap B=\emptyset$.

Given a set $A$ of points on $X$, we are interested in the following questions:
\begin{enumerate}
  \item Is $A$ in general position?
  \item What is the $\Aut_\k(X)$-orbit of $A$?
\end{enumerate}
The following lemma transfers these questions to the $\k$-structure $(Y,\Gamma)$ of $X$.

\begin{lemma}\label{lem:BasicKStructureStuff}
  Let $X,Y$ be two smooth projective surfaces defined over a perfect field $\k$ and $\rho\colon X_L\to Y_L$ a birational morphism defined over a finite Galois extension $L/\k$, giving the $\k$-structure $(Y,\Gamma)$ of $X$. Let $A\subset  X$ be a set of points defined over $L$.
  \begin{enumerate}
    \item\label{it:BasicKStructureStuff--genpos} The points in $A$ are in general position on $X$ if and only if $\rho(A)\subset Y$ is in general position with $\Bs(\rho^{-1})$ on $Y$. In particular, if $A$ is in general position on $X$ then $\rho(A)$ does not contain a base point of the $\k$-structure.
    \item\label{it:BasicKStructureStuff--orbit} Assume that $A\cap \Exc(\rho)=\emptyset$. Then the points of $A$ form a $\Gal(\bar \k/\k)$-orbit if and only if the points in $\rho(A)$ form a $\Gamma$-orbit on $Y$.
    \item\label{it:BasicKStructureStuff--aut} Let $B\subset X$ be a set of points defined over $L$ of the same size as $A$, and assume that both are disjoint with $\Exc(\rho)$. Then $A$ and $B$ are $\Aut_\k(X)$-equivalent if and only if $\rho(A)$ and $\rho(B)$ are
    $\rho\circ\Aut_\k(X)\circ\rho^{-1}$-equivalent.
  \end{enumerate}
\end{lemma}

\begin{proof}
  These are direct consequences of the definitions.
\end{proof}

The following lemma is classical:
\begin{lemma}
  Let $X_d\in\Dl_d$ for some $d\in\{5,6,8\}$ for a perfect field $\k$. There exists a finite Galois extension $L/\k$ such that there exists an $L$-isomorphism $X_L\simto Y$, where $Y$ is one of the following:
  \begin{enumerate}
    \item If $d=8$, then $Y=\p^1_L\times\p^1_L$, and $L$ is the splitting field of the base point of $\p^2_\k\link21 X_8$.
    \item If $d=6$, then $Y=\{([x:y:z],[u:v:w])\mid xu=yv=zw\}\subset\p^2_L\times\p^2_L$ is the blow-up of $\p^2_L$ at the coordinate points $[1:0:0],[0:1:0],[0:0:1]$, and $L=L_2L_3$ where $L_3$ (respectively $L_2$) is the splitting field of the base point of $X_8\link31 X_6$ (respectively $\p^2_\k\link21 X_8$) for some $X_8\in\Dl_8$.
    \item If $d=5$, then $Y$ is the blow-up of $\p^2_L$ at $[1:0:0],[0:1:0],[0:0:1],[1:1:1]$, and $L$ is the splitting field of the base point of $\p^2_\k\link51 X_5$.
  \end{enumerate}
  We call a smallest such extension \emph{splitting field} of $X$.
\end{lemma}

 A classification of del Pezzo surfaces and their automorphisms over any perfect field can be found in \cite{Schneider-Zimmermann} (degree $6$ and $8$) and \cite{boitrel-delPezzo} (degree $5$).
 In the following, we will focus on finite fields, or more generally on quasi-finite fields.
 We will give explicit equations of the $\k$-structure of $X\in\Dl_8\cup\Dl_5\cup\Dl_6$ via birational Galois descent.

 We present here the proof of Proposition~\ref{prop:UniqueMinimaldP}, which relies on the following subsections.

\begin{proof}[Proof of Proposition~\ref{prop:UniqueMinimaldP}]
  A finite field $\k=\F_q$ has a unique extension of degree $e$, namely $\F_{q^e}$, for any integer $e$. For $d=9$, $\Dl_9=\{\p^2\}$ follows from Châtelet's theorem.
  For $d=8$, $|\Dl_8|=1$ follows for example from \cite[Lemma 3.2(3)]{Schneider-Zimmermann}.
  For $d=6$, $|\Dl_6|=1$ follows from \cite[Lemma 4.6(3), and Remark~B.3]{Schneider-Zimmermann}.
  For $d=5$, $|\Dl_5|=1$ follows from Lemma~\ref{lem:kStructureX5} below.

  The explicit $\k$-structures are described in the following subsections in Lemma~\ref{lem:kStructureX8} ($X_8$), Lemma~\ref{lem:kStructureX6} ($X_6$), and Lemma~\ref{lem:kStructureX5} ($X_5$).
\end{proof}

\subsection{Del Pezzo surface of degree $8$}

\begin{lemma}[{\cite[Lemma~3.2]{Schneider-Zimmermann}}]\label{lem:kStructureX8}
  Let $X_8\in\Dl_8$ with splitting field $L/\k$.
  Then $X_8$ has a $\k$-structure $(\p^1\times\p^1,\Gamma)$ over $L$, where $\Gamma$ is generated by
  \[(x,y)\mapsto(y^g,x^g)\]
  and $g\in\Gal(L/\k)\simeq\Z/2\Z$ is the generator of the Galois group.
\end{lemma}

\begin{lemma}[{\cite[Lemma~3.5]{Schneider-Zimmermann}}]\label{lem:X8auto}
  The automorphisms of the $\k$-structure of $X_8\in\Dl_8$ on $\p^1\times\p^1$ over $L$ are of the form $(x,y)\mapsto (Ax,A^gy)$ or $(x,y)\mapsto(Ay,A^gx)$ for $A\in\PGL_2(L)$, where $g\in\Gal(L/\k)$ is its generator.
\end{lemma}

\begin{proof}
  The group $\Aut_L(\p^1\times\p^1)$ is generated by $(x,y)\mapsto(y,x)$ (which is $\Gamma$-invariant), and maps of the form $(x,y)\mapsto (Ax, By)$ for $A,B\in\PGL_2(L)$.
  Let $g\in\Gal(L/\k)\simeq\Z/2$ be the generator.
  Note that $(y^g,x^g)\circ (Ax,By)= (B^gy^g,A^gx^g)$ if and only if $B=A^g$.
\end{proof}

\subsection{Del Pezzo surface of degree $6$}

\begin{lemma}\label{lem:kStructureX6}
  Let $X_6\in\Dl_6$ with splitting field $F/\k$ be coming from $X_8\in\Dl_8$ with splitting field $L/\k$.
  Assume that $\Gal(F/\k)\simeq\Z/3\Z$ and $\Gal(LF/\k)\simeq \Z/6\Z$.
  Then $X_6$ is a $\k$-structure on $\p^2$ over $LF$ with base points $[1:0:0]$, $[0:1:0]$, $[0:0:1]$ generated by
  \[[x:y:z]\mapsto [xz:xy:yz]^g\]
  where $g\in\Gal(LF/\k)$ generates $\Gal(LF/\k)\simeq \Gal(L/\k)\times\Gal(F/\k)$.
\end{lemma}

\begin{proof}
  Let $\chi\colon X_8\link31 X_6$ be a link with geometric base points $p_1,p_2,p_3\in X_8(F)$, ordered in such a way that $g(p_{i\mod 3})=p_{i+1\mod 3}$ for $i=1,2,3$.
  Let $\rho_{X_8}\colon(X_8)_L\simto \p^1\times\p^1_L$ be the isomorphism realising the $\k$-structure of $X_8$ from Lemma~\ref{lem:kStructureX8}.
  Hence, $(\rho_{X_8})_{LF}(p_i)=([1:a_i],[1:b_i])\in(\p^1\times\p^1)(LF)$ form a $\Gamma_{X_8}=\langle (x,y)\mapsto (y^g,x^g) \rangle$-orbit of size $3$. In particular, $g$ acts as \[g\colon a_1\mapsto b_2\mapsto a_3\mapsto b_1\mapsto a_2\mapsto b_3.\]
  Consider now $\beta\in\Aut_L(\p^1\times\p^1)$ that sends $p_1 \mapsto ([1:0],[1:0])$, $p_2 \mapsto ([0:1],[0:1])$, $p_3 \mapsto ([1:1],[1:1])$. Explicitly, $\beta^{-1}=(Mx,M'y)$, where
  \begin{align*}
    M &= \smallmat{a_3-a_2 & a_1-a_3\\ a_1(a_3-a_2) & a_2(a_1-a_3)}, & M^{-1} &= \smallmat{\frac{a_2}{a_2-a_3} & \frac{-1}{a_2-a_3}\\ \frac{a_1}{a_1-a_3} & \frac{-1}{a_1-a_3}},
  \end{align*}
  and $M'$ the same but $a_i$ replaced with $b_i$, for $i=1,2,3$.
  Writing $\varphi_{X_8}\colon (x,y)\mapsto(y,x)$, one can compute $\widehat{\varphi_{X_8}}=g^{-1}(\beta)\circ\varphi_{X_8}\circ\beta^{-1}$:\begin{align*}
    g^{-1}(\beta)\circ\varphi_{X_8}\circ\beta^{-1}(x,y) &=
    g^{-1}(\beta)(M'y,Mx)\\
    &= (g^{-1}(M^{-1})M'y, g^{-1}(M'^{-1})Mx)\\
    &=\left(\smallmat{0 & -1\\ 1 & -1}y, \smallmat{0 & -1\\ 1 & -1}x\right).
  \end{align*}
  A birational map $\eta\colon\p^1\times\p^1\link34\p^2$ as in Lemma~\ref{lemma:DescriptionOfBertiniEtc} with base points $([1:0],[1:0])$, $([0:1],[0:1])$, $([1:1],[1:1])$ can be written as
  \begin{align*}
    \eta\colon ([x_0:x_1],[y_0:y_1]) & \mapsto  [(x_0-x_1)y_0y_1:x_0y_1(y_0-y_1):x_1y_0(y_0-y_1)],\\
    \eta^{-1}\colon [x:y:z] &\mapsto ([y(x-z):z(x-y)],[x-z:x-y]).
  \end{align*}
  Finally, one can compute that $\eta\circ\widehat{\varphi_{X_8}}\circ\eta^{-1}$ is given by $[zx:yx:zy]$.
\end{proof}

\begin{remark}\label{rem:AutoOfDelPezzo6OverKbar}
  Moreover, $X_6$ is a $\k$-structure on \[S=\{([x:y:z],[u:v:w])\in\p^2\times\p^2 \mid xu=yv=zw\}\]
  over $LF$ without base points, generated by \[([x:y:z],[u:v:w])\mapsto ([v:w:u],[y:z:x])^g,\] where $g\in\Gal(LF/\k)\simeq\Z/6\Z$ is a $6$-cycle.
  The automorphisms of $S$ over a field $K$ are well known: It holds that $\Aut_K(S) = (K^*)^2\rtimes (\Sym_3\times\Z/2\Z)$, where $\Sym_3$ is the group of permutations of $x,y,z$, such as
  \[
    ([x:y:z],[u:v:w])\mapsto ([y:z:x],[v:w:u]),
  \]
  and $\Z/2\Z$ is generated by the involution
  \[
    ([x:y:z],[u:v:w])\mapsto([u:v:w],[x:y:z]),
  \]
  and $(a,b)\in(K^*)^2$ denotes the toric map
  \begin{align*}
    t_{(a,b)}\from ([x:y:z],[u:v:w])\mapsto ([x:ay:bz],[abu:bv:aw]).
  \end{align*}
\end{remark}

\begin{lemma}\label{lem:X6auto}
  Let $\k=\F_q$ be a finite field, and denote by $g\in\Gal(\F_{q^6}/\F_q)$ the Frobenius morphism.
  Let $X_6\in\Dl_6$ be a del Pezzo surface of degree $6$ over $\F_q$, given as a $\k$-structure on $\p^2$ over $\F_{q^6}$ generated by $g\circ \varphi_g$, where $\varphi_g\colon [x:y:z]\mapsto[xz:xy:yz]$.

  Then each automorphism of the $\k$-structure on $\p^2$ is of the form \[t_{(a,b)}\circ \sigma^k\circ\iota^i,\]
  where \begin{enumerate}
    \item $\sigma\colon[x:y:z]\mapsto [y:z:x]$ is a permutation of order $3$ and $k\in\{0,1,2\}$,
    \item $\iota\colon[x:y:z]\mapsto [yz:xz:xy]$ is the standard quadratic involution and $i\in\{0,1\}$,
    \item $t_{(a,b)}\colon [x:y:z]\mapsto [x:ay:bz]$ is a toric map with $a\in \F_{q^6}^*$ satisfying $a^{q^2-q+1}=1$ and $b=a^q$.
  \end{enumerate}
\end{lemma}

\begin{proof}
  Write $L=\F_{q^2}$ and $F=\F_{q^3}$, and so $LF=\F_{q^6}$.
  We want to determine the subgroup of $\Aut_{LF}(S)$ given by
  \[\{f\in (({LF})^*)^2 \rtimes (\Sym_3\times\Z/2\Z) \mid f\circ g \circ \varphi_g =g\circ\varphi_g\circ f\}.\]
  It is enough to consider the projection onto the first coordinates.
  Note that
  \[
      t_{(a,b)}^{-1}=t_{(a^{-1},b^{-1})}=\iota\comp t_{(a,b)}\comp\iota,
  \]
  that $s\circ\iota=\iota\circ s$ and $s\circ g= g\circ s$ for all $s\in\Sym_3$, and that $g\circ\iota=\iota\circ g$. Observe that $\varphi_g=\sigma\circ\iota$.

  We write $f=t\circ s \circ \iota^i\in\Bir_{LF}(\p^2)$, where $t=t_{a,b}$ is a toric map with $a,b\in {LF}^*$, $s\in\Sym_3$ a permutation, and $i\in\{0,1\}$.
  We compute
  \begin{align*}
    f\circ g\circ\varphi_g & = (t\circ s\circ\iota^i)\circ (\iota \sigma\circ g) \\
    & = (t\circ  \iota\circ s) \circ  (\iota^i\circ g\circ \sigma) \\
    &= (\iota\circ t^{-1}\circ s) \circ (g\circ\sigma\circ \iota^i),
  \end{align*}
  and
  \begin{align*}
    g\circ \varphi_g\circ f & = (\iota\circ\sigma\circ g) \circ (t\circ s \circ \iota^{i}).
  \end{align*}
  Therefore, if $f$ commutes with $g\circ\varphi_g$, then by multiplying with $\iota$ from the left and $\iota^{i}$ from the right we obtain the equation
  \begin{align*}
    t^{-1}\circ s \circ g\circ \sigma  & = \sigma\circ g \circ t\circ s.
  \end{align*}
  Note  that $g$, $t$, and $t^{-1}$ all fix $x=0$, $y=0$, and $z=0$. This implies in particular that $\sigma$ and $s$ commute. Hence, $s$ belongs to the subgroup of order $3$ in $\Sym_3$, which is generated by $\sigma$. That is, $s=\sigma^k$ for $k\in\{0,1,2\}$.

  We use this to find conditions on $t=t_{(a,b)}$ for $a,b\in LF^*$. Replacing $s=\sigma^k$ above, we find
  \begin{align*}
    t^{-1}\circ g\circ \sigma^{k+1} &= \sigma\circ g\circ t\circ \sigma^k.
  \end{align*}
  Multiplying by $\sigma^{-k}$ on both sides, the above equation gives \begin{align*}
    t^{-1}\circ g \circ \sigma &= \sigma\circ g\circ t.
  \end{align*}
  Hence, $t_{(a^{-1},b^{-1})}\circ \sigma=\sigma\circ t_{g(a),g(b)}$ and so $[1:a^{-1}:b^{-1}]=[g(a):g(b):1]$, which gives $b=g(a)$, and $\frac{ag^2(a)}{g(a)})=1$.
   Since $g\in\Gal(\F_{q^6}/\F_q)$ is the Frobenius map, we obtain the conditions $b=a^q$ and $1=a^{q^2-q+1}$.
  Hence, $(a,b)=(a,a^q)$ for $a\in\F_{q^6}$. Conversely, any such $f$ commutes with $g\circ\varphi_g$.
\end{proof}

More generally, surfaces in $\Dl_6$ are in the situation of \cite[Lemma~4.6 and Remark~B.3]{Schneider-Zimmermann}.

\subsection{Del Pezzo surface of degree $5$}

For finite fields, surfaces in $\Dl_5$ correspond to \cite[Proposition~3.28(1)]{boitrel-delPezzo}.

\begin{lemma}\label{lem:kStructureX5}
  Let $X_5\in\Dl_5$ with splitting field $L/\k$.
  Assume that $\Gal(L/\k)\simeq\Z/5\Z$.
  Then $X_5$ has a $\k$-structure $(\p^2,\Gamma)$ over $L$ with base points $[1:0:0], [0:1:0], [0:0:1], [1:1:1]$, where $\Gamma$ is generated by
    \[[x:y:z]\mapsto [xy:y(x-z):x(y-z)]^g \]
  and $g\in\Gal(L/\k)$ denotes a $5$-cycle.

  In particular, there exists a unique surface in $\Dl_5$ with splitting field $L/\k$.
\end{lemma}

\begin{proof}
  Let $p\in\p^2$ be the point of degree $5$ such that $X_5$ is obtained by the blow-up $Z\to\p^2$ of $p$ followed by the contraction $Z\to X_5$ of the strict transform of the conic $C$ going through $p$.
  One can order $p=\{p_1,\ldots,p_5\}\subset\p^2$ in such a way that $g$ acts on $p$ as $(12345)$.
  Write $L_{ij}\subset X_5$ for the strict transform of the lines through $p_i$ and $p_j$, and write $E_i\subset X_5$ for the strict transform of the exceptional divisor of $p_i$.
  Let $\rho\colon (X_5)_L\to\p^2_L$ denote the contraction of the four $(-1)$-curves $L_{i5}$ for $i=1,2,3,4$. After composing with an element of $\Aut_{L}(\p^2)$ one can assume that $L_{15},L_{25},L_{35},L_{45}$ are contracted onto the points $p_{15}=[1:0:0], p_{25}=[0:1:0], p_{35}=[0:0:1], p_{45}=[1:1:1]$, respectively.
  We write again $L_{ij}$ for their images in $\p^2$.
  Note that $C$ is a point in $\p^2$, $E_5$ is a conic, and $L_{ij}$ is the line through $p_{k5}$ and $p_{l5}$ for $\{i,j,k,l\}=\{1,2,3,4\}$. See Figure~\ref{figure:X5} for the relevant lines and points.

  Observe the action of $g$ on $X_5$:
  \begin{align*}
  L_{12} \mapsto L_{23}\mapsto L_{34}\mapsto L_{45}\mapsto L_{15}\mapsto L_{12}, \\
  L_{13}\mapsto L_{24} \mapsto L_{35} \mapsto L_{14} \mapsto L_{25} \mapsto L_{13}.
  \end{align*}
Therefore, $\varphi_g=g^{-1}(\rho)\circ\rho^{-1}$ is a quadratic transformation not defined at
$p_{15}, p_{25}, p_{35}$ and contracting the lines
  \begin{align*}
    L_{14}=(x=0) &\mapsto p_{25}=[0:1:0],\\
    L_{24}=(y=0) &\mapsto p_{35}=[0:0:1],\\
    L_{34}=(z=0) &\mapsto p_{45}=[1:1:1],
  \end{align*}
  and sending $p_{45}$ onto $p_{15}$. Therefore, $\varphi_g([x:y:z])=[xy:y(x-z):x(y-z)]$.

  The last part of the statement follows with Remark~\ref{rem:UniqueKStructure}.
\end{proof}

\begin{figure}
  \caption{The contraction of {${L_{45},L_{35},L_{25},L_{15}\subset X_5}$ from $X_5$ to $\p^2$ over $L$}}
  \label{figure:X5}
  \begin{center}
    \begin{tikzpicture}[scale=0.9]
    \node [label={[label distance=0.1cm]-90:$p_{15}$}] (1) at (-2, 2) {$\bullet$};
    \node [label={[label distance=0.1cm]90:$p_{45}$}] (4) at (-2, -2) {$\bullet$};
    \node [label={[label distance=0.1cm]3:$p_{35}$}] (3) at (2, -2) {$\bullet$};
    \node [label={[label distance=0.1cm]357:$p_{25}$}] (2) at (2, 2) {$\bullet$};
    \draw [shorten >=-0.5cm, shorten <= -0.5cm] (1.center) to node[below]{$L_{34}$} (2.center);
    \draw [shorten >=-0.5cm, shorten <= -0.5cm] (1.center) to node[above left]{} (3.center);
    \node[label={[label distance=0.5cm]-65:$L_{24}$}]at (1.south){};
    \draw [shorten >=-0.5cm, shorten <= -0.5cm] (2.center) to node[right]{$L_{14}$} (3.center);
    \draw [shorten >=-0.5cm, shorten <= -0.5cm] (2.center) to node[below left]{} (4.center);
    \node [label={[label distance=0.5cm]65:$L_{13}$}] at (4.north) {};
    \draw [shorten >=-0.5cm, shorten <= -0.5cm] (3.center) to node[above]{$L_{12}$} (4.center);
    \end{tikzpicture}
  \end{center}
\end{figure}

\begin{lemma}\label{lem:X5auto}
  Let $\k=\F_q$ be a finite field. Let $X_5\in\Dl_5$ and let $(\p^2,\Gamma)$ be its $\k$-structure over $L=\F_{q^5}$ from Lemma~\ref{lem:kStructureX5}.
  The group of automorphisms of $(\p^2,\Gamma)$ is the cyclic group of order $5$ generated by
  \[[x:y:z]\mapsto [xy:y(x-z):x(y-z)].\]
\end{lemma}

\begin{proof}(See also \cite[Proposition~3.28(1)]{boitrel-delPezzo})
  It is well known that the automorphism group $\Aut_L(X_5)$ is isomorphic to $\Sym_5$:
  Taking the same notation as in the proof of Lemma~\ref{lem:kStructureX5}, the automorphisms in $\Aut_L(\p^2)$ that fix $\{p_{15},\ldots,p_{45}\}$ lift to automorphisms of $X_5$ fixing $E_5$, and the standard quadratic map $[x:y:z]\mapsto[yz:xz:xy]$ lifts to an automorphism of $X_5$ corresponding to the permutation $(15)$. See \cite[Proposition 6.3.7]{blanc06} for details.

  Hence, the automorphisms respecting the $\k$-structure are exactly those that commute with the $5$-cycle, which concludes the proof.
\end{proof}

\section{Points in general position on del Pezzo surfaces}\label{sec:PointsGeneralPosition}

\subsection{Points on $\p^2$}
\begin{lemma}\label{lem:P2pointdeg34}
  Let $\k=\F_q$ be a finite field and $d\in\{3,4\}$. Let $a_1,\ldots,a_d\in\F_{q^d}$ a $\Gal(\F_{q^d}/\F_q)$-orbit of size $4$ and consider the point $r$ of degree $d$ given by components $r_i=[1:a_i:a_i^2]\in\p^2(\F_{q^d})$. Then for any point $p\in\p^2$ of degree $d$ in general position, there exists $\alpha\in\Aut_\k(\p^2)$ such that $\alpha(p)=q$.
\end{lemma}
\begin{proof}
  This is \cite[Lemma~6.11]{Schneider-relations}, using that for finite fields all points of degree $d$ have residue field $\F_{q^d}$.
\end{proof}

\begin{lemma}\label{lem:P2pointdeg5}
	Let $\k=\F_2$. There is a unique Galois orbit of size $5$ in general position on $\p^2$, up to $\PGL_3(\F_2)$.
\end{lemma}
\begin{proof}
	Let $\PPPP$ be a point of degree $5$ in general position on $\p^2$, so there is a irreducible conic $C$ passing through $\PPPP$.
	Since $\PGL_3(\F_2)$ acts transitively on all irreducible conics and since $\Aut_{\F_2}(C)\simeq\Aut_{\F_2}(\p^1)$, it is enough to observe that $\PGL_2(\F_2)$ acts transitively on the Galois orbits $[a_i:1]$ of size $5$ in $\p^2_{\F_2}$, where $a_i$ are the roots of an irreducible polynomial of degree $5$.
	Indeed, there are exactly six irreducible polynomials in $\F_2[x]$ of degree $5$, namely
  \begin{align*}
    f_1&=x^5+x^4+x^3+x^2+1,  &f_2&=x^5+x^3+x^2+x+1,  &f_3&=x^5+x^4+x^3+x+1, \\ f_4&=x^5+x^4+x^2+x+1, & f_5&=x^5+x^3+1, &f_6&=x^5+x^2+1.
  \end{align*}
  Note that $x\mapsto {1}/{x}$ exchanges the roots of $f_1$ and $f_2$, $f_3$ and $f_4$, and $f_5$ and $f_6$.
  The automorphism $x\mapsto x+1$ exchanges $f_5$ with $f_1$, and $f_6$ with $f_3$.
\end{proof}

\begin{lemma}\label{lem:P2pointdeg68}
  Let $\k=\F_q$ be a finite field and $d\in\{6,8\}$. Let $\lambda\in\F_{q^{d/2}}$ such that it forms a $\Gal(\bk/\k)$-orbit of size $\frac{d}{2}$.
  Let $p\in\p^2_{\k}$ be a point of degree $d$ such that no three of the components $p_1,\ldots,p_d\in\p^2(\F_{q^d})$ are collinear.
  Then there exists $\alpha\in\Aut_\k(\p^2)$ such that after reordering the $p_i$ either
  \begin{enumerate}
    \item\label{it:P2pointdeg68--1} $\alpha(p_1)=[a:\lambda:1]$, or
    \item\label{it:P2pointdeg68--2} $\alpha(p_1)=[a:\lambda^2+\lambda a:1]$,
  \end{enumerate}
  where $a\in\F_{q^d}$ forms a $\Gal(\bk/\k)$-orbit of size $d$.
\end{lemma}
\begin{proof}
  Let $g\in\Gal(\F_{q^d}/\F_q)$ be a $d$-cycle, and order the $p_i$ such that $g$ acts as the permutation $(1\cdots d)$ on $p_1,\ldots,p_d$.
  Write $L_{ij}$ for the line between $p_i$ and $p_j$ in $\p^2_{\F_{q^d}}$.
  Since $d$ is even, the line $L=L_{1,1+\frac{d}{2}}$ forms a $\Gal(\bk/\k)$-orbit of lines of size $\frac{d}{2}$.
  There are two possibilities: Either the line $L$ (and hence all lines of the orbit) contains an $\F_q$-point, or none of them does.

  In the first case, by applying an element of $\Aut_\k(\p^2)$ we can assume that the $\F_q$-point is $[1:0:0]$, and hence $L$ is given by $y-\mu z=0$ where $\mu\in \F_{q^{d/2}}$ forms a $\Gal(\bk/\k)$-orbit of size $\frac{d}{2}\in\{2,3\}$. Note that there exists an element of $\Aut_\k(\p^1)$ that sends $[1:\mu]$ onto $[1:\lambda]$ (version of \cite[Lemma~6.11]{Schneider-relations} for $\p^1$).
  As the dual point of $L$ is $[0:1:\mu]$, there exists therefore an element of $\Aut_\k(\p^2)$ that sends $[0:1:\mu]$ onto $[0:1:\lambda]$. Therefore, an element of $\Aut_\k(\p^2)$ sends $L$ onto $y-\lambda z=0$, giving~\ref{it:P2pointdeg68--1}.

  In the second case, the line $L$ does not contain any $\F_q$-point. Hence, the dual point of $L$ forms an orbit of points of size $d\in\{3,4\}$ such that no three of its components are collinear. Lemma~\ref{lem:P2pointdeg34} implies that after composing with an element of $\Aut_\k(\p^2)$, the dual point of $L$ can be chosen to be $[\lambda:-1:\lambda^2]$. So the line $L$ is given by $y=\lambda x+\lambda^2 z$, giving~\ref{it:P2pointdeg68--2}.
\end{proof}

\begin{algorithm}\label{alg:P2points}
  Let $\k=\F_q$.
  The points of degree $d\in\{3,6,7,8\}$ on $\p^2$, up to $\Aut_\k(\p^2)\simeq\PGL_3(\k)$, can be determined using the following algorithm:
  \begin{description}
    \item [Step 0] List the elements of $\F_{q^d}^*$ that form a $\Gal(\F_{q^d}/\F_q)$-orbit of size $d$. Choose a set $S$ that contains exactly one element per orbit.
    \item [Step 1] List all $[a:b:1]$ with $a\in S$ and $b\in \F_{q^d}^*$. If $d\in\{6,8\}$, it is enough to consider $b\in\{\lambda_i,\lambda_i^2+\lambda_i a\}$ for a fixed orbit $\lambda_1,\ldots,\lambda_{d/2}$ of size $d/2$. Take only those points such that no three of its components are collinear.
    \item [Step 2] Take only one orbit up to the action of $\Aut_\k(\p^2)\simeq\PGL_3(\k)$.
    \item [Step 3] Check that the points are in general position.
  \end{description}
\end{algorithm}

\begin{lemma}\label{lem:F2pointsP2}
  Let $\k=\F_2$.
  Any point of degree $d\in\{3, 6,7,8\}$ on $\p^2$ is one of the following, up to $\Aut_\k(\p^2)$:
  \begin{enumerate}
    \item If $d=3$, there is a multiplicative generator $a\in\F_{q^3}$ such that $p=[a:a^2:1]$.
    \item If $d=6$, there is a multiplicative generator $a\in{\F_{2^6}}^*$ such that
    \[p=[a^k:a^{9}:1], \, k\in\{2,12\}.\]
    \item If $d=7$, there is a multiplicative generator $a\in{\F_{2^7}}^*$ such that
    \[p=[a:a^k:1], \, k\in\{5, 9, 10, 11, 17,18, 22, 24, 26, 39\}.\]
    \item If $d=8$, ther is a multiplicative generator $a\in{\F_{2^8}}^*$ such that
    \[p = [a^k :a^{17} :1]\]
    and $k\in\{1, 3, 5, 9, 10, 11, 13, 22, 26, 39, 47, 58\}$, or
    \[p=[a^k :a^{2\cdot 17}+a^{17k} :1]\]
    and $k\in\{1, 5, 6, 7, 9, 10, 11, 13, 14, 15, 18, 19, 21, 22, 23, 25, 26, 27, 35, 38, 41, 42, 43, 45, 46, 54\}$.
  \end{enumerate}
\end{lemma}

\begin{proof}
  We have implemented Algorithm~\ref{alg:P2points} for $\k=\F_2$. We have $|\Aut_\k(\p^2)|=168$. In each step, the obtained lists have the following sizes:
  \begin{center}
    \begin{tabular}{M || M | M | M | M } 
      d & \text{Step }0 & \text{Step }1 & \text{Step }2& \text{Step }3 \\
      \hline
      3 & 2 & 8 & 1 & 1 \\
      6 & 9 & 43 & 4 & 2 \\
      7 & 18 & 2184 & 13 & 10 \\
      8 & 30 & 236 & 47 & 38. \\
    \end{tabular}
  \end{center}
  Note that the case $d=3$ was also dealt with in Lemma~\ref{lem:P2pointdeg34}.
\end{proof}

\begin{remark}
  For $\k=\F_3$ on $\p^2$, we have $|\Aut_\k(\p^2)|=5616$ and find the following sizes in our algorithm. (For $d\in\{7,8\}$, the author's computer cannot finish the computations.)
  \begin{center}
    \begin{tabular}{M || M | M | M | M } 
      d & \text{Step }0 & \text{Step }1 & \text{Step }2& \text{Step }3 \\
      \hline
      3 & 8 & 144 & 1 & 1 \\
      6 & 116 & 627 & 18 & 11 \\
      7 & 312 & ? & ? & ? \\
      8 & 810 & 6462 & ? & ?. \\
    \end{tabular}
  \end{center}
\end{remark}

\subsection{Points on $X_8$}

\begin{lemma}\label{lem:X8points}
  Let $X_8\in\Dl_8$ be a del Pezzo surface of degree $8$ over a finite field $\k=\F_q$, and consider its $\k$-structure $(\p^1\times\p^1,\Gamma)$ over $L=\F_{q^2}$ from Lemma~\ref{lem:kStructureX8}.
  Let $p\in X_8$ be a point of degree $d\geq 3$ in general position and set $n=\lcm(2,d)\in\{d,2d\}$.
  Then, $p$ corresponds on $\p^1\times\p^1$ to a $\Gamma$-orbit of $([a:1],[b:1])\in\p^1\times\p^1$ with $a,b\in\F_{q^n}^*$. Moreover, if $d$ is odd, then $b=a^{q^d}$.
\end{lemma}
\begin{proof}
  Let $g$ be the Frobenius map.
  If $d$ is even, then for $(x,y)\in\p^1\times\p^1$
  \[(g\circ\varphi_g)^d(x,y) =(g^d(x), g^d(y)),\]
  and if $d$ is odd, then
  \[(g\circ\varphi_g)^d(x,y) =(g^d(y),g^d(x)).\]
  In particular, if the $\Gamma$-orbit of $([1:0],[y_0:y_1])$ is of size $\geq 3$ then it is not in general position since it contains two components on $[1:0]\times\p^1$.
  Hence, $p$ corresponds to the $\Gamma$-orbit of size $d$ of a point $([a:1],[b:1])$ with $a,b\in\F_{q^n}^*$.
  Moreover, if $d$ is odd, then $b=g^d(a)=a^{q^d}$.
\end{proof}

\begin{algorithm}\label{alg:X8points}
  Let $\k=\F_q$ and $X_8\in\Dl_8$ with $\k$-structure $(\p^1\times\p^1,\Gamma)$ from Lemma~\ref{lem:kStructureX8}.
  The points of degree $d\in\{4,6,7\}$ on $X_8$, up to $\Aut_\k(X_8)$, can be determined as points on $(\p^1\times\p^1,\Gamma)$ using the following algorithm, where $n=\lcm(2,d)$:
  \begin{description}
    \item [Step 0] List all elements in $\F_{q^n}$ that form a $\Gal(\F_{q^n}/\F_q)$-orbit of size $n=\lcm(2,d)$, and choose a set $S$ that consists of exactly one element per orbit.
    \item[Step 1] By Lemma~\ref{lem:X8points}, it is enough to consider points of the form $([x:1],[y:1])$ for $x,y\in\F_{q^n}^*$. Moreover, if $d$ is odd, then $y=x^{q^d}$. Up to applying the involution of $\p^1\times\p^1$ that exchanges the two rulings (which corresponds to an element of $\Aut_\k(X_8)$ by Lemma~\ref{lem:X8auto}), one can assume that $x\in S$.
    Create a list of such $\Gamma$-orbits.
    \item [Step 2] Take one $\Gamma$-orbit up to the action of $\Aut_\k(X_8)$, seen on $(\p^1\times\p^1,\Gamma)$ via Lemma~\ref{lem:X8auto}.
    \item [Step 3] Check which of the $\Gamma$-orbits are in general position, using Lemma~\ref{lemma:MatrixForComputation}.
  \end{description}
\end{algorithm}

\begin{lemma}\label{lem:F2pointsX8}
  Let $\k=\F_2$ and $X_8\in\Dl_8$ with $\k$-structure $(\p^1\times\p^1,\Gamma)$ from Lemma~\ref{lem:kStructureX8}.
  Any point of degree $d\in\{4,6,7\}$ on $X_8$, up to $\Aut_\k(X_8)$, corresponds to one of the following points $p$ on $(\p^1\times\p^1,\Gamma)$:
  \begin{enumerate}
    \item There is no point of degree $4$.
    \item If $d=6$, there is a multiplicative generator $a\in\F_{q^6}^*$ such that
    \[p=([a:1], [a^k:1]), \, k\in\{3,5,6,7,13\}.\]
    \item If $d=7$, there is a multiplicative generator $a\in\F_{q^{14}}^*$ such that
    \[p=([a^k:1],[a^{128k}:1]),\]
    where $k\in\{1, 3, 5, 7, 9, 11, 13, 15, 17, 21, 25, 29, 33, 37, 47, 61, 87, 133\}$.
  \end{enumerate}
\end{lemma}

\begin{proof}
 We have implemented Algorithm~\ref{alg:X8points} for $\k=\F_2$. We have $|\Aut_{\F_2}(X_8)|=$ In each step, the obtained lists have the following sizes:
 \begin{center}
   \begin{tabular}{M || M | M | M | M } 
     d & \text{Step }0 & \text{Step }1 & \text{Step }2& \text{Step }3 \\
     \hline
     4 & 3 & 36 & 2 & 0 \\
     6 & 9 & 477 & 8 & 5 \\
     7 & 1161 & 1161 & 21 & 18. \\
   \end{tabular}
 \end{center}
\end{proof}

\begin{remark}
  For $\k=\F_3$, we have $|\Aut_{\F_3}(X_8)|=1440$ and we find the following sizes in our algorithm.
  \begin{center}
    \begin{tabular}{M || M | M | M | M } 
      d & \text{Step }0 & \text{Step }1 & \text{Step }2& \text{Step }3 \\
      \hline
      4 & 18 & 1296 & 4 & 1 \\
      6 & 116 & 80620 & 73 & 63 \\
      7 & ? & ? & ? & ?. \\
    \end{tabular}
  \end{center}
\end{remark}

\subsection{Points on $X_5$}
\begin{lemma}\label{lem:X5points}
  Let $X_5\in\Dl_5$ over a finite field $\k=\F_q$, and consider its $\k$-structure $(\p^2,\Gamma)$ over $L=\F_{q^5}$ as in Lemma~\ref{lem:kStructureX5}.
  Let $p\in X_5$ be a point of degree $d\in\{3,4\}$ in general position and set $n=\lcm(5,d)$.
  Then, $p$ corresponds on $\p^2$ to a $\Gamma$-orbit of $[1:a:b]$ with $a,b\in\F_{q^n}^*$ such that the following hold:
  \begin{enumerate}
    \item If $d=3$, then $a\in \F_{q^{15}}$ is a root of $x^{q^{6}+q^3+1} -x^{q^{6}} -x +1\in\F_{q}[x]$, and $b=a^{q^3+1}$.
    \item If $d=4$, then $a\in\F_{q^{20}}$ is a root of $x^{q^{8}+1}+x^{q^4}-1\in\F_{q}[x]$ and $b=1-a^{q^4}$.
  \end{enumerate}
\end{lemma}

\begin{proof}
  Write $p_1,\ldots,p_d\in\p^2(\F_{q^n})$ for the points in $\p^2$ corresponding to the components of $p$. Since $p$ is in general position, none of the $p_i$ is collinear with two of the four coordinate points. So we can write $p_1=[1:a:b]$ with $a,b\in\F_{q^n}\setminus\{0,1\}$.
  AS the Frobenius map $g$ and $\varphi_g$ commute, $(g\circ\varphi_g)^d=\varphi_g^d\circ g^d$.
  We write the first iterations of $\varphi_g$:
  \begin{align*}
    \varphi_g([x:y:z]) &= [xy:y(x-z):x(y-z)],\\
    \varphi_g^2([x:y:z]) &= [y(x-z):z(x-z):z(x-y)],\\
    \varphi_g^3([x:y:z]) &= [y(x-z):x(y-z):y(y-z)],\\
    \varphi_g^4([x:y:z]) &= [x(x-z):x(x-y):(x-z)(x-y)],
  \end{align*}
  and observe that $\varphi_g^5=\id$.
  Hence, if $d=4$ we find that the equality $[1:a:b]= \varphi_g^4\circ g^d([1:a:b])$ gives
  \[[1:a:b] = [1:\frac{1-a^{q^d}}{1-b^{q^d}}:1-a^{q^d}],\]
  hence $b=1-a^{q^d}$ and $a=\frac{1-a^{q^d}}{1-(1-a^{q^d})^{q^d}}=\frac{1-a^{q^d}}{a^{q^{2d}}}$, which implies that $a\in\F_{q^{20}}$ is a root of $x^{q^{2d}+1}+x^{q^d}-1$.

  If $d=3$, we get \[[1:a:b] = [1:\frac{(a-b)^{q^d}}{a^{q^d}(1-b^{q^d})}:\frac{(a-b)^{q^d}}{1-b^{q^d}}].\]
  This gives two equations
  \begin{align*}
    a^{q^d+1}(1-b^{q^d}) &=(a-b)^{q^d},\\
    b(1-b^{q^d}) &= (a-b)^{q^d}.
  \end{align*}
  As $b\neq 1$, we have $b^{q^d}\neq 1$ and so this gives $b=a^{q^d+1}$ and $a^{q^d+1}(1-(a^{q^d+1})^{q^d})=(a-a^{q^d+1})^{q^d}$. Hence
  \[a^{q^{2d}+q^d+1} -a^{q^{2d}} -a +1 =0.\]
\end{proof}

\begin{algorithm}\label{alg:X5points}
  Let $\k=\F_q$ and $X_5\in\Dl_5$ with $\k$-structure $(\p^2,\Gamma)$ from Lemma~\ref{lem:kStructureX5}.
  The points of degree $d\in\{3,4\}$ on $X_5$, up to $\Aut_\k(X_5)$, can be determined as points on $(\p^2,\Gamma)$ using the following algorithm, where $n=\lcm(5,d)$:
  \begin{description}
    \item [Step 0] Determine the elements $a,b\in \F_{q^n}$ that satisfy the conditions of Lemma~\ref{lem:X5points} and check which of the points $[1:a:b]$ give a $\Gamma$-orbit of size $d$.
    \item [Step 1] Take only one $\Gamma$-orbit up to the action of $\Aut_\k(X_5)$ via Lemma~\ref{lem:X5auto}.
    \item [Step 2] Check which of the $\Gamma$-orbits are in general position, using Lemma~\ref{lemma:MatrixForComputation}.
  \end{description}
\end{algorithm}

\begin{lemma}\label{lem:F2pointsX5}
  Let $\k=\F_2$ and $X_5\in\Dl_5$ with $\k$-structure $(\p^2,\Gamma)$ from Lemma~\ref{lem:kStructureX5}.
  Any point of degree $d\in\{3,4\}$ on $X_5$, up to $\Aut_\k(X_5)$, corresponds to one of the following points $p$ on $(\p^2,\Gamma)$:
  \begin{enumerate}
    \item If $d=3$, there is a multiplicative generator $a\in\F_{2^{15}}^*$ such that
    \[p=[1:a^k:a^{9k}], \, k\in\{281, 551, 635, 2867\}.\]
    \item If $d=4$, there is a multiplicative generator $a\in\F_{2^{20}}^*$ such that
    \[p=[1:a^k:1-a^{16k}],\]
    where $k\in\{121, 10293, 17789, 18725, 40151, 40331, 43157, 50865, 77161, 169277, 211821, 216373\}$.
  \end{enumerate}
\end{lemma}

\begin{proof}
 We have implemented Algorithm~\ref{alg:X5points} for $\k=\F_2$. In each step, the obtained lists have the following sizes:
 \begin{center}
   \begin{tabular}{M || M | M | M } 
     d & \text{Step }0 & \text{Step }1 & \text{Step }2 \\
     \hline
     3 & 60 & 4 & 4  \\
     4 & 240 & 12 & 12. \\
   \end{tabular}
 \end{center}
\end{proof}

\begin{remark}
  For $\k=\F_3$ on $X_5$, we find the following sizes:
  \begin{center}
    \begin{tabular}{M || M | M | M | M } 
      d & \text{Step }0 & \text{Step }1 & \text{Step }2 \\
      \hline
      3 & 720 & 48&  48 \\
      4 & 6480 & 324 & 324. \\
    \end{tabular}
  \end{center}
\end{remark}

\subsection{Points on $X_6$}

\begin{lemma}\label{lem:X6points}
  Let $X_6\in\Dl_6$ a surface over a finite field $\k=\F_q$, with $\k$-structure on $\p^2$ as in Lemma~\ref{lem:kStructureX6}. Let $d\in\{2,3,4,5\}$ and set $n=\lcm(6,d)$.
  Let $p\in X_6$ be a point of degree $d$ in general position.
  Then $p$ corresponds to a $\Gamma$-orbit of $[a:b:1]$ with $a,b\in\F_{q^n}$ such that the following holds:
  \begin{enumerate}
    \item If $d=2$, then $b\in\F_{q^6}$ satisfies $b^{q^{4}+q^2+1}=1$, and $a=b^{-q^{2}}$.
    \item If $d=3$, then $a,b\in\F_{q^6}$ satisfy $a^{q^3+1}=b^{q^3+1}=1$.
    \item If $d=4$, then $a\in\F_{q^{12}}$ satisfies $a^{q^{8}+q^4+1}=1$, and $b=a^{-q^4}$.
    \item If $d=5$, then $b\in\F_{q^{30}}$ satisfies $b^{q^{10}-q^5+1}=1$, and $a=b^{q^5}$.
  \end{enumerate}
\end{lemma}

\begin{proof}
	Denote by $g$ the Frobenius map, and list the first few iterates of $\varphi_g$:
  \begin{align*}
    \varphi_g &= [xz:xy:yz],&
    \varphi_g^2 &= [z:x:y],&
    \varphi_g^3 &= [yz:xz:xy],\\
    \varphi_g^4 &= [y:z:x],&
    \varphi_g^5 &= [xy:yz:xz],&
		\varphi_g^6 &=[x:y:z].
  \end{align*}
  Let $p_1\in\p^2(\F_{q^n})$ be the point corresponding to a component of $p$.
  Since $p$ is in general position, the point $p_1$ is not collinear with any of the three coordinate points. In particular, we can write $p_1=[a:b:1]$ with $a,b\in\F_{q^n}^*$.

  If $d=2$, then $[a:b:1]=[1:a^{q^d}:b^{q^d}]$ gives $a=b^{-q^d}$ and $b=(ab^{-1})^{q^d}=b^{-(q^{2d}+q^d)}$, and so $b^{q^{2d}+q^d+1}=1$.

  If $d=3$, we find $[a:b:1]=[b^{q^d}:a^{q^d}:(ab)^{q^d}]$, which gives $a^{q^d+1}=b^{q^d+1}=1$.

  If $d=4$, then $[a:b:1]=[b^{q^d}:1:a^{q^d}]$ implies $b=a^{-q^d}$ and $a=(ab^{-1}){q^d}=a^{-(q^{2d}+q^d)}$, and so $a^{q^{2d}+q^d+1}=1$.

  If $d=5$, we find $[a:b:1]=[(ab)^{q^d}:b^{q^d}:a^{q^d}]$, which implies $a=b^{q^d}$, and $b=(ba^{-1})^{q^d}=b^{q^d-q^{2d}}$, giving $b^{q^{2d}-q^d+1}=1$.
\end{proof}

\begin{remark}\label{rem:rootsOfUnity}
  Fortunately, roots of unity can be listed efficiently:
  Let $\F_{q^n}$ be a finite field and let $a\in\F_{q^n}^*$ be a multiplicative generator.
  Let $m\geq 1$ be an integer that divides $q^n-1$ (otherwise, there does not exist a $m$-th root of unity).
  Then the subgroup of elements $x\in\F_{q^n}^*$ such that $x^m=1$ is generated by $a^{\frac{q^n-1}{m}}$.
\end{remark}

\begin{algorithm}\label{alg:X6points}
  Let $\k=\F_q$ and $X_6\in\Dl_6$ with $\k$-structure $(\p^2,\Gamma)$ from Lemma~\ref{lem:kStructureX6}.
  The points of degree $d\in\{2,3,4,5\}$ on $X_6$, up to $\Aut_\k(X_6)$, can be determined as points on $(\p^2,\Gamma)$ using the following algorithm, where $n=\lcm(6,d)$:
  \begin{description}
    \item [Step 1] List the elements $[a:b:1]$ for $a,b\in\F_{q^n}$ that satisfy the condition of Lemma~\ref{lem:X6points} using Remark~\ref{rem:rootsOfUnity}. Form its $\Gamma$-orbit and check that it is of size $d$.
    \item [Step 2] Take only one $\Gamma$-orbit up to the action of $\Aut_\k(X_6)$ seen via Lemma~\ref{lem:X6auto}.
    \item [Step 3] Check that the $\Gamma$-orbit is in general position using Lemma~\ref{lemma:MatrixForComputation}.
  \end{description}
\end{algorithm}

\begin{lemma}\label{lem:F2pointsX6}
  Let $\k=\F_2$ and $X_6\in\Dl_6$ with $\k$-structure $(\p^2,\Gamma)$ from Lemma~\ref{lem:kStructureX6}.
  Any point of degree $d\in\{2,3,4,5\}$ on $X_6$, up to $\Aut_\k(X_6)$, corresponds to one of the following points $p$ on $(\p^2,\Gamma)$:
  \begin{enumerate}
    \item If $d=2$, there is a multiplicative generator $a\in\F_{2^{6}}^*$ such that
    \[p=[ a^{51}: a^3:  1].\]
    \item If $d=3$, there is a multiplicative generator $a\in\F_{2^{6}}^*$ such that
    \[p=[ 1: a^{7k}: 1],\,  k\in\{1,3\}.\]
    \item If $d=4$, there is a multiplicative generator $a\in\F_{2^{12}}^*$ such that
    \[p=[ a^{15k} : a^{-16\cdot 15k} :1 ],\,  k\in\{1, 3, 7, 9\}.\]
    \item If $d=5$, there is a multiplicative generator $a\in\F_{2^{30}}^*$ such that
    \[p=[a^{32rk}: a^{rk}:1 ],\]
    where $r=\frac{2^{30}-1}{2^{10}-2^5+1}$ and $k\in\{1, 3, 5, 7, 9,15, 17,19,23,25,29\}$.
  \end{enumerate}
\end{lemma}

\begin{proof}
 We have implemented Algorithm~\ref{alg:X6points} for $\k=\F_2$. In each step, the obtained lists have the following sizes:
 \begin{center}
   \begin{tabular}{M || M | M | M } 
     d & \text{Step }1 & \text{Step }2 & \text{Step }3 \\
     \hline
      2& 18 & 1 & 1 \\
      3& 78 & 3& 2 \\
      4& 252 & 4 & 4 \\
      5& 990 & 11 & 11.
   \end{tabular}
 \end{center}
\end{proof}

\begin{remark}
  For $\k=\F_3$ on $X_6$, we find the following sizes:
  \begin{center}
    \begin{tabular}{M || M | M | M | M } 
      d & \text{Step }1 & \text{Step }2 & \text{Step }3 \\
      \hline
      2 & 84 &2 & 2  \\
      3 & 777 & 9 & 7  \\
      4 & 6552 & 42 & 42  \\
      5 & 58800 & 280 & 280.
    \end{tabular}
  \end{center}
\end{remark}

\section{Birational maps}\label{sec:BirationalMaps}

Let $X$ be a minimal del Pezzo surface that is rational.
We say that two Sarkisov links $\chi,\chi'\colon X\dashrightarrow X$ are \emph{equivalent} if there exists $\alpha,\beta\in\Aut_\k(X)$ such that $\chi'=\beta\circ\chi\circ\alpha$.

After having seen tools to classify points in del Pezzo-general position in the previous section, we explain how to describe Sarkisov links \emph{in practice}.

\subsection{Symmetric Sarkisov links dominated by a cubic or quartic surface}

\begin{proposition}\label{prop:symmetricDescription}
  Let $\k$ be a perfect field, and $\chi\colon X\link aa X$ a symmetric Sarkisov link on a rational surface $X\in\Dl$, such that its minimal resolution $Z$ is a del Pezzo surface of Picard rank $2$ and degree $3$ or $4$.
  Let $(Y,\Gamma)$ be a $\k$-structure of $X$ over $L/\k$, given by $\rho\colon X_L\to Y_L$.
  Write $\widehat\chi=\rho\chi\rho^{-1}\colon Y_L\dashrightarrow Y_L$.
  Then, we are in one of the following cases:
  \begin{enumerate}
    \item\label{it:symmetricDescription--P2-66} If $\chi\colon\p^2\link66\p^2$ centered at $p=\{p_1,\ldots,p_6\}\subset \p^2(\bar\k)$, then $\widehat\chi=\chi$ is given by $[x:y:z]\mapsto[f_0:f_1:f_2]$ where $f_0,f_1,f_2\in\k[x,y,z]$ are a web of quintics with a singularity at $p_i$ for $i=1,\ldots,6$.
    Geometrically, $\chi$ is given by the blow-up of $p$, followed by the contraction of the six conics through five of the six geometric points.
    \item\label{it:symmetricDescription--X6} If $X=X_6$, let $b_1,b_2,b_3\in\p^2$ be the base points of the $\k$-structure.
    \begin{enumerate}
    	\item\label{it:symmetricDescription--X6-33} If $\chi\colon X_6\link 33 X_6$ centered at $p=\{p_1,p_2,p_3\}$, then $\widehat\chi$ is given by
    	$[x:y:z]\mapsto[f_0:f_1:f_2]$ where $f_0,f_1,f_2\in L[x,y,z]$ are a web of quintics with a singularity at $p_i$ and at $b_i$ for $i=1,2,3$.
    	Moreover, $\chi$ is given by the blow-up of $p$, followed by the contraction of the three $(-1)$-curves corresponding to the conics through $b_1,b_2,b_3$ and two of the three $p_i$.
    	\item\label{it:symmetricDescription--X6-22} If $\chi\colon X_6\link22 X_6$ centered at $p=\{p_1,p_2\}$, then $\widehat\chi$ is given by
    	$[x:y:z]\mapsto[f_0:f_1:f_2]$ where $f_0,f_1,f_2\in L[x,y,z]$ are a web of cubics going through $p_1$ and $p_2$ and $b_1,b_2,b_3$, with a singularity at $p_1$.
    	Moreover, $\chi$ is given by the blow-up of $p_1,p_2$, followed by the contraction of the line through $p_1$ and $p_2$, and the conic through $p_1,p_2$ and $b_1,b_2,b_3$.
    \end{enumerate}
    \item\label{it:symmetricDescription--X8-44} If $\chi \colon X_8\link44 X_8$ centered at $p=\{p_1,\ldots,p_4\}$, then $\widehat\chi$ is given by $([y_0:y_1],[z_0:z_1])\mapsto ([f_0:f_1],[g_0:g_1])$ where $f_0,f_1\in L[y,z]$ is of bidegree $(1,2)$ going through $p_i$, and $g_0,g_1\in L[y,z]$ of bidegree $(2,1)$ going through $p_i$, for $i=1,\ldots,4$.
    Moreover $\widehat\chi$ is given by the blow-up of $p$, followed by the contraction of the four diagonals going through three of the four geometric points.
  \end{enumerate}
\end{proposition}

\begin{proof}
  \ref{it:symmetricDescription--P2-66}: This is case $X=X'=\p^2$, $d=6$ of Lemma~\ref{lemma:DescriptionOfBertiniEtc}.

  \ref{it:symmetricDescription--X6}\ref{it:symmetricDescription--X6-33} This is again case $X=X'=\p^2$, $d=6$ of Lemma~\ref{lemma:DescriptionOfBertiniEtc}.

  \ref{it:symmetricDescription--X6}\ref{it:symmetricDescription--X6-22} This is case $X=X'=\p^2$, $d=5$ of Lemma~\ref{lemma:DescriptionOfBertiniEtc}, where $E_1$ is the exceptional divisor of $p_1$.

  \ref{it:symmetricDescription--X8-44}: This is case $X=X'=\p^1\times\p^1$, $d=4$ of Lemma~\ref{lemma:DescriptionOfBertiniEtc}.
\end{proof}

\begin{example}[$\p^2\link66\p^2$ over $\k=\F_2$]\label{Ex:P2-66}
  Let $p$ be the point corresponding to the Galois orbit of $[a^k:a^9:1]$ as in Lemma~\ref{lem:F2pointsP2} where the minimal polynomial of $a$ is $x^6+x^4 + x^3 + x + 1$.
  Then, the web of quintics with $p$ as singular points gives $[x:y:z]\mapsto[f_0:f_1:f_2]$, where
  \begin{align*}
    f_0&=x^5 +x^4z + x^3(y^{2} + y z) + x^2(y^{3} + y^{2} z + z^{3}) +x(y^{3} z + y^{2} z^{2} + z^{4}) + y^{4} z,\\
f_1&=x^4 y + x^2(y^{2} z + y z^{2} + z^{3}) + x(y^{4} + y^{2} z^{2} + y z^{3}) + y^{4} z,\\
f_2&= x^4 z + x^2(y^{2} z + y z^{2}) + x(y^{4} + y^{3} z + y^{2} z^{2} + z^{4}) + y^{4} z
  \end{align*}
  if $k=2$, and if $k=12$
  \begin{align*}
    f_0&=x^5 +x^3(y^{2} + y z + z^{2}) + x(y^{2} + y z + z^{2})^2+ y^{5} + y^{4} z + z^{5},\\
    f_1 &= x^4 y + x^2y(y^{2} + y z + z^{2}) + xy(y^{3} + y z^{2} +z^{3}) + y^{5} + y^{4} z + z^{5},\\
    f_2 &= x^4z + x^2z(y^{2} + y z + z^{2})+xz(y^{3} + y z^{2} + z^{3}) + y^{5} + y^{4} z + z^{5}.
  \end{align*}
  One can check that both are involutions.
\end{example}

\begin{example}[$X_6\link22 X_6$ over $\k=\F_2$]\label{Ex:X6-22}
    Let $p$ be the point corresponding to the Galois orbit of $p_1=[a^{51}:a^3:1]$ of size $2$, as in Lemma~\ref{lem:F2pointsX6}, where the minimal polynomial of $a$ is $x^6 + x^4 + x^3 + x + 1$.
    Then, the web of cubics going through $p$ and the three coordinate points with singular point $p_1$ gives $\widehat f\colon [x:y:z]\mapsto[f_0:f_1:f_2]$, where
    \begin{align*}
      f_0&= x^2( a^{24} y + a^{ 27} z) +  x (a^{ 18}y^2  + a^{ 3} y z + a^{60 } z^2 )+ yz(a^{51 } y + a^{54 } z)\\
      f_1& = x^2(a^{39}y + a^{15 } z) + x(a^{33 }  y^2  +  y z + a^{3 } z^2) + yz(a^{21 } y + a^{ 60} z) \\
      f_2 &= x^2 (y+ a^{39 } z) + x(a^{ 48} y^2  + a^{51 } y z+ a^{9 } z^2) + yz(a^{18 } y + a^{57 }  z)
    \end{align*}
    Moreover, one can check that $\widehat f$ commutes with $\Gamma$, and it is an involution.
\end{example}

\begin{example}[$X_6\link33 X_6$ over $\k=\F_2$]\label{Ex:X6-33}
    Let $p$ be the point corresponding to the Galois orbit of $p_1=[1:a^{7k}:1]$ for $k\in\{1,3\}$ of size $3$, as in Lemma~\ref{lem:F2pointsX6}, where the minimal polynomial of $a$ is $x^6 + x^4 + x^3 + x + 1$.
    Then, the web of quintics with a singularity at $p_1,p_2,p_3$ as well as at the coordinate points gives $\widehat f\colon [x:y:z]\mapsto[f_0:f_1:f_2]$ is
    in the case $k=1$ given by
    \begin{align*}
      f_0=& x^{3}(a^{35} y^{2}+  y za^{4}+  z^{2}a^{29} ) + x^{2}( y^{3}+  y^{2} za^{43}+  y z^{2}a^{47}+  z^{3}a^{33}) \\
      &+ xyz( y^{2} a^{18}+ y za+  z^{2}a^{57}) + y^{2} z^{2}(ya^{44}+  z a^{20})\\
      f_1=& x^{3}(a^{6} y^{2}+  y za^{39}+  z^{2}a^{17} ) + x^{2}( y^{3}a^{53}+  y^{2} za^{62}+  y z^{2}a^{4}+  z^{3}a^{50}) \\
      &+ xyz( y^{2} a^{16}+ y z a^{46}+ z^{2}a^{9}) + y^{2} z^{2}(ya^{14}+  z)\\
      f_2=& x^{3}(a^{11} y^{2}+  y za^{36}+  z^{2} ) + x^{2}( y^{3}a^{5}+  y^{2} za^{16}+  y z^{2}a^{58}+  z^{3}a^{56}) \\
      &+ xyz( y^{2} a^{30}+ y z a^{59}+ z^{2}a) + y^{2} z^{2}(ya^{24}+  z a^{23})
    \end{align*}
    and in the case of $k=3$, writing $b=a^{21}$ (which satisfies $b^2=b+1$) it is given by
    \begin{align*}
      f_0=&  x^{3} (b^2 y^{2}+ b  y z+ z^{2}) +
  x^{2}(b^2 y^{3}+b y^{2}z +  y z^{2}+ b^2 z^{3})+
 xyz (b y^{2} +  y z+ b^2  z^{2}) +
  y^{2} z^{2}( y + b^2 z)\\
    f_1= & b x^{3}( y^{2} + y z+ z^{2} )+
 b^2 x^{2} (y^{3}+
 y^{2} z+
  y z^{2}+
  z^{3})+
  xyz (y^{2} +
 yz+
   z^{2})+
 b y^{2} z^{2}(y+z)\\
  f_2=& x^{3}(b y^{2}+ b^2  y z+ z^{2})+
  x^{2}(  y^{3}+ b  y^{2} z+  b^2  y z^{2}+   z^{3})+
  xyz  (y^{2} + b y z+ b^2  z^{2})+
  y^{2} z^{2}(y + b z).
    \end{align*}
    Moreover, one can check that in both cases $\widehat f$ commutes with $\Gamma$, and it is an involution.
\end{example}

\begin{proposition}\label{prop:F2CubicQuarticAreInvolution}
  Let $\k=\F_2$, and let $Z$ be a del Pezzo surface of Picard rank $2$ associated to a Sarkisov link $\chi\colon X_d\link aa X_d$ with $K_Z^2=d-a\in\{3,4\}$.
  Then there exists $\alpha\in\Aut_\k(X_d)$ such that $\alpha\circ\chi$ is an involution.
\end{proposition}

\begin{proof}
  The following lists all such $\chi$:

  $\p^2\link66\p^2$: By Lemma~\ref{lem:F2pointsP2} there are exactly two such links, and they can be chosen to be involutions as in Example~\ref{Ex:P2-66}.

  $X_8\link44 X_8$: Such a link does not exist over $\k=\F_2$ by Lemma~\ref{lem:F2pointsX8}.

  $X_6\link 33 X_6$: By Lemma~\ref{lem:F2pointsX6} there are exactly two such links, and they can be chosen to be involutions as in Example~\ref{Ex:X6-33}.

  $X_6\link 22 X_6$: By Lemma~\ref{lem:F2pointsX6} there is a unique such link, and it can be chosen to be an involution as in Example~\ref{Ex:X6-22}.
\end{proof}

\begin{remark}
  The above statement is not true in general, already for $\k=\F_3$:
  Let $a\in\F_{3^6}$ be a root of $x^6-x^4+x^2-x-1$, and consider the Galois orbit of $[a^4:a^{11}:1]$ on $\p^2$ (of size $6$). Taking a web $[f_0:f_1:f_2]$ of quintics with a singularity at the six points, we can check onto which points the conics through five of the points are contracted onto. These are the base points of the inverse map.
  Then, one can check that there is no element of $\Aut_{\F_3}(\p^2)$ sending the base points of the inverse onto the base points.
\end{remark}

\subsection{Counting Sarkisov links over $\F_2$}
\begin{proof}[Proof of Theorem~\ref{thm:CountingSarkisovLinks}]
	Let $Z$ be a rational del Pezzo surface of Picard rank $2$, or in other words, a rank $2$ fibration over $\Spec\k$. By the classification of Sarkisov links, either $Z\simeq\p^1\times\p^1$, or it gives a Sarkisov link $\chi \colon X_1\link ab X_2$ as in the list (\cite{Iskovskikh96}, see also \cite[Figure~1]{LamySchneider}).
  The number of such links, up to equivalence, equals the number of points in (del Pezzo-)general position of degree $a$ on $X_1$, up to $\Aut_\k(X_1)$. This number was computed in Lemma~\ref{lem:F2pointsP2}, (respectively Lemma \ref{lem:F2pointsX8}, \ref{lem:F2pointsX6}, \ref{lem:F2pointsX5}) for $X_1=\p^2$, (respectively $X_1=X_8$, $X_1=X_6$, $X_1=X_5$).

  Moreover, if $Z$ is of degree $1$ (respectively $2$) then the associated link $\chi$ is symmetric, and it is (equivalent to) a conjugate of the Bertini (respectively Geiser) involution on $Z$ (as in \cite[Proposition~4.5]{LamySchneider}).
  If $Z$ is of degree $3$ or $4$ and induces a symmetric link, then $\chi$ is equivalent to an involution by Proposition~\ref{prop:F2CubicQuarticAreInvolution}.
  The only remaining symmetric link is $\p^2\link33\p^2$, which is a quadratic map equivalent to an involution \cite[Corollary~4.2(Q1)]{LamySchneider}.
  This shows that all symmetric Sarkisov links are equivalent to involutions. As a consequence, every such Sarkisov link determines a non-isomorphic del Pezzo surface $Z$ of rank $2$, and so the number of links equals the number of del Pezzo surfaces of rank $2$, up to isomorphism.
\end{proof}

\subsection{Generating the Cremona group $\Bir_{\F_2}(\p^2)$}\label{sec:GeneratorsOfCremonaOverF2}

\begin{remark}\label{rem:KnownFromLamySchneider}
  In \cite{LamySchneider}, the following was proved for any perfect field:
  \begin{enumerate}
    \item\label{it:KnownFromLamySchneider--BertiniGeiser} A symmetric Sarkisov link $X\dashrightarrow X$ whose minimal resolution $Z$ is a del Pezzo surface of degree $1$ (respectively $2$), then the link is conjugate to the Bertini involution on $Z$ (respectively the Geiser involution), up to post-composition with an element of $\Aut_\k(X)$ \cite[Proposition~4.5]{LamySchneider}.
    \item\label{it:KnownFromLamySchneider--comp} The following composition of Sarkisov links lie in $\langle \Aut_\k(\p^2), \JJ_1 \rangle$:
    \begin{enumerate}
      \item $\p^2\link21 X_8\link12 \p^2$ \cite[Corollary~4.2(Q2)]{LamySchneider},
      \item $\p^2\link51 X_5\link15\p^2$ \cite[Lemma~4.6(1) and Lemma~4.13]{LamySchneider},
      \item $\p^2\link21 X_8\link31 X_6\link13 X_8\link12 \p^2$ \cite[Lemma~4.7 and Lemma~4.13]{LamySchneider}.
    \end{enumerate}
    \item\label{it:KnownFromLamySchneider--loop} The composition $\p^2\link51 X_5\link25 X_8\link12\p^2$ can be written as the product of a Geiser involution on a del Pezzo surface of degree $2$ that is given by the blow-up of a point of degree $2$ and a point of degree $5$ in general position on $\p^2$, and elements in $\langle \Aut_\k(\p^2), \JJ_1 \rangle$ \cite[Lemma~4.11 and Lemma~4.19]{LamySchneider}.
  \end{enumerate}
\end{remark}

It is important to note that \cite[Lemma~4.19]{LamySchneider} mentioned in the remark above uses \cite[Lemma~A.5]{LamySchneider}, which in turn uses a lemma that was proven in an earlier version of this text. For completeness, we state it again:

\begin{lemma}\label{lemma:5plus2general}
  Let $\kk$ be a perfect field. Let $\{p_1,\ldots,p_5\}\subset \PPP^2(\bar\kk)$ be a Galois orbit of size $5$, not all five collinear, and $\{q_1,q_2\}\subset \PPP^2(\bar\kk)$ a Galois orbit of size $2$.
  Then, either all seven points lie on a conic, or the seven points are in general position.
\end{lemma}
\begin{proof}
  Let $L/\kk$ be the splitting field of degree $2$ of $q_1,q_2$, and let $F/\kk$ be the splitting field of degree $5$ of $p_1,\ldots,p_5$.
  Let $\sigma$ be the generator of $\Gal(L/\kk)$.
  Since $\Gal(F/\kk)$ is a transitive subgroup of the symmetric group $S_5$ and $5$ is prime, it contains a $5$-cycle, which we will call $\tau$. We have that $\tau(q_i)=q_i$ for $i=1,2$. Up to reordering the $p_i$, we can assume that $\tau$ acts as $(12345)$ on the $p_i$.

  Assume now that three different points $p_i,p_j,p_k$ are on a line $L_{ijk}$.
  We may assume that $\{i,j,k\}$ are either $\{1,2,3\}$ or $\{1,2,4\}$.
  In the first case, $\tau(L_{123})=L_{234}$ and $\tau^2(L_{123})=L_{345}$ and so all five points are on the same line, a contradiction.
  In the second case, $\tau(L_{124})=L_{235}$ and $\tau^2(L_{124})=L_{341}$ and again all five points are on the same line, a contradiction.

  If one of the $p_i$ is on the line through $q_1$ and $q_2$, which is a line defined over $\kk$, then all the $p_i$ are on that line, a contradiction.
  If $q_1$ lies on the line $L_{i,j}$ through $p_i$ and $p_j$, choose $1\leq n\leq 4$ such that $\tau^n(i)=j$.
  Then $q_1=\tau^n(q_1)\in \tau^n(L_{ij})=L_{j,k}$, where $k=\tau^{2n}(i)\neq i$. So $p_k$ lies on the line through $q_1$ and $p_j$, which is the line $L_{ij}$ containing also $p_i$, and so the three points $p_i,p_j,p_k$ are on a line, a contradiction.

  As no three points of the $p_i$ are collinear, there exists a unique conic through the five points, and this conic is defined over $\kk$.
  If $q_1$ lies on this conic, then $q_2$ lies on it, too. So all seven points are on a conic.
  Assume that $q_1$ and $q_2$ lie on a conic with four of the points of the Galois orbit of size $5$, say all except $p_i$, and call this conic $C_i$.
  We can assume that $i=1$.
  So $\tau(C_1)=C_2$, which is a conic containing the five points $q_1, q_2$, $p_3$, $p_4$, $p_5$, all of whom also lie on $C_1$. Hence $C_1=C_2$ and all seven points lie on a conic.
\end{proof}

\begin{lemma}\label{lemma:TwoGeiserLoops}
  Let $\k=\F_2$. There are exactly two pairs of a Galois orbit of size $5$ and one of size $2$ on $\PPP^2$ in general position, up to $\Aut_{\FFF_2}(\PPP^2)$.
\end{lemma}
\begin{proof}
  Recall that in Lemma~\ref{lem:P2pointdeg5} we have seen that there exists a unique Galois orbit $\PPPP_5$ of size $5$ in general position on $\p^2$ over $\F_2$, and let $C$ be the conic through it.  Let $\PPPP_2$ be a Galois orbit of size $2$ and let $L$ be the line through it.
  Note that on any conic over $\FFF_2$ there are exactly three $\FFF_2$ points, say $p_1$, $p_2$ and $p_3$, and on any line over $\FFF_2$ there is exactly one Galois orbit of size $2$.
  So out of the seven $\FFF_2$-lines in $\PPP^2$, exactly three are tangent to $C$, say $T_1$, $T_2$ and $T_3$, and exactly three cut $C$ at two $\FFF_2$-points, say $L_1$, $L_2$, and $L_3$. The seventh line cuts $C$ at two points that form a Galois orbit of size $2$, so by assumption this is not the line $L$.
  Note that $\Aut_{\FFF_2}(C)\subset \Aut_{\FFF_2}(\PPP^2)$ acts transitively on the three points $p_1,p_2$ and $p_3$, and it fixes the two sets $\{T_1,T_2,T_3\}$ and $\{L_1,L_2,L_3\}$.
  Therefore, there are exactly two different Galois orbits of size $2$, one coming from the size $2$ orbit on $T_1$, and one coming from $L_1$.
\end{proof}

\begin{lemma}\label{lem:CountingLoops}
  Let $\k=\F_2$. Then there are exactly two isomorphism classes of del Pezzo surfaces of degree $2$ and Picard rank $3$ that can be realised as the blow-up of $\p^2$ at one point of degree $2$ and one point of degree $5$.
\end{lemma}

\begin{proof}
  Lemma~\ref{lemma:TwoGeiserLoops} gives that there are at most $2$ such del Pezzo surfaces $Z_i$ given by the blow-up of a point $p_i$ of degree $2$ and a point $q_i$ of degree $5$ for $i=1,2$. In \cite[Figure~B.12]{LamySchneider} one sees that there are exactly two contractions $Z_i\to\p^2$, and the Geiser involution on $Z_i$ exchanges the two. As there is no element of $\Aut_\k(\p^2)$ that sends $p_1,q_1$ onto $p_2,q_2$, the two surfaces $Z_1, Z_2$ are not isomorphic.
\end{proof}

\begin{remark}\label{remark:PGL3F2}
	Note that $\Aut_{\FFF_2}(\PPP^2)\simeq\PGL_3(\FFF_2)$ has $168$ elements, and it is generated by two matrices.
	(Every finite simple group is generated by two elements, see \cite{steinberg62, miller01,AG84}.)
	Explicitly, the generators can be chosen to be
	$A=\left(\begin{smallmatrix}1 & 1 & 0\\ 0 & 1 & 0\\ 0 & 0 & 1\end{smallmatrix}\right)$ and
	$B=\left(\begin{smallmatrix}0 & 0 & 1\\ 1 & 0 & 0\\ 0 & 1 & 0\end{smallmatrix}\right)$.
	This can be seen either by translating Steinberg's generators into matrix form, or by taking from \cite{BrownEzra09} the three generators $B$,
	$B_2=\left(\begin{smallmatrix}1 & 0 & 0\\ 0 & 0 & 1\\ 0 & 1 & 0\end{smallmatrix}\right)$, and
	$B_3=\left(\begin{smallmatrix}1 & 0 & 0\\ 0 & 0 & 1\\ 0 & 1 & 1\end{smallmatrix}\right)$ of $\PGL_3(\FFF_2)$ and observing that
	\begin{align*}
		B_2=& (AB)^3 B (AB)^3,\\
		B_3=& (BA)^3 B (BA)^2.
	\end{align*}
\end{remark}

In an earlier version of this text, Theorem~\ref{thm:F2GeneratedByInvolutions} has been proved independently. To reduce overlap, we give here a proof depending on \cite{LamySchneider}.

\begin{proof}[Proof of Theorem~\ref{thm:F2GeneratedByInvolutions}]
  We prove that for $\k=\F_2$, the plane Cremona group $\Bir_\k(\p^2)$ is generated by involutions.
  Let $H\subset\Bir_\k(\p^2)$ be the subgroup generated by all involutions in $\Bir_\k(\p^2)$.

  By Proposition~\ref{proposition:GeneratorsAnyPerfectField}, $\Bir_\k(\p^2)$ is generated by $\JJ_1$, all $J_\PPPP$ where $\PPPP\subset \p^2$ is a set of four $\bar\k$-points, no three collinear, that forms either one Galois orbit of size $4$, or two Galois orbits of size $2$, and the group $G_{\leq 8}$ (containing $\Aut_\k(\p^2)$).
  By Lemma~\ref{lemma:SameFourPoints} there exists a unique point of degree $4$ and a unique set of two points of degree $2$ (no three collinear), up to pre- and postcomposition with $\Aut_\k(\p^2)$. Notation~\ref{not:EquationsX} give an explicit choice $\pi_i\colon\p^2\dashrightarrow\p^1$ with base points $\PPPP_i$ of such conic fibrations, for $i=2,4$.

  By Proposition~\ref{prop:ParametrizationFiberingType} for the case $\k=\F_2$, any $\varphi\in\Bir_\k(\p^2,\pi_i)=\JJ_i$ can be written as a composition of involutions of the form $\chi_{i,[u:v]}$ as in Example~\ref{ex:ExplicitMaps}, elements in $\Aut_\k(\p^2)$, and elements in $\JJ_1$.

  It remains to find a number $d$ such that there exist $d$ maps in $G_{\leq8}$ such that they together with $\JJ_1, \Aut_\k(\p^2)$ and $\JJ_2$, $\JJ_4$ generate $\Bir_\k(\p^2)$.
  By Remark~\ref{rem:KnownFromLamySchneider} it is enough to compute the number of symmetric Sarkisov links, and the number of del Pezzo surfaces of degree $2$ that are obtained as the blow-up of $\p^2$ at a point of degree $2$ and a point of degree $5$. There exist exactly $2$ of the latter (Lemma~\ref{lem:CountingLoops}).
  We count the number of symmetric Sarkisov links from Theorem~\ref{thm:CountingSarkisovLinks}:
  \[38+18+11+12+10+5+4+4+2+2+0+1+1=108.\]
  This gives $110$ maps. However, the Sarkisov link $\chi\colon\p^2\link33\p^2$ lies in $\langle \Aut_\k(p^2),J_1\rangle$ \cite[Figure~B.3, relation $\PPPP(\p^2;1,3)$]{LamySchneider}, and so it is redundant.
  By Remark~\ref{remark:PGL3F2}, $\Aut_{\F_2}$ is generated by the two automorphisms $[x:y:z]\mapsto [x+y:y:z]$ and $[x:y:z]\mapsto [z:x:y]$.
  This gives that $\Bir_{\F_2}(\p^2)$ is generated by $\JJ_1, \JJ_2,\JJ_4$ and $2+107+2=111$ additional maps.

  Finally, $\JJ_1$ is generated by involutions (by Lemma~\ref{lemma:GeneratorsOfJ1}, $\JJ_1\simeq \PGL_2(\k(t))\rtimes\PGL_2(\k)$, and $\PGL_2(K)$ is generated by involutions for any field $K$, by Lemma~\ref{lemma:GeneratorsOfPGL2}),
  $\Aut_{\F_2}$ is generated by involutions (Lemma~\ref{lemma:PGL3GeneratedByInvolutions}),
  the $107$ symmetric Sarkisov links are equivalent to involutions (Theorem~\ref{thm:CountingSarkisovLinks}), so after pre- or postcomposing with automorphisms of $X\in\Dl$, which can be decomposed into elements in $\JJ_1$ and $\Aut_\k(\p^2)$ by Remark~\ref{rem:KnownFromLamySchneider}\ref{it:KnownFromLamySchneider--comp}, the symmetric Sarkisov links can be chosen to be involutions,
  and the two Geiser involutions on the del Pezzo surface of degree $2$ obtained as the blow-up of $\p^2$ at a point of degree $2$ and one of degree $5$ are involutions.
  Therefore, $\Bir_{\F_2}(\p^2)$ is generated by involutions.
\end{proof}

\end{document}